\documentclass[10pt]{article}
\usepackage{multicol}
\usepackage{indentfirst}
\usepackage{latexsym}
\usepackage{bm}
\usepackage[top=1in, bottom=1in, left=1.25in, right=1.25in]{geometry}

\usepackage{amssymb,amsmath,amsthm,amsfonts,amsbsy,mathrsfs}
\usepackage{hyperref}
\usepackage{array}
\usepackage{color}
\usepackage{booktabs}
\usepackage{dsfont}
\usepackage{enumitem}

\theoremstyle{plain}
\newtheorem{theorem}{Theorem}[section]
\newtheorem{lemma}{Lemma}[section]
\newtheorem{corollary}{Corollary}[section]
\newtheorem{proposition}{Proposition}[section]

\theoremstyle{definition}

\theoremstyle{remark}
\newtheorem{remark}{Remark}[section]

\newtheorem{step}{Step}
\newtheorem*{solution*}{Solution}

\makeatletter 
\@addtoreset{case}{lemma}
\@addtoreset{case}{theorem}
\@addtoreset{case}{proposition}
\@addtoreset{case}{corollary}
\makeatother 

\begin{document}

\title{Regularized Stokes Immersed Boundary Problems in\\ Two Dimensions: Well-posedness, Singular Limit,\\ and Error Estimates}
\author{Jiajun Tong\\[10pt]University of California, Los Angeles}
\date{}
\maketitle

\begin{abstract}
Inspired by the numerical immersed boundary method,
we introduce regularized Stokes immersed boundary problems in two dimensions
to describe regularized motion of a 1-D closed elastic string in a 2-D Stokes flow,
in which a regularized $\delta$-function is used to mollify the flow field and singular forcing.
We establish global well-posedness of the regularized problems, and prove that as the regularization parameter diminishes,
string dynamics in the regularized problems converge to that in the Stokes immersed boundary problem with no regularization.
Viewing the un-regularized problem as a benchmark, we derive error estimates under various norms for the string dynamics.
Our rigorous analysis shows that the regularized problems achieve improved accuracy if the regularized $\delta$-function is suitably chosen.
This may imply potential improvement in the numerical method, which is worth further investigation.
\end{abstract}
{\small

}

\section{Introduction}
The immersed boundary problem models elastic structures moving and interacting with a surrounding fluid:
the structures apply elastic force to the fluid and alter the flow field, while in turn the flow moves and deforms the structures \cite{peskin2002immersed,lin2017solvability}. 
Mathematically, it features hydrodynamics equations with time-varying forcing supported on possibly lower-dimensional moving objects, whose motion is governed by the flow.
Numerical methods for solving the immersed boundary problem, known as the immersed boundary method \cite{peskin2002immersed,peskin1972flowPhD,peskin1972flow}, has proven to be a powerful computational tool to study such coupled motion in physics, biology, and medical sciences \cite{dillon1995microscale,bottino1998computational,mcqueen2000three,lim2004simulations,zhu2002simulation,miller2005computational}.

Inspired by the numerical method, in this paper, we propose a regularized version of the two-dimensional Stokes immersed boundary problem in the continuous setting \cite{lin2017solvability,tong2018thesis}.
It is a PDE system describing motion of a 1-D closed elastic string immersed in Stokes flow in $\mathbb{R}^2$, yet with forcing and string motion being mollified on a small spatial scale, which formally approximates the Stokes immersed boundary problem without regularization.
Its precise formulation will be provided in Section \ref{section: regularized Stokes IB problem}.
We shall study global well-posedness of the regularized problem, and 
rigorously justify its convergence to the un-regularized problem in the string dynamics as the regularization parameter diminishes.
Viewing their difference as error, we show that the error bounds depend not only on the regularization parameter and regularity of the solution, but also on the mollifier in a crucial way.
This work can be an important step towards the error estimates for the numerical immersed boundary method in a fully discrete case.
Beyond that, it also suggests possible improvements of the numerical method, which is worth further investigation.

\subsection{The 2-D Stokes immersed boundary problem}
Let us first introduce the two-dimensional Stokes immersed boundary problem without regularization, which we shall call the un-regularized or the original problem in the rest of the paper.
It describes a 1-D closed elastic string moving in 2-D Stokes flow \cite{lin2017solvability}.
We parameterize the moving string by $X(s,t)$, where $s\in \mathbb{T} \triangleq \mathbb{R}\backslash 2\pi \mathbb{Z} = [-\pi,\pi)$ is the Lagrangian coordinate and where $t$ is the time variable.
Note that $s$ is not the arc-length parameter.
Then the un-regularized problem is formally given by
\begin{align}
&\;-\Delta u+ \nabla p = f(x,t),\label{eqn: Stokes equation}\\
&\;\mathrm{div}\, u = 0,\quad |u|,|p|\rightarrow 0\mbox{ as }|x|\rightarrow \infty,\label{eqn: incompressible condition}\\
&\;f(x,t) = \int_{\mathbb{T}}F_X(s,t)\delta(x-X(s,t))\,ds,\label{eqn: elastic forcing}\\
&\;\frac{\partial X}{\partial t}(s,t) = u(X(s,t),t)\triangleq U_X(s,t),\quad
X(s,0) = X_0(s).\label{eqn: motion of the membrane and initial configuration}
\end{align}
\eqref{eqn: Stokes equation} and \eqref{eqn: incompressible condition} describe the Stokes flow in $\mathbb{R}^2$ with decay condition at infinity: $x$ is the spatial (Eulerian) coordinate; $u$ represents the divergence-free velocity field; and $p$ is the pressure.
Here we implicitly assumed that the fluid viscosity is normalized; indeed, we can always achieve this by properly redefining $u$, $p$ and $t$.
In physics, the stationary Stokes equation is suitable for describing the fluid motion that is far from being turbulent, i.e., the Reynolds number is close to zero, which is the case when spatial scale of the fluid motion is small, or when the fluid moves slowly, or when the fluid is highly viscous.
$f$ denotes the elastic force exerted on the fluid, defined by \eqref{eqn: elastic forcing}.
$F_X$ is the elastic force density in the Lagrangian coordinate associated with the string configuration $X$.
In general, it is given by \cite{peskin2002immersed}
\begin{equation}
F_X(s,t) = \partial_s\left(\mathcal{T}(|X'(s,t)|, s,t)\frac{X'(s,t)}{|X'(s,t)|}\right),
\label{eqn: Lagrangian representation of the elastic force}
\end{equation}
where $X'(s,t)$ denotes $\partial_s X(s,t)$.
$\mathcal{T}$ is the tension in the string. 
In the simple case of Hookean elasticity, for instance, $\mathcal{T}(|X'(s,t)|,s,t) = k_0 |X'(s,t)|$, with $k_0>0$ being the Hooke's constant, and thus $F_X(s,t) = k_0X_{ss}(s,t)$.
We may assume $k_0 = 1$ by properly redefining $u$, $p$, and $t$.
In \eqref{eqn: elastic forcing}, we formally use the Dirac $\delta$-measure to bridge the Lagrangian and Eulerian coordinates.
This implies that $f$ is a singular force only supported on the moving string.
Finally, \eqref{eqn: motion of the membrane and initial configuration} specifies the initial string configuration, and requires that the string to move with flow, where $U_X$ denotes the string velocity.
The equations \eqref{eqn: Stokes equation}-\eqref{eqn: motion of the membrane and initial configuration} readily admit an autonomous dynamics and there is no need to specify initial flow field, since $u$ is instantaneously determined by $f$. 

The 2-D Stokes immersed boundary problem has received increasing attention recently from the analysis community.
In recent works by Lin and the author \cite{lin2017solvability,tong2018thesis}, we study its well-posedness with $F_X = X_{ss}$.
By virtue of the stationary Stokes equation, $u$, $p$, and $U_X$ are completely determined by the present string configuration $X$.
Thus, the system \eqref{eqn: Stokes equation}-\eqref{eqn: motion of the membrane and initial configuration} can be reformulated into a \emph{contour dynamic equation}, 
\begin{equation}\label{eqn: contour dynamic equation for the singular problem}
\partial_t X(s,t) = U_X(s,t) = \frac{1}{4\pi}\mathrm{p.v.} \int_{\mathbb{T}} -\partial_{s'}[G(X(s,t)-X(s',t))]X'(s',t)\,ds',\quad X(s,0) = X_0(s),
\end{equation}
where
\begin{equation}
G(x) = \frac{1}{4\pi}\left(-\ln|x|Id + \frac{x\otimes x}{|x|^2}\right)
\label{eqn: stokeslet}
\end{equation}
is the fundamental solution of the velocity field for 2-D stationary Stokes equation \cite{pozrikidis1992boundary}.
Here $Id$ denotes the $2\times 2$-identity matrix.
Once \eqref{eqn: contour dynamic equation for the singular problem} is solved, $u$ and $p$ can be recovered by \eqref{eqn: Stokes equation}-\eqref{eqn: elastic forcing}.

We prove local well-posedness of \eqref{eqn: contour dynamic equation for the singular problem} by utilizing its intrinsic dissipation.

\begin{proposition}[\cite{lin2017solvability,tong2018thesis}]\label{prop: local well-posedness of the singular problem}
Suppose $X_0(s) \in H^{5/2}(\mathbb{T})$, and assume there exists $\lambda>0$, such that
\begin{equation}
|X_0(s_1)-X_0(s_2)|\geq \lambda|s_1-s_2|,\quad \forall\, s_1, s_2\in \mathbb{T}.
\label{eqn: well-stretched condition}
\end{equation}
Then there exists $T_0 = T_0(\lambda, \|X_0\|_{\dot{H}^{5/2}})\in(0,+\infty]$ and a unique solution $X(s,t)\in C_{[0,T_0]} H^{5/2}\cap L^2_{T_0} H^{3}(\mathbb{T})$ of \eqref{eqn: contour dynamic equation for the singular problem}, satisfying that
\begin{equation}
\|X\|_{C_{[0,T_0]} \dot{H}^{5/2}\cap L^2_{T_0} \dot{H}^{3}(\mathbb{T})}\leq 4\|X_0\|_{\dot{H}^{5/2}(\mathbb{T})},\quad \|X_t\|_{L^2_{T_0} \dot{H}^{2}(\mathbb{T})}\leq \|X_0\|_{\dot{H}^{5/2}(\mathbb{T})},
\label{eqn: a priori estimate for the local solution in the main theorem}
\end{equation}
and that for $\forall\,s_1,s_2\in\mathbb{T}$ and $t\in[0,T_0]$,
\begin{equation}
\left|X(s_1,t) - X(s_2,t)\right| \geq \frac{\lambda}{2}|s_1 - s_2|.
\label{eqn: uniform bi lipschitz constant of the local solution in the main theorem}
\end{equation}
Moreover, the solution depends continuously on the initial data.
\end{proposition}
\noindent Here \eqref{eqn: well-stretched condition} is called \emph{the well-stretched condition}, and
\begin{equation}
C_{[0,T_0]}H^{5/2}(\mathbb{T}) = C([0,T_0]; H^{5/2}(\mathbb{T})),\quad L^2_{T_0}H^{3}(\mathbb{T}) = L^2([0,T_0]; H^{3}(\mathbb{T})).
\end{equation}
We also prove global well-posedness of \eqref{eqn: contour dynamic equation for the singular problem} when $X_0$ is sufficiently close to an equilibrium, which is an evenly parameterized circular configuration.
Moreover, such solution converges exponentially to an equilibrium.
Regularity of $u$ recovered from $X(s,t)$ is carefully studied in \cite{tong2018thesis}. 
In a parallel work, Mori et al.\;\cite{mori2017well} establish similar local and global well-posedness results for \eqref{eqn: contour dynamic equation for the singular problem} in $C^{1,\alpha}$-spaces.
They also show improved regularity of $X(s,t)$ for positive time and a blowup criterion.
When there is no blowup, they characterize global behavior of the solution.
Rodenberg \cite{rodenberg20182d} proves local well-posedness of \eqref{eqn: Stokes equation}-\eqref{eqn: motion of the membrane and initial configuration} with elastic force $F_X$ of general form.

In the general immersed boundary problem, the stationary Stokes equation needs to be replaced by the Navier-Stokes equations.
To the best of our knowledge, its well-posedness is still open.

\subsection{The $\varepsilon$-regularized 2-D Stokes immersed boundary problem}
\label{section: regularized Stokes IB problem}
Now we introduce \emph{the $\varepsilon$-regularized 2-D Stokes immersed boundary problem} as follows:
\begin{align}
&\;-\Delta u^\varepsilon+ \nabla p^\varepsilon = f^\varepsilon(x,t),\label{eqn: Stokes equation in the regularized IB problem}\\
&\;\mathrm{div}\, u^\varepsilon = 0,\quad |u^\varepsilon|,|p^\varepsilon|\rightarrow 0\mbox{ as }|x|\rightarrow \infty,\label{eqn: incompressible condition in the regularized IB}\\
&\;f^\varepsilon(x,t) = \int_{\mathbb{T}}F_{X^\varepsilon}(s,t)\delta_\varepsilon(x-X^\varepsilon(s,t))\,ds,\label{eqn: regularization of forcing}\\
&\;\frac{\partial X^\varepsilon}{\partial t}(s,t) = \int_{\mathbb{R}^2} u^\varepsilon(x,t)\delta_\varepsilon(X^\varepsilon(s,t)-x)\,dx\triangleq U^\varepsilon_{X^\varepsilon}(s,t),\quad
X^\varepsilon(s,0) = X_0(s).\label{eqn: regularized motion of the membrane and initial configuration}
\end{align}
Apart from the minor difference in notation, compared to \eqref{eqn: Stokes equation}-\eqref{eqn: motion of the membrane and initial configuration},
the singular $\delta$-measure is now replaced by $\delta_\varepsilon$, a regularized approximation of the Dirac $\delta$-measure defined by
\begin{equation}
\delta_\varepsilon(x) = \frac{1}{\varepsilon^2}\phi\left(\frac{x}{\varepsilon}\right).
\label{eqn: def of delta_eps}
\end{equation}
Here $\phi$ is sufficiently regular and compactly supported;
more assumptions on $\phi$ will be specified later. 
With $\delta_\varepsilon$, the elastic force along the string is mollified to become regular and it is spread to a small neighborhood of the string.
Besides that, the string velocity in \eqref{eqn: regularized motion of the membrane and initial configuration} is now determined by averaging the ambient flow field in a small neighborhood of the string using the same function $\delta_\varepsilon$, as opposed to setting the string velocity to be exactly equal to the fluid velocity at that point in \eqref{eqn: motion of the membrane and initial configuration}.
As in the un-regularized case, $u^\varepsilon$, $p^\varepsilon$, and $U_{X^\varepsilon}^\varepsilon$ are fully determined by $X^\varepsilon$ at the present time.
Formally, the $\varepsilon$-regularized problem approximates the un-regularized one.

The $\varepsilon$-regularized problem is motivated by the numerical immersed boundary method.
The numerical method involves spatial discretization of the flow field using Eulerian grid, and parameterization of the immersed elastic structures using Lagrangian coordinates.
Flow field is solved on the Eulerian grid, while elastic force and velocity of the elastic object are evaluated only on the Lagrangian marker points.
However, these two sets of coordinates do not agree in general.
In order to let them communicate, before the flow field is computed in each time step, the elastic force needs to be spread from the Lagrangian marker points to adjacent Eulerian grid points in a suitable way, while after the flow field is solved on the Eulerian grid, motion of the immersed structure, or the velocity at the Lagrangian points, needs to be determined via interpolation.
In practice, such spreading and interpolation are realized by a smoothed approximation of the Dirac $\delta$-function.
See e.g.~\cite{peskin2002immersed,mittal2005immersed} for more details.
Therefore, it is natural to 
introduce a similar regularization to the PDE problem, which can be heuristically viewed as the continuous system discretized and computed by the numerical method.
It is would be interesting to rigorously study the regularized problem and figure out whether and how it approximates the original problem,
as this can shed light on the analysis and justification of the numerical immersed boundary method in the fully discrete setting.

Regularization of singular physical quantities or singular integral kernels is also seen in other numerical methods,
such as the vortex methods \cite{leonard1980vortex} and the method of regularized Stokeslet \cite{cortez2001method,cortez2005method}.


\subsection{Main results}\label{section: main results}
Unless otherwise stated (for example in Section \ref{section: static error estimates}), we shall focus on the case where the string has Hookean elasticity with normalized Hooke's constant, i.e., 
$F_Y = Y_{ss}$.
We first prove global well-posedness of the $\varepsilon$-regularized problem \eqref{eqn: Stokes equation in the regularized IB problem}-\eqref{eqn: regularized motion of the membrane and initial configuration}.

\begin{theorem}[Global well-posedness]
\label{thm: global well-posedness of the regularized problem}

Assume $\phi$, the profile of the regularized $\delta$-function in \eqref{eqn: def of delta_eps}, is compactly supported in a ball $B_{c_0}(0)\subset\mathbb{R}^2$ centered at the origin with radius $c_0$, satisfying that $\phi(x) = \phi(-x)$ for all $x\in \mathbb{R}^2$.
Fix $\varepsilon>0$.
\begin{enumerate}
  \item If $X_0\in H^1(\mathbb{T})$ and $\phi\in W^{2,1}(\mathbb{R}^2)$,
  \eqref{eqn: Stokes equation in the regularized IB problem}-\eqref{eqn: regularized motion of the membrane and initial configuration} admits a unique global solution, such that $X^\varepsilon\in C_{[0,+\infty)}^1H^1(\mathbb{T})$, and $\nabla u^\varepsilon\in Lip([0,+\infty); L^2(\mathbb{R}^2))$.

\item For any $\beta>1$, if $X_0\in H^\beta(\mathbb{T})$ and $\phi\in C^{\lceil\beta\rceil,1}(\mathbb{R}^2)$, \eqref{eqn: Stokes equation in the regularized IB problem}-\eqref{eqn: regularized motion of the membrane and initial configuration} admits a unique global solution, such that $X^\varepsilon\in C_{[0,+\infty),loc}^1H^\beta(\mathbb{T})$, and $\nabla u^\varepsilon\in Lip([0,+\infty); L^2(\mathbb{R}^2))$.
\end{enumerate}
\end{theorem}
\begin{remark}
Here the smoothness assumptions on $\phi$ may not be the sharpest.
Yet, it is noteworthy that the 4-point regularized $\delta$-function commonly used in the numerical immersed boundary method \cite{peskin2002immersed} admits $W^{2,1}$-regularity. 

\end{remark}

In \cite{tong2018thesis}, we prove Theorem \ref{thm: global well-posedness of the regularized problem} in special cases $\beta = 1$ and $\beta = 5/2$. 
In fact, we also show global well-posedness for the regularized problem with full Navier-Stokes equation.
The proof of Theorem \ref{thm: global well-posedness of the regularized problem} for other $\beta$ is a straightforward generalization.
We shall present the whole proof in Section \ref{section: well-posedness of regularized problem} for completeness.

Since $\delta_\varepsilon$ approximates the Dirac $\delta$-function in distribution, it is natural to believe that as $\varepsilon \rightarrow 0$, 
$X^\varepsilon(s,t)$ determined by \eqref{eqn: Stokes equation in the regularized IB problem}-\eqref{eqn: regularized motion of the membrane and initial configuration}
should converges in certain sense to the solution $X(s,t)$ of \eqref{eqn: contour dynamic equation for the singular problem}, provided that they start from identical (or converging) initial data. 
In fact, we can show that
\begin{theorem}[Convergence and error estimates of the $\varepsilon$-regularized problem]
\label{thm: error estimates of the regularized problem}

Assume 
$\phi\in C_0^\infty(\mathbb{R}^2)$ satisfies that
\begin{itemize}
\item $\phi$ is radially symmetric; 
\item $\phi$ is normalized, i.e., $\int_{\mathbb{R}^2} \phi(x)\,dx = 1$.
\end{itemize}
Define
\begin{equation}
m_1 
= \int_{\mathbb{R}^2} |x|\cdot\phi*\phi(x)\,dx,
\label{eqn: def of M_1}
\end{equation}
and
\begin{equation}\label{eqn: def of M_2}
m_2 
= \int_{\mathbb{R}^2}|x|^2\cdot \phi*\phi(x)\,dx.
\end{equation}

Fix $\theta\in [\frac{1}{4},1)$.
Suppose $X_0\in H^{2+\theta}(\mathbb{T})$ satisfies the well-stretched condition \eqref{eqn: well-stretched condition} with $\lambda>0$.
Suppose $X\in C_{[0,T]}H^{2+\theta}(\mathbb{T})$ is a (local) solution of 
the contour dynamic equation \eqref{eqn: contour dynamic equation for the singular problem} of the original problem for some $T>0$.
Let $X^\varepsilon\in C_{[0,T]}H^{2+\theta}(\mathbb{T})$ be the unique solution of the string motion in the $\varepsilon$-regularized problem \eqref{eqn: Stokes equation in the regularized IB problem}-\eqref{eqn: regularized motion of the membrane and initial configuration}.
Assume that for all $\varepsilon\ll \lambda $ and $t\in [0,T]$,
\begin{enumerate}[label=(\roman*)]
  \item \label{assumption: uniform boundedness} $\|X^\varepsilon(\cdot,t)\|_{\dot{H}^{2+\theta}(\mathbb{T})}\leq M$, and $\|X(\cdot,t)\|_{\dot{H}^{2+\theta}(\mathbb{T})}\leq M$;
  \item \label{assumption: uniform stretching} $X^\varepsilon(\cdot,t)$ and $X(\cdot,t)$  satisfy the well-stretched condition \eqref{eqn: well-stretched condition} with constant $\lambda/2$.
\end{enumerate}
Then as $ \varepsilon\rightarrow 0$,
\begin{equation}\label{eqn: weak star convergence with assumption on the uniform bound}
X^\varepsilon\rightharpoonup X\quad \mbox{weak-* in } C_{[0,T]}H^{2+\theta}(\mathbb{T}),
\end{equation}
and
\begin{equation}\label{eqn: convergence with assumption on the uniform bound}
X^\varepsilon\rightarrow X\quad \mbox{ in } C_{[0,T]}H^{\gamma}(\mathbb{T}).
\end{equation}
for all $\gamma<2+\theta$.
More precisely, define $\tilde{\varepsilon} = \varepsilon/\lambda$ to be the normalized regularization parameter. 
Then for $\tilde{\varepsilon}\ll 1$, with $C = C(\theta,\lambda^{-1}M,T)$,
\begin{align}
\|X^\varepsilon-X\|_{C_{[0,T]}H^{1/2}(\mathbb{T})}\leq &\;C\left(m_1 \tilde{\varepsilon} +\tilde{\varepsilon}^{1+\theta} |\ln\tilde{\varepsilon}|^\theta\right)M,
\label{eqn: error estimate H gamma greater than 1/2}\\
\|X^\varepsilon-X\|_{C_{[0,T]}\dot{H}^{1}(\mathbb{T})}\leq&\; C|\ln \tilde{\varepsilon}|^{\frac{1}{2}}\left(m_1 \tilde{\varepsilon} +\tilde{\varepsilon}^{1+\theta}|\ln\tilde{\varepsilon}|^\theta\right)M,
\label{eqn: error estimate H gamma between 1 and 2}\\
\|X^\varepsilon-X\|_{C_{[0,T]}\dot{H}^{2}(\mathbb{T})}\leq &\;C|\ln \tilde{\varepsilon}|^{\frac{1}{2}}\tilde{\varepsilon}^{\theta} M,
\label{eqn: error estimate H gamma greater than 2}
\end{align}
and $\|X^\varepsilon-X\|_{C_{[0,T]}\dot{H}^{2+\theta}(\mathbb{T})}\leq CM$.
Estimates in intermediate $H^\gamma$-spaces can be derived by interpolation.

If $m_2 = 0$, the logarithmic factors $|\ln\tilde{\varepsilon}|^\theta$ in \eqref{eqn: error estimate H gamma greater than 1/2} and \eqref{eqn: error estimate H gamma between 1 and 2} can be removed.
\end{theorem}

\begin{remark}
The smoothness assumption on $\phi$ may be weaken.
The radial symmetry of $\phi$ is not essential, but it simplifies the analysis significantly (see Section \ref{section: contour dynamic formulation}).
Note that in the numerical immersed boundary method, the regularized $\delta$-functions are in the form of product of one-dimensional profiles \cite{peskin2002immersed}, which are not radially symmetric unless they are of Gaussian-type.
\end{remark}
\begin{remark}\label{rmk: artificial bound on H2.5 norm}
The assumptions \ref{assumption: uniform boundedness} and \ref{assumption: uniform stretching} are crucial and they can not be removed.
Unfortunately, in this work, we are not able to rigorously prove them or construct an example in which either of them fails.
We claim that this stems from an essential difference between the regularized and the un-regularized problems, which we heuristically explain as follows.
In the un-regularized problem, it has been shown that the string velocity can effectively damp high frequencies in the string configuration \cite{lin2017solvability}.
In the $\varepsilon$-regularized problem, however, the string velocity $U_{X^\varepsilon}^\varepsilon$  is obtained by first mollifying the flow field and then making restriction onto the string.
In this process, high-frequency information of $X^\varepsilon$ that is encoded in the flow field gets almost eliminated, especially for those frequencies higher than $O(\tilde{\varepsilon}^{-1})$.
As a result, although $U_{X^\varepsilon}^\varepsilon$ may approximate $U_{X^\varepsilon}$ pretty well over low frequencies with wave numbers up to $O(\tilde{\varepsilon}^{-1})$,
it is not clear if the former one can damp higher frequencies in $X^\varepsilon$ as the latter one does.
It is then possible that high frequencies in $X^\varepsilon$ may grow without being well-controlled.
In this work, instead of dealing with this subtle issue, we made assumptions \ref{assumption: uniform boundedness} and \ref{assumption: uniform stretching} to simplify our analysis.
\end{remark}

\begin{remark}
Theorem \ref{thm: error estimates of the regularized problem} indicates that in general, over the time interval $[0,T]$,
$\|X^\varepsilon-X\|_{H^1(\mathbb{T})}$ is of order $\tilde{\varepsilon}$ up to logarithmic factors. 
This roughly agrees with the well-known fact that the original numerical immersed boundary method can only achieve first-order accuracy in the vicinity of 
a truly lower-dimensional deformable object \cite{beyer1992analysis},
although higher accuracy may be attained in smoother problems \cite{griffith2005order,griffith2007adaptive} or via sophisticated extension techniques \cite{stein2016immersed,stein2017immersed}.
However, our analysis also shows that when $m_1 = 0$, the error bounds get improved:
in this case, for $\gamma\in [1,2+\theta]$, $\|X^\varepsilon-X\|_{H^{\gamma}(\mathbb{T})}$ is bounded by $\tilde{\varepsilon}^{2+\theta-\gamma}$ up to logarithmic factors over the time interval $[0,T]$.
The power of $\tilde{\varepsilon}$ seems very natural given the intuition that $X^\varepsilon$ should agree with $X$ very well over frequencies up to $O(\tilde{\varepsilon}^{-1})$ while they are both bounded in $\dot{H}^{2+\theta}$-semi-norm.
Similar improvement is also seen in Theorem \ref{thm: convergence and error estimates of the low-frequency regularized IB method} below.
Both of them arise from the key estimate in Proposition \ref{prop: static regularization error estimate for Hookean case} or Proposition \ref{prop: L^2 error estimate of the string velocity} which we will present later.
Since $m_1$ is a constant only depending on $\phi$, it suggests a possible way of improving accuracy of the $\varepsilon$-regularized problem by suitably choosing $\phi$. 
It potential numerical implication is worth further investigation.
See discussions on this in Section \ref{section: discussion}.
\end{remark}

Our last result aims at establishing convergence and error estimates of a regularized problem as $\varepsilon\rightarrow 0$, without extra assumptions like \ref{assumption: uniform boundedness} and \ref{assumption: uniform stretching}. 
To state the result, we introduce a further adaptation of the $\varepsilon$-regularized problem.
With $F_X = X_{ss}$ and
$N\in \mathbb{Z}_+$ to be chosen, we consider
\begin{align}
&\;-\Delta u^{\varepsilon,N}+ \nabla p^{\varepsilon,N} = f^{\varepsilon,N}(x,t),\label{eqn: Stokes equation in the low-frequency regularized IB problem}\\
&\;\mathrm{div}\, u^{\varepsilon,N} = 0,\quad |u^{\varepsilon,N}|,|p^{\varepsilon,N}|\rightarrow 0\mbox{ as }|x|\rightarrow \infty,\label{eqn: incompressible condition in the low-frequency regularized IB}\\
&\;f^{\varepsilon,N}(x,t) = \int_{\mathbb{T}}\mathcal{P}_N X_{ss}^{\varepsilon,N}(s,t)\delta_\varepsilon(x-X^{\varepsilon,N}(s,t))\,ds,\label{eqn: regularization of forcing in the low-frequency problem}\\
&\;\frac{\partial X^{\varepsilon,N}}{\partial t} = \mathcal{P}_N \left[\int_{\mathbb{R}^2} u^{\varepsilon,N}(x,t)\delta_\varepsilon(X^{\varepsilon,N}(s,t)-x)\,dx\right]\triangleq U^{\varepsilon,N}_{X^{\varepsilon,N}}(s,t),\label{eqn: low-frequency regularized motion of the membrane}\\
&\;X^{\varepsilon,N}(s,0) = \mathcal{P}_N X_0(s).\label{eqn: low-frequency regularized initial configuration}
\end{align}
Here $\mathcal{P}_N$ is the linear projection operator to the space of functions containing Fourier modes with wave numbers no greater than $N$.
To be more precise, define Fourier transform in $L^2(\mathbb{T})$ and its inverse to be
\begin{equation}\label{eqn: Fourier transform on T}
\hat{g}_k = \int_{\mathbb{T}} g(s)e^{-iks}\,ds,\quad
g(s) = \frac{1}{2\pi}\sum_{k\in\mathbb{Z}}\hat{g}_ke^{iks}.
\end{equation}
Then for $g\in L^2(\mathbb{T})$, $\mathcal{P}_N$ is defined by
\begin{equation}
\mathcal{P}_N g(s)\triangleq\frac{1}{2\pi}\sum_{|k| \leq N} \hat{g}_k e^{iks},\quad s\in \mathbb{T}.
\label{eqn: projection operator}
\end{equation}
We call \eqref{eqn: Stokes equation in the low-frequency regularized IB problem}-\eqref{eqn: low-frequency regularized initial configuration} \emph{$(\varepsilon,N)$-regularized 2-D Stokes immersed boundary problem}.
Compared with \eqref{eqn: Stokes equation in the regularized IB problem}-\eqref{eqn: regularized motion of the membrane and initial configuration}, the new system only allows the string configuration as well as the elastic force in the Lagrangian coordinate to have frequencies no higher than $N$.
In fact, the projection in \eqref{eqn: regularization of forcing in the low-frequency problem} may be omitted. 

We remark that this adaptation mimics the numerical scenario where $N$ Lagrangian marker points are used to represent the string configuration,
although projection to low frequencies is not what gets implemented in the numerical method.
On the other hand, by properly choosing $N$, we can get rid of potential growth of Fourier coefficients in high frequencies,
which is the main reason the extra assumptions are needed in Theorem \ref{thm: error estimates of the regularized problem}.
Indeed, for the $(\varepsilon,N)$-regularized problem, we can prove the following theorem.

\begin{theorem}[Well-posedness, convergence, and error estimates of the $(\varepsilon,N)$-regularized problem]
\label{thm: convergence and error estimates of the low-frequency regularized IB method}

Suppose $\phi$ satisfies the assumptions in Theorem \ref{thm: error estimates of the regularized problem}.
Let $\theta \in [\frac{1}{4},1)$.
Assume $X_0 \in H^{2+\theta}(\mathbb{T})$ satisfies the well-stretched condition \eqref{eqn: well-stretched condition} with $\lambda>0$, and
$\|X_0\|_{\dot{H}^{2+\theta}(\mathbb{T})}= M_0$.
Define $\tilde{\varepsilon} = \varepsilon/\lambda$ to be the normalized regularization parameter.
Then
\begin{enumerate}
  \item For all $\varepsilon,N>0$,
\eqref{eqn: Stokes equation in the low-frequency regularized IB problem}-\eqref{eqn: low-frequency regularized initial configuration}
admits a unique global solution $X^{\varepsilon,N}(s,t)$, such that $X^{\varepsilon,N}\in C^1_{[0,+\infty),loc}H^{2+\theta}(\mathbb{T})$.

\item Suppose that \eqref{eqn: contour dynamic equation for the singular problem} admits a solution $X(s,t)\in C_{[0,T_*]} H^{2+\theta}(\mathbb{T})$ for some $T_* >0$, such that for all $t\in [0,T_*]$,
\begin{equation}
\|X(\cdot,t)\|_{\dot{H}^{2+\theta}(\mathbb{T})}\leq C_*M_0,
\label{eqn: bound for benchmark solution in the maximal time interval}
\end{equation}
and
\begin{equation}\label{eqn: well-stretched condition of benchmark solution in the maximal time interval}
X(\cdot,t)\mbox{  satisfies the well-stretched condition with constant }\lambda/2.
\end{equation}

Then there exists $c_*>0$ and $N_*>0$, which depend on $\lambda$, $M_0$, and $X_0$, such that
\begin{enumerate}
  \item For all $\tilde{\varepsilon}\ll 1$ and $N\in [N_*, c_* \tilde{\varepsilon}^{-1}]$,
we have for all $t\in [0,T_*]$,
\begin{equation}
\|X^{\varepsilon,N}(\cdot,t)\|_{\dot{H}^{2+\theta}(\mathbb{T})}\leq 2C_*M_0,
\end{equation}
and
\begin{equation}
X^{\varepsilon,N}(\cdot,t)\mbox{  satisfies the well-stretched condition with constant }\lambda/4.
\end{equation}
\item For any sequence $(\varepsilon, N)\rightarrow (0,+\infty)$ such that $\tilde{\varepsilon} N\leq c_*$,
\begin{equation}\label{eqn: weak star convergence with assumption on the uniform bound eps N problem}
X^{\varepsilon,N}\rightharpoonup X\quad \mbox{weak-* in } C_{[0,T_*]}H^{2+\theta}(\mathbb{T}),
\end{equation}
and for all $\gamma<2+\theta$,
\begin{equation}\label{eqn: convergence with assumption on the uniform bound eps N problem}
X^{\varepsilon,N}\rightarrow X\quad \mbox{ in } C_{[0,T_*]}H^{\gamma}(\mathbb{T}).
\end{equation}
More precisely, for all $t\in [0,T_*]$,
\begin{itemize}
\item For $\beta$ satisfying 
\begin{equation}
0<\beta<\min\left\{\theta,\frac{1}{2}\right\},
%
\label{eqn: admissible range of beta}
\end{equation}
we have
\begin{equation}
\begin{split}
&\;\|(X^{\varepsilon,N}-X)(t)\|_{H^{1/2}(\mathbb{T})}\\
\leq&\;C(e^{-tN/4}N^{-\frac{3}{2}-\theta}+N^{-{\frac{5}{2}}})M_0\\
&\;+C(N^{-\frac{3}{2}-\theta-\beta}+m_1\tilde{\varepsilon}+\tilde{\varepsilon}^{1+\theta}|\ln \tilde{\varepsilon}|^\theta)M_0,
\end{split}
\label{eqn: H 1/2 difference between the solution X eps N and X}
\end{equation}
and
\begin{equation}
\begin{split}
&\;\|(X^{\varepsilon,N}-X)(t)\|_{\dot{H}^1(\mathbb{T})}\\
\leq&\;C(e^{-tN/4}N^{-1-\theta}+N^{-2})M_0\\
&\;+C(\ln N)^{\frac{1}{2}}(N^{-\frac{3}{2}-\theta-\beta}+m_1\tilde{\varepsilon}+\tilde{\varepsilon}^{1+\theta}|\ln \tilde{\varepsilon}|^\theta)M_0,
\end{split}
\label{eqn: H 1 difference between the solution X eps N and X}
\end{equation}
where $C = C(\beta, \theta, \lambda^{-1}M, T_*)$.
If, in addition, $m_2 = 0$, the logarithmic factors $|\ln\tilde{\varepsilon}|^\theta$ in \eqref{eqn: H 1/2 difference between the solution X eps N and X} 
and \eqref{eqn: H 1 difference between the solution X eps N and X}
can be removed.
\item
For $\beta$ satisfying
\begin{equation}
0<\beta\leq \min\left\{\theta,\frac{1}{2}\right\},\mbox{ and }\beta\not=\frac{1}{2}\mbox{ when }\theta =\frac{1}{2},
%
\label{eqn: admissible range of beta larger}
\end{equation}
we have
\begin{equation}
\begin{split}
&\;\|(X^{\varepsilon,N}-X)(t)\|_{\dot{H}^2(\mathbb{T})}\\
\leq&\; C(e^{-tN/4}N^{-\theta}+N^{-1})M_0+C(\ln N)^{\frac{1}{2}}(N^{ -\frac{1}{2}-\theta-\beta}+\tilde{\varepsilon}^\theta)M_0,
\end{split}
\label{eqn: H 2 difference between the solution X eps N and X}
\end{equation}
where $C = C(\beta, \theta, \lambda^{-1}M, T_*)$, and
\begin{equation}
\|(X^{\varepsilon,N}-X)(t)\|_{\dot{H}^{2+\theta}(\mathbb{T})}
\leq C(e^{-tN/4}+N^{-1+\theta}
+(\tilde{\varepsilon} N)^{\theta})M_0,
\label{eqn: H 2 theta difference between the solution X eps N and X}
\end{equation}
where $C = C(\theta, \lambda^{-1}M, T_*)$.
\item Estimates in intermediate $H^\gamma$-spaces can be derived by interpolation.
\end{itemize}

\end{enumerate}
\end{enumerate}
\end{theorem}

\begin{remark}
Compared to Theorem \ref{thm: error estimates of the regularized problem}, extra terms involving $N$ show up in the error estimates due to the projection.
From a numerical point of view, the most natural choice of $N$ would be $N\sim O(\tilde{\varepsilon}^{-1})$, although $N$ does not exactly correspond to the number of Lagrangian markers in the discrete setting.
If we do take $N = c\tilde{\varepsilon}^{-1}$, 
the error estimates here reduce to
\begin{align}
\|X^{\varepsilon,N}-X\|_{C_{[0,T_*]}H^{1/2}(\mathbb{T})}\leq &\;C\left(m_1 \tilde{\varepsilon} +\tilde{\varepsilon}^{1+\theta} |\ln\tilde{\varepsilon}|^\theta\right)M_0,\\
\|X^{\varepsilon,N}-X\|_{C_{[0,T_*]}\dot{H}^{1}(\mathbb{T})}\leq&\; C|\ln \tilde{\varepsilon}|^{\frac{1}{2}}\left(m_1 \tilde{\varepsilon} +\tilde{\varepsilon}^{1+\theta}|\ln\tilde{\varepsilon}|^\theta\right)M_0,\label{eqn: simplified H1 bounds in the low frequency case}\\
\|X^{\varepsilon,N}-X\|_{C_{[0,T_*]}\dot{H}^{2}(\mathbb{T})}\leq &\;C|\ln \tilde{\varepsilon}|^{\frac{1}{2}}\tilde{\varepsilon}^{\theta} M_0,\\
\|X^{\varepsilon,N}-X\|_{C_{[0,T_*]}\dot{H}^{2+\theta}(\mathbb{T})}\leq&\; CM_0,
\end{align}
with $C = C(\theta, \lambda^{-1}M, T_*)$, which coincide with those in Theorem \ref{thm: error estimates of the regularized problem}.
However, it is a bit surprising that, if we are allowed to ignore those exponentially-decaying terms in \eqref{eqn: H 1/2 difference between the solution X eps N and X},
\eqref{eqn: H 1 difference between the solution X eps N and X}, \eqref{eqn: H 2 difference between the solution X eps N and X}, and \eqref{eqn: H 2 theta difference between the solution X eps N and X},
which will become negligible even when $t$ is small,
we may take $N$ much smaller than $O(\tilde{\varepsilon}^{-1})$ without worsening the error bounds. 
This suggests that in order to track the string dynamics in the regularized problem as accurately as possible, we may need much fewer Fourier modes than $O(\tilde{\varepsilon}^{-1})$ to represent the string.
Numerical implication of this result is worth further investigation.
\end{remark}

\begin{remark}
The uniquenss of $X^{\varepsilon,N}$ together with \eqref{eqn: weak star convergence with assumption on the uniform bound eps N problem} implies that the solution $X$ assumed in Theorem \ref{thm: convergence and error estimates of the low-frequency regularized IB method} should be unique.
\end{remark}

\subsection{Scheme of the proofs and organization of the paper}
\label{section: scheme of the proof}
Let us take the $\varepsilon$-regularized problem as an example to sketch the idea of proving convergence and error estimates.

Recall that in the analysis of the original problem with the Hookean elasticity \cite{lin2017solvability,mori2017well},
\eqref{eqn: Stokes equation}-\eqref{eqn: motion of the membrane and initial configuration} is first reduced (under some assumptions) to the contour dynamic equation \eqref{eqn: contour dynamic equation for the singular problem},
and it is then rewritten as $\partial_t X = \mathcal{L}X +g_X$.
Here $\mathcal{L} X= -\frac{1}{4}(-\Delta)^{1/2}X$ captures the principal singular part in $U_X$ that is derived by linearizing the integrand of \eqref{eqn: contour dynamic equation for the singular problem} around $s'=s$, while $g_X$ is a nonlinear nonlocal term collecting all the remaining terms.
It is observed that $\mathcal{L}$ is a dissipative operator and $g_X$ turns out to be sufficiently regular, 
which enables us to prove well-posedness of \eqref{eqn: contour dynamic equation for the singular problem}.

We shall take a similar path for the $\varepsilon$-regularized problem by focusing on the string motion and using the original problem as a benchmark.
By \eqref{eqn: motion of the membrane and initial configuration} and \eqref{eqn: regularized motion of the membrane and initial configuration}, we derive that
\begin{equation}
\partial_t X^\varepsilon = U_{X^\varepsilon}+ (U_{X^\varepsilon}^\varepsilon-U_{X^\varepsilon})=\mathcal{L}X^\varepsilon + g_{X^\varepsilon}+(U_{X^\varepsilon}^\varepsilon-U_{X^\varepsilon}),
\label{eqn: contour dynamic equation of X_eps using X as a benchmark}
\end{equation}
and
\begin{equation}
\partial_t (X^\varepsilon-X) = 
\mathcal{L}(X^\varepsilon-X) + (g_{X^\varepsilon}-g_X)+(U_{X^\varepsilon}^\varepsilon-U_{X^\varepsilon}).
\label{eqn: difference between contour dynamic equation between X and X_eps}
\end{equation}
In order to bound $X^\varepsilon-X$, thanks to the dissipative nature of the operator $\partial_t -\mathcal{L}$,
it suffices to study $g_{X^\varepsilon}-g_X$ and $U^\varepsilon_{X^\varepsilon}-U_{X^\varepsilon}$. 
In particular, it would be ideal to show that the mapping $X\mapsto g_X$ is (locally) Lipschitz in suitable function spaces, and $U^\varepsilon_{X^\varepsilon}-U_{X^\varepsilon}$ can be treated as a small error.
We shall implement this idea to prove Theorem \ref{thm: error estimates of the regularized problem} in Section \ref{section: proof of dynamic error estimate assuming boundedness}.
The proof of Theorem \ref{thm: convergence and error estimates of the low-frequency regularized IB method} uses a similar approach, with the estimates for $U^\varepsilon_{X^\varepsilon}-U_{X^\varepsilon}$ handled more carefully.

Now it is clear that a key ingredient to prove convergence and error estimates is to establish estimates for $U^\varepsilon_Y-U_Y$ for a given string configuration $Y$. 
We call such estimates \emph{static estimates} for the regularization error in the string velocity, as $U^\varepsilon_Y-U_Y$ 
only depends on $Y$ at a single time but not on its history or future evolution. 
We can show that
\begin{proposition}[Static estimates for $U_Y^\varepsilon-U_Y$, Hookean case]\label{prop: static regularization error estimate for Hookean case}
Let $\phi$, 
$m_1$, and $m_2$ be defined in Theorem \ref{thm: error estimates of the regularized problem}. 
Suppose $Y\in H^{2+\theta}(\mathbb{T})$ with $\theta\in[1/4,1)$, satisfying the well-stretched condition \eqref{eqn: well-stretched condition} with $\lambda>0$, and $F_Y = Y_{ss}$.
Given $\varepsilon>0$, let $U_Y$ and $U^\varepsilon_Y$ be the string velocities corresponding to $Y$ in the original and the $\varepsilon$-regularized problems,
defined in \eqref{eqn: contour dynamic equation for the singular problem} and
\eqref{eqn: regularized motion of the membrane and initial configuration}, respectively.
Provided that $\varepsilon\ll \lambda$,
\begin{equation}
\begin{split}
\|U^\varepsilon_Y-U_Y\|_{L^2(\mathbb{T})}
\leq &\;\frac{m_1 \varepsilon}{\pi\lambda}\left\| \frac{F_Y(s)\cdot Y'(s)}{|Y'(s)|}\right\|_{L^2(\mathbb{T})}+\frac{C\varepsilon^{1+\theta} \ln^\theta(\lambda /\varepsilon)}{\lambda^{1+\theta}}\|Y\|_{\dot{H}^{2+\theta}(\mathbb{T})}\\
&\; +\frac{C\varepsilon^2 \ln(\lambda/\varepsilon)}{\lambda^5}\|Y'\|_{L^\infty(\mathbb{T})}^2 \|Y''\|^2_{L^4(\mathbb{T})},
\end{split}
\label{eqn: L^2 error estimate final version Hookean case}
\end{equation}
where $C$'s are universal constants depending on $\theta$.
Moreover,
\begin{equation}
\|U^\varepsilon_Y-U_Y\|_{\dot{H}^1(\mathbb{T})}\leq \frac{C\varepsilon^\theta}{\lambda^\theta}\|Y\|_{\dot{H}^{2+\theta}(\mathbb{T})}+\frac{C \varepsilon}{\lambda^4}\|Y'\|_{L^\infty(\mathbb{T})}^2\|Y''\|_{L^4(\mathbb{T})}^2.
\label{eqn: H^1 error estimate of the string velocity Hookean case}
\end{equation}

If, in addition, $m_2 = 0$, the logarithmic factors in \eqref{eqn: L^2 error estimate final version Hookean case} can be removed.
\end{proposition}

Similar static estimates for the regularization error have been derived in fully discrete settings for the velocity field $u$ and pressure $p$ in the Stokes immersed boundary problem \cite{liu2012properties,liu2014p,mori2008convergence}, and more generally, for solutions of differential equations with singular source terms (see e.g., \cite{tornberg2004numerical} and references therein).
Since static error estimates are of independent interest, we shall study them in Section \ref{section: static error estimates} with greater generality, by considering elastic force $F$ of general form.

The rest of the paper is organized as follows.
In Section \ref{section: well-posedness of regularized problem}, we prove Theorem \ref{thm: global well-posedness of the regularized problem}. 
Section \ref{section: static error estimates} will be devoted to establishing static estimates for the regularization error in the string velocity. 
We will first 
formulate the problem with greater generality 
in Sections \ref{section: elastic force of more general form}-\ref{section: contour dynamic formulation}.
For clarity, we collect statements of the static estimates in Section \ref{section: statements of static error estimates}, of which Proposition \ref{prop: static regularization error estimate for Hookean case} is a special case.
We will prepare some preliminary estimates for them in Section \ref{section: preliminary estimates}, and show detailed proofs in Sections \ref{section: L^2 static error}-\ref{section: static H^1 error estimates}.
First-time readers may skip these sections, so as to not get distracted by the technicality there.
We then prove Theorem \ref{thm: error estimates of the regularized problem} in Section \ref{section: proof of dynamic error estimate assuming boundedness} and Theorem \ref{thm: convergence and error estimates of the low-frequency regularized IB method} in Section \ref{section: proof of singular limit of low-frequency regularized problem}.
We conclude the paper with discussions in Section \ref{section: discussion} on the improvement of the regularized $\delta$-function as well as future problems.
In Appendix \ref{section: proofs of auxiliary lemmas}, we will prove some auxiliary lemmas from Section \ref{section: static error estimates}.
In Appendix \ref{section: a priori estimate for L}, we show a priori estimates involving the operator $\mathcal{L} = -\frac{1}{4}(-\Delta)^{1/2}$.
Finally, some estimates involving the nonlinear term $g_X$ in the contour dynamic equation 
are proved in Appendix \ref{section: a priori estimate for g_Y}.

\section{Global Well-posedness of the $\varepsilon$-Regularized Problem}
\label{section: well-posedness of regularized problem}

\subsection{Well-posedness for $H^1$-initial data}
For completness, we first prove Theorem \ref{thm: global well-posedness of the regularized problem} with $\beta = 1$ by recasting the proof in \cite{tong2018thesis}.
The idea of establishing local well-posedness is to view \eqref{eqn: regularized motion of the membrane and initial configuration} as an ODE of $X$ in the Banach space $H^1(\mathbb{T})$ --- this is the case thanks to the regularization.
Then local well-posedness can be proved by applying the classic Picard Theorem in Banach spaces \cite[Theorem 3.1]{majda2002vorticity}.
Global well-posedness for $H^1$-initial data should follow from a continuation argument \cite[Theorem 3.3]{majda2002vorticity} combined with an energy estimate, which shows that $\|X\|_{\dot{H}^1(\mathbb{T})}$ is uniformly bounded for all time. 

\begin{proof}[Proof of Theorem \ref{thm: global well-posedness of the regularized problem} with $\beta = 1$]
With $M>1$ to be determined, we define
$$
O_{M} = \{Z\in H^1(\mathbb{T}):~\|Z\|_{\dot{H}^1(\mathbb{T})}< M\}.
$$
It is non-empty and open in $H^1(\mathbb{T})$.
We take $M$ suitably large such that $X_0\in O_M$.
It suffices to show that for all $Y\in O_M$, $U^\varepsilon_Y \in H^1(\mathbb{T})$, and the mapping $Y\mapsto U^\varepsilon_Y$ is (locally) Lipschitz in $O_M$.
Here $U_Y^\varepsilon$ is defined in \eqref{eqn: regularized motion of the membrane and initial configuration}.
With abuse of notations, we still use $f^\varepsilon$ and $u^\varepsilon$ to denote the quantities in \eqref{eqn: Stokes equation in the regularized IB problem}-\eqref{eqn: regularization of forcing} corresponding to $Y$.
We also take arbitrary $Y_i$ $(i = 1,2)$ in $O_{M}$, and let $f^\varepsilon_i$, $u^\varepsilon_i$ and $U^\varepsilon_{Y_i}$ be the quantities in \eqref{eqn: Stokes equation in the regularized IB problem}-\eqref{eqn: regularized motion of the membrane and initial configuration} corresponding to $Y_i$.

\begin{step}[From the string configuration to the forcing]\label{step: from the string configuration to the forcing on the fluid}

Since $Y\in H^1(\mathbb{T})$, $Y_{ss}\in H^{-1}(\mathbb{T})$ in \eqref{eqn: regularization of forcing}.
By a density argument,
\begin{equation}
\begin{split}
f^\varepsilon(x) =&\; \int_{\mathbb{T}}Y_{s}(s)\cdot \nabla \delta_\varepsilon(x-Y(s)) Y_{s}(s)\,ds\\
= &\; \mathrm{div}_x\left[\int_{\mathbb{T}} \delta_\varepsilon(x-Y(s)) Y_{s}(s)\otimes Y_s(s)\,ds\right].
\label{eqn: new formula for the regularized forcing}
\end{split}
\end{equation}
It is then easy to show 
\begin{align}
\|f^\varepsilon\|_{L^2(\mathbb{R}^2)} \leq &\;\|\nabla \delta_\varepsilon\|_{L^2(\mathbb{R}^2)} M^2\leq C(M,\varepsilon),\label{eqn: L^2 estimates on the regularized forcing generated by Y in O_M}\\
\|f^\varepsilon\|_{H^{-1}(\mathbb{R}^2)} \leq &\;\|\delta_\varepsilon\|_{L^2(\mathbb{R}^2)} M^2\leq C(M,\varepsilon),\label{eqn: H-1 estimates on the regularized forcing generated by Y in O_M}
\end{align}
and
\begin{equation}
\begin{split}
&\;\|f^\varepsilon_1-f^\varepsilon_2\|_{H^{-1}(\mathbb{R}^2)}\\
\leq &\;\int_{\mathbb{T}}\|\delta_\varepsilon\|_{L^2(\mathbb{R}^2)}|Y_{1,s}-Y_{2,s}||Y_{1,s}|\,ds+\int_{\mathbb{T}}\|\delta_\varepsilon\|_{L^2(\mathbb{R}^2)}|Y_{2,s}||Y_{1,s}-Y_{2,s}|\,ds\\
&\;+\int_{\mathbb{T}}|Y_1-Y_2|\|\nabla \delta_\varepsilon\|_{L^2(\mathbb{R}^2)}|Y_{2,s}|^2\,ds\\
\leq &\; C(M,\varepsilon)\|Y_1-Y_2\|_{H^1(\mathbb{T})}.
\end{split}
\label{eqn: local lipschitz of f on X}
\end{equation}
In the last line, we used Sobolev embedding $H^1(\mathbb{T})\hookrightarrow C(\mathbb{T})$.

By Sobolev embedding, $Y(\mathbb{T})$ is contained in the $B_R(\bar{Y})$ with radius $R = CM$, and
\begin{equation}
\bar{Y} = \frac{1}{2\pi}\int_\mathbb{T}Y(s)\,ds.
\end{equation}
Since $\delta_\varepsilon$ is supported on $B_{c_0\varepsilon}(0)$, $f^\varepsilon$ is supported in $B_{R_\varepsilon}(\bar{Y})$ where $R_\varepsilon = R+c_0\varepsilon$.
Moreover, $\int_{\mathbb{R}^2}f^\varepsilon(x)\,dx = 0$ thanks to \eqref{eqn: new formula for the regularized forcing}.

\end{step}

\begin{step}[From the forcing to the velocity field]\label{step: from the forcing on the fluid to the fluid velocity field}
By classic estimates of the stationary Stokes equation in $\mathbb{R}^2$ and the fact that $f^\varepsilon$ has integral zero on $\mathbb{R}^2$, the mapping $f^\varepsilon \mapsto \nabla u^\varepsilon$ is well-defined and Lipschitz continuous from $H^{-1}(\mathbb{R}^2)$ to $L^2(\mathbb{R}^2)$.
Namely,
\begin{align}
\|\nabla u^\varepsilon\|_{L^2(\mathbb{R}^2)}\leq &\;C\|f^\varepsilon\|_{H^{-1}(\mathbb{R}^2)},\label{eqn: bounding H1 norm of u by H-1 norm of f in the Stokes case}\\
\|\nabla u^\varepsilon_1-\nabla u^\varepsilon_2\|_{L^2(\mathbb{R}^2)}\leq &\;C\|f^\varepsilon_1-f^\varepsilon_2\|_{H^{-1}(\mathbb{R}^2)},\label{eqn: bounding H1 norm of u by H-1 norm of f in the Stokes case Lipschitz continuity}
\end{align}
for some universal constant $C$.

We also would like to derive an $L^\infty$-estimate of $u^\varepsilon$ for the next step.
Since $f^\varepsilon$ has integral zero, we may represent $u^\varepsilon$ by using fundamental solution of the stationary Stokes equation, i.e.
\begin{equation}
u^\varepsilon(x) = \int_{\mathbb{R}^2} G(x-y)f^\varepsilon(y)\,dy,
\end{equation}
where $G$ is the fundamental solution of the velocity field defined in \eqref{eqn: stokeslet}.
By Minkowski inequality, \eqref{eqn: L^2 estimates on the regularized forcing generated by Y in O_M}, and the fact that $f^\varepsilon$ is compactly supported,
\begin{equation}
\|u^\varepsilon\|_{L^\infty(B_{2R_\varepsilon}(\bar{Y}))}\leq C\|G\|_{L^2(B_{3R_\varepsilon}(0))}\|f^\varepsilon\|_{L^2(B_{R_\varepsilon}(\bar{Y}))}\leq C(R_\varepsilon)\|f^\varepsilon\|_{L^2(\mathbb{R}^2)}\leq C(M,\varepsilon).
\end{equation}
On the other hand,
\begin{equation}
\begin{split}
\|u^\varepsilon\|_{L^\infty(B^c_{2R_\varepsilon}(\bar{Y}))}\leq &\; C\|\nabla G\|_{L^\infty(B^c_{R_\varepsilon}(0))}\left\|\int_{\mathbb{T}} \delta_\varepsilon(x-Y(s)) Y_{s}(s)\otimes Y_s(s)\,ds\right\|_{L^1(B_{R_\varepsilon}(\bar{Y}))}\\
\leq &\; C(R_\varepsilon)M^2\leq C(M,\varepsilon).
\end{split}
\end{equation}
Combining the above two estimates, we obtain
\begin{equation}
\|u^\varepsilon\|_{L^\infty(\mathbb{R}^2)}\leq C(M,\varepsilon).
\label{eqn: L^infty estimate of regularized u}
\end{equation}
\end{step}

\begin{step}[From the velocity field to the string motion]
By \eqref{eqn: regularized motion of the membrane and initial configuration} and \eqref{eqn: L^infty estimate of regularized u},
\begin{equation}
\begin{split}
|U_Y^\varepsilon(s)|\leq &\; \int_{\mathbb{R}^2} |u^\varepsilon(x)||\delta_\varepsilon(Y(s)-x)|\,dx\leq C\|u^\varepsilon(x)\|_{L^\infty(\mathbb{R}^2)}\leq C(M,\varepsilon),
\end{split}
\label{eqn: L^2 estimate of membrane velocity}
\end{equation}
and
\begin{equation}
|\partial_s U_Y^\varepsilon(s)|= \left|Y_s(s)\int_{\mathbb{R}^2} u^\varepsilon(x)\nabla\delta_\varepsilon(Y(s)-x)\,dx\right|\leq C(M,\varepsilon)|Y_s(s)|.
\label{eqn: H^1 estimate of membrane velocity}
\end{equation}
Combining \eqref{eqn: H-1 estimates on the regularized forcing generated by Y in O_M}, \eqref{eqn: bounding H1 norm of u by H-1 norm of f in the Stokes case}, \eqref{eqn: L^2 estimate of membrane velocity} and \eqref{eqn: H^1 estimate of membrane velocity}, we find that
\begin{equation}
\|U_Y^\varepsilon\|_{H^1(\mathbb{T})}\leq C(M,\varepsilon).
\label{eqn: H^1 estimate of membrane velocity full norm stationary Stokes}
\end{equation}

Now we turn to show that the map $Y\mapsto U_Y^\varepsilon$ is Lipschitz continuous in $O_M$.
We shall derive an $H^1$-estimate for $U^\varepsilon_{Y_1}- U^\varepsilon_{Y_2}$.
First we assume $Y_1$ and $Y_2$ both have zero mean on $\mathbb{T}$; this implies $Y_i(\mathbb{T})\subset B_{R}(0)$ and thus $B_{R_\varepsilon}(0)$ covers the supports of $u^\varepsilon_1(\cdot)\delta_\varepsilon(Y_i(s)-\cdot)$.
Hence,
\begin{equation}
\begin{split}
&\;|U^\varepsilon_{Y_1}- U^\varepsilon_{Y_2}|\\
= &\;\left|\int_{B_{R_\varepsilon}(0)} u^\varepsilon_1(x)\delta_\varepsilon(Y_1(s)-x)\,dx - \int_{B_{R_\varepsilon}(0)} u^\varepsilon_2(x)\delta_\varepsilon(Y_2(s)-x)\,dx \right|\\
\leq &\; \int_{B_{R_\varepsilon}(0)} |u^\varepsilon_1-u^\varepsilon
_2||\delta_\varepsilon(Y_1(s)-x)|+|u^\varepsilon_2||\delta_\varepsilon(Y_1(s)-x)-\delta_\varepsilon(Y_2(s)-x)| \,dx \\
\leq &\; C(\varepsilon)\left(\|u^\varepsilon_1-u^\varepsilon_2\|_{L^2(B_{R_\varepsilon}(0))}+\|u^\varepsilon_2\|_{L^\infty(\mathbb{R}^2)}|Y_1-Y_2|\right),
\end{split}
\label{eqn: L^2 estimate of difference in membrane velocity}
\end{equation}
and
\begin{equation}
\begin{split}
&\;|\partial_s U^\varepsilon_{Y_1}- \partial_s U^\varepsilon_{Y_2}|\\
=&\;\left|\partial_s\int_{B_{R_\varepsilon}(0)} u^\varepsilon_1(x)\delta_\varepsilon(Y_1(s)-x)\,dx - \partial_s\int_{B_{R_\varepsilon}(0)} u^\varepsilon_2(x)\delta_\varepsilon(Y_2(s)-x)\,dx \right|\\
\leq &\; |Y_{1,s}(s)|\int_{B_{R_\varepsilon}(0)} |u^\varepsilon_1-u^\varepsilon_2||\nabla \delta_\varepsilon(Y_1(s)-x)|\,dx\\
&\;+|Y_{1,s}(s)-Y_{2,s}(s)|\int_{B_{R_\varepsilon}(0)} |u^\varepsilon_2||\nabla \delta_\varepsilon(Y_1(s)-x)|\,dx\\
&\;+|Y_{2,s}(s)|\left|\int_{B_{R_\varepsilon}(0)} u^\varepsilon_2(x)(\nabla \delta_\varepsilon(Y_1(s)-x) - \nabla \delta_\varepsilon(Y_2(s)-x))\,dx\right|\\
\leq &\; C(\varepsilon)\left(|Y_{1,s}|\|u^\varepsilon_1-u^\varepsilon_2\|_{L^2(B_{R_\varepsilon}(0))}+\|u^\varepsilon_2\|_{L^2(B_{R_\varepsilon}(0))}|Y_{1,s}-Y_{2,s}|\right.\\
&\;\qquad\left.+|Y_{2,s}|\|\nabla u^\varepsilon_2\|_{L^2(B_{R_\varepsilon}(0))}|Y_1-Y_2|\right).
\end{split}
\label{eqn: H^1 estimate of difference in membrane velocity}
\end{equation}
We need to bound $\|u^\varepsilon_1-u^\varepsilon_2\|_{L^2(B_{R_\varepsilon}(0))}$.
By Young's inequality,
\begin{equation}
\|u^\varepsilon_1-u^\varepsilon_2\|_{L^2(B_{R_\varepsilon}(0))}\leq \|G\|_{L^2(B_{2R_\varepsilon}(0))}\|f_1^\varepsilon -f_2^\varepsilon \|_{L^1(B_{R_\varepsilon}(0))}.
\end{equation}
We argue as in \eqref{eqn: local lipschitz of f on X} to obtain that
\begin{equation}
\|f^\varepsilon_1-f^\varepsilon_2\|_{L^1(\mathbb{R}^2)} \leq C(M,\varepsilon)\|Y_1-Y_2\|_{H^1(\mathbb{T})}.
\end{equation}
Here we need the assumption $\nabla^2 \phi \in L^1(\mathbb{R}^2)$; recall that $\phi$ is the profile of the regularized $\delta$-function in \eqref{eqn: def of delta_eps}.
Hence,
\begin{equation}
\|u^\varepsilon_1-u^\varepsilon_2\|_{L^2(B_{R_\varepsilon}(0))}\leq C(M,\varepsilon)\|Y_1-Y_2\|_{H^1(\mathbb{T})}.
\end{equation}
Combining this with \eqref{eqn: L^2 estimate of difference in membrane velocity} and \eqref{eqn: H^1 estimate of difference in membrane velocity}, we conclude that
\begin{equation}
\|U^\varepsilon_{Y_1}- U^\varepsilon_{Y_2}\|_{H^1(\mathbb{T})}\leq C(M,\varepsilon)\|Y_1-Y_2\|_{H^1(\mathbb{T})}.
\label{eqn: crude bound on difference of Y_t}
\end{equation}

For general $Y_1$ and $Y_2$, since $U^\varepsilon_{Y_i} = U^\varepsilon_{Y_i-\overline{Y_i}}$, the same estimate still holds.
Therefore, we prove that the map $Y\mapsto U_Y^\varepsilon$ is Lipschitz continuous in $O_M$.
By Picard Theorem in Banach space \cite[Theorem 3.1]{majda2002vorticity}, there exists $T>0$ and a unique local solution $X^\varepsilon\in  C^1_{[0,T]}(O_M)$ describing the string dynamics, which depends continuously on the initial data.
By the derivation above, $\nabla u^\varepsilon \in Lip([0,T];L^2(\mathbb{T}))$.

\end{step}

\begin{step}[Energy law and global well-posedness]
We shall derive an energy estimate to prove the global well-posedness.
We take $s$-derivative of \eqref{eqn: regularized motion of the membrane and initial configuration}
\begin{equation}
\frac{\partial X_s^\varepsilon}{\partial t}(s,t) = \int_{\mathbb{R}^2}X_{s}^\varepsilon(s,t)\cdot \nabla \delta_\varepsilon(X^\varepsilon(s,t)-x) u^\varepsilon(x,t)\,dx.
\end{equation}
It is valid to do so since $X^\varepsilon\in C^1_{[0,T]}H^1(\mathbb{T})$.
Taking inner product with $X_s^\varepsilon$ on $\mathbb{T}$, we find by \eqref{eqn: new formula for the regularized forcing} and energy estimate of the stationary Stokes equation that,
\begin{equation}
\begin{split}
\frac{1}{2}\frac{d}{dt}\|X_s^\varepsilon\|_{L^2(\mathbb{T})}^2 = &\;-\int_\mathbb{T}ds\int_{\mathbb{R}^2}X_{s}^\varepsilon(s,t)\cdot \nabla\delta_\varepsilon(x-X^\varepsilon(s,t)) u^\varepsilon(x,t)\cdot X_{s}^\varepsilon(s,t)\,dx\\
=&\; -\int_{\mathbb{R}^2}f^\varepsilon(x,t)\cdot u^\varepsilon(x,t)\,dx\\
=&\; -\|\nabla u^\varepsilon\|_{L^2(\mathbb{R}^2)}^2.
\end{split}
\label{eqn: energy estimate of the Stokes case}
\end{equation}
In the first line, we used the assumption $\phi(x) = \phi(-x)$.
This implies that for $\forall\,t\in[0,T]$, $\|X^\varepsilon\|_{\dot{H}^1(\mathbb{T})}(t) \leq \|X_0\|_{\dot{H}^1(\mathbb{T})} < M$.
Then global well-posedness follows from a continuation argument \cite[Theorem 3.3]{majda2002vorticity}.
\end{step}

This proves well-posedness for $H^1$-initial data in Theorem \ref{thm: global well-posedness of the regularized problem}.
\end{proof}

\subsection{Well-posedness for smoother initial data}

\begin{proof}[Proof of Theorem \ref{thm: global well-posedness of the regularized problem} with $\beta>1$]
Again, we shall use the Picard Theorem in Banach spaces.
Let
\begin{equation}
O_{M,M'} = \{Z\in H^{\beta}(\mathbb{T}):~\|Z\|_{\dot{H}^1(\mathbb{T})}<M,~\|Z\|_{H^{\beta}(\mathbb{T})}<M'\},
\end{equation}
with $M'\geq M\geq 1$ suitably large such that $X_0\in O_{M,M'}$.
All the estimates derived in the previous part of the proof still hold for $\forall\, Y,Y_i\in O_{M,M'}$, with estimates only depending on $M$ and $\varepsilon$, but not on $M'$ or $\beta$.

By \eqref{eqn: regularized motion of the membrane and initial configuration}, thanks to $\phi$ being compactly supported and sufficiently smooth, by taking $[\beta]$-th derivative of $\delta_\varepsilon(Y(s)-x)$ with respective to $s$ and collecting all possible terms,
\begin{equation}
\begin{split}
&\;\|U_Y^\varepsilon\|_{H^{\beta}(\mathbb{T})} \\
\leq &\;
\int_{B_{R_\varepsilon}}|u^\varepsilon(x)|\left(\|\delta_\varepsilon(Y(s)-x)\|_{L^2(\mathbb{T})}+\|\delta_\varepsilon(Y(s)-x)\|_{\dot{H}^\beta(\mathbb{T})}\right)\,dx\\
\leq &\; C(M,\varepsilon)+ \int_{R_\varepsilon}dx\,|u^\varepsilon(x)|\cdot C(\beta)\sum_{\substack{1\leq k\leq [\beta]\\ 1\leq n_1\leq \cdots\leq n_k\\ n_1+\cdots +n_k = [\beta]}}
\|Y^{(n_1)}\otimes \cdots \otimes Y^{(n_k)}:\nabla ^{k}\delta_\varepsilon (Y(s)-x)\|_{\dot{H}^{\beta-[\beta]}(\mathbb{T})}.
\end{split}
\label{eqn: decomposing H beta bound of U_Y}
\end{equation}
Here we abused the notation $\dot{H}^{0}(\mathbb{T})$ when $\beta$ is an integer; we simply understand it as $L^2(\mathbb{T})$.
When $\beta -[\beta]\in(0,1)$, since $n_i\geq 1$, 
\begin{equation}
\begin{split}
&\;\|Y^{(n_1)}\otimes \cdots \otimes Y^{(n_k)}:\nabla ^{k}\delta_\varepsilon (Y(s)-x)\|_{\dot{H}^{\beta-[\beta]}(\mathbb{T})}\\
\leq &\; C(\beta)\|Y^{(n_1)}\|_{H^1(\mathbb{T})}\cdots \|Y^{(n_{k-1})}\|_{H^1(\mathbb{T})} \|Y^{(n_k)}\|_{H^{\beta-[\beta]}(\mathbb{T})}\|\nabla^{k}\delta_\varepsilon (Y(s)-x)\|_{H^1(\mathbb{T})}\\
\leq &\; C(\beta)\|Y_s\|_{\dot{H}^{n_1}(\mathbb{T})}\cdots \|Y_s\|_{\dot{H}^{n_{k-1}}(\mathbb{T})} \|Y_s\|_{\dot{H}^{\beta-[\beta]+n_k-1}(\mathbb{T})}\\
&\;\quad \cdot (\|\nabla^{k}\delta_\varepsilon (Y(s)-x)\|_{L^2(\mathbb{T})}+\|Y_s\|_{L^2(\mathbb{T})}\|\nabla^{k+1}\delta_\varepsilon (Y(s)-x)\|_{L^\infty(\mathbb{T})})\\
\leq &\; C(\varepsilon,\beta)\|Y_s\|_{L^2(\mathbb{T})}^{k-1}\|Y_s\|_{\dot{H}^{\beta-1}(\mathbb{T})}(1+\|Y_s\|_{L^2(\mathbb{T})})\\
\leq &\; C(M,\varepsilon,\beta)M'.
\end{split}
\label{eqn: bounding H beta seminorm}
\end{equation}
In the second last line, we used interpolation inequality among $H^s$-seminorms, as well as the assumption that $\phi\in C^{[\beta]+1}(\mathbb{R}^2)$.
When $\beta -[\beta] = 0$, i.e., when $\beta$ is an integer, we may replace $\|\nabla^{k}\delta_\varepsilon (Y(s)-x)\|_{H^1(\mathbb{T})}$ in \eqref{eqn: bounding H beta seminorm} by $\|\nabla^{k}\delta_\varepsilon (Y(s)-x)\|_{L^\infty(\mathbb{T})}$.
In this way, we can derive an estimate of the same form, yet only assuming $\phi\in C^\beta(\mathbb{R}^2)$.
Combining \eqref{eqn: bounding H beta seminorm} with \eqref{eqn: decomposing H beta bound of U_Y}, we conclude that
\begin{equation}
\|U_Y^\varepsilon\|_{H^{\beta}(\mathbb{T})}\leq C(M,\varepsilon,\beta)M'.
\label{eqn: H beta bound of U_Y}
\end{equation}
Using a similar argument as in \eqref{eqn: H^1 estimate of difference in membrane velocity}, \eqref{eqn: decomposing H beta bound of U_Y}, and \eqref{eqn: bounding H beta seminorm}, but with more complicated derivation, we may also prove that the map $Y\mapsto U_Y^\varepsilon$ is Lipschitz continuous in $O_{M,M'}$,
\begin{equation}
\|U_{Y_1}^\varepsilon-U_{Y_2}^\varepsilon\|_{H^\beta(\mathbb{T})} \leq C(M',\varepsilon,\beta)\|Y_1-Y_2\|_{H^{\beta}(\mathbb{T})},\quad\forall\, Y_1,Y_2\in O_{M,M'}.
\end{equation}
Here we will need the assumption that $\phi\in C^{\lceil\beta\rceil,1}(\mathbb{R}^2)$.
We omit the details.
Then the local well-posedness immediately follows from the Picard Theorem in $O_{M,M'}\subset H^{\beta}(\mathbb{T})$.

In order to show global well-posedness, we use the energy estimate \eqref{eqn: energy estimate of the Stokes case} as well as \eqref{eqn: H beta bound of U_Y} to derive that,
$$
\frac{d}{dt}\|X^\varepsilon\|_{H^\beta(\mathbb{T})}\leq C\left(\|X_0\|_{\dot{H}^1(\mathbb{T})},\varepsilon,\beta\right) \|X^\varepsilon\|_{H^\beta(\mathbb{T})}.
$$
This implies an a priori bound for the local solution
$$
\|X^\varepsilon\|_{C_{[0,T]}H^{\beta}(\mathbb{T})}\leq \exp\left[C\left(\|X_0\|_{\dot{H}^1(\mathbb{T})},\varepsilon,\beta\right)T\right]\|X_0\|_{H^{\beta}(\mathbb{T})}.
$$
Then the global well-posedness follows from a continuation argument.
\end{proof}

\begin{remark}\label{rmk: deteriorating bounds of the regularized solution}
In spite of the global well-posedness, $\|X(t)\|_{H^{\beta}(\mathbb{T})}$ only admits an exponentially growing bound.
As $\varepsilon\rightarrow 0^+$, its growth rate deteriorates (i.e., increases) very quickly and diverges to $+\infty$, which is not helpful for proving any convergence of $X^\varepsilon$.
\end{remark}

\begin{remark}\label{rmk: well-stretched condition for the regularized problem}
As opposed to Proposition \ref{prop: local well-posedness of the singular problem}, Theorem \ref{thm: global well-posedness of the regularized problem} does not assume the well-stretched condition \eqref{eqn: well-stretched condition} for $X_0$.
However, for given $\varepsilon$, if we do impose that with the stretching constant $\lambda$ for sufficiently smooth initial data, say $X_0\in H^\beta(\mathbb{T})$ with $\beta >3/2$, the solution $X^\varepsilon$ should satisfy the well-stretched condition with stretching constant $\lambda/2$ in short time, which may depend on $\varepsilon$.
In fact, this can be derived from the facts that $X^\varepsilon(t)$ is (locally) continuous in $H^\beta(\mathbb{T})$, and $H^\beta(\mathbb{T})\hookrightarrow Lip(\mathbb{T})$ for $\beta> 3/2$.
\end{remark}

\section{Static Error Estimates for the String Velocity}\label{section: static error estimates}
In this section, we shall establish static estimates for $U_Y^\varepsilon-U_Y$ for a given string configuration $Y$.
Its motivation has been briefly discussed in Section \ref{section: scheme of the proof}.
We will first 
set up the analysis in Sections \ref{section: elastic force of more general form}-\ref{section: contour dynamic formulation}.
Main results of this section, of which Proposition \ref{prop: static regularization error estimate for Hookean case} is a special case, are collected in Section \ref{section: statements of static error estimates}.
Their proofs are left to Sections \ref{section: preliminary estimates}-\ref{section: static H^1 error estimates}; first-time readers may skip them so as not to get distracted from the bigger picture of the paper.

\subsection{Assumptions on the elastic force of general form}\label{section: elastic force of more general form}
Static estimates for the regularization error are of independent interest, and we shall discuss them with greater generality.
In this section, we introduce assumptions on the elastic force of more general form that will be used throughout this section.

Given a string configuration $Y$, the elastic force in the Lagrangian coordinate is generally given by 
\begin{equation}
F_Y(s) = \partial_s\left(\mathcal{T}(|Y'(s)|,s)\frac{Y'(s)}{|Y'(s)|}\right).
\label{eqn: Lagrangian representation of the elastic force revisited}
\end{equation}
Here $\mathcal{T}: \mathbb{R}_+\times \mathbb{T}\rightarrow [0,+\infty)$ is the tension in the string, which depends on the position and the way the string is locally stretched, characterized by $p = |Y'(s)|$ with abuse of notations.
In physics, $\mathcal{T} = \mathcal{T}(p,s)$ is determined by the local constitutive law of elasticity of the string material:
assuming the string material has a local elastic energy density $\mathcal{E}=\mathcal{E}(p,s)$,
which is allowed to be spatially inhomogeneous, then $\mathcal{T}(p,s) = \partial_p \mathcal{E}(p,s)$.
For instance, Hookean elasticity admits $\mathcal{E}(p) = k_0p^2/2$ and $\mathcal{T}(p) = k_0p$, with $k_0$ being the Hooke's constant.

Define
\begin{equation}
S(p,s) = \frac{\mathcal{T}(p,s)}{p}
\label{eqn: definition of the generalized stiffness coefficient}
\end{equation}
to be the generalized stiffness coefficient.
In the Hookean elasticity case, $S(p)\equiv k_0$, which exactly characterizes stiffness of the string.
With this notation,
\begin{equation}\label{eqn: rewriting F using S}
F_Y(s) = \partial_s[S(|Y'(s)|,s)Y'(s)].
\end{equation}
In the rest of this section, we will study elastic force in this general form, where $S = S(p,s)$ satisfies the following assumption.
\begin{enumerate}[label = (\alph*)]
\item \label{assumption: C 1,1 regularity assumption}$S = S(p,s)\in C^{1,1}_{loc}(\mathbb{R}_+\times \mathbb{T})$. 
To be more precise, for $\forall\,0<m<M$, $\exists\,\mu(m,M)>0$, such that for all $(p,s)\in [m,M]\times \mathbb{T}$,
\begin{equation}
|S(p,s)|+|\partial_s S(p,s)|+M|\partial_p S(p,s)|\leq \mu(m,M),
\label{eqn: C0 and C1 bound for S}
\end{equation}
and for all $p_1,p_2\in[m,M]$ and $s_1,s_2\in \mathbb{T}$,
\begin{equation}
\begin{split}
&\;|\partial_s S(p_1,s_1)-\partial_s S(p_2,s_2)| + M|\partial_p S(p_1,s_1)-\partial_p S(p_2,s_2)|\\
\leq &\;\mu(m,M)(|s_1-s_2|+M^{-1}|p_1-p_2|).
\end{split}
\label{eqn: C11 bound for S}
\end{equation}

\end{enumerate}
Having homogeneous polynomials of $p$ in mind as prototypical examples of $S$, we purposefully add extra powers of $M$ above in order not to break homogeneity of the estimates.

Although the assumption is not intended to be the weakest or the most comprehensive, it is general enough to include a broad family of elasticity models.
For instance, any spatially homogeneous elasticity model with $\mathcal{T} = \mathcal{T}(p)$ for $\mathcal{T}\in C^3_{loc}(0,\infty)$ satisfies the assumptions.
In particular, linear elasticity model, with $\mathcal{T}(p) = k(p-p_0)$ and $S(p) = k-kp_0/p$, is admissible. 
Here $p_0\geq 0$ is the natural length of fully relaxed string material; when $p_0 = 0$, it characterizes the Hookean elasticity. 
An example that does not fulfill the assumption is the finitely extensible nonlinear elastic (FENE) model \cite{warner1972kinetic}, given by, e.g.,
\begin{equation}
\mathcal{T}(p) = \frac{kp}{1-(p/p_{\mathrm{max}})^2},\quad S(p) = \frac{k}{1-(p/p_{\mathrm{max}})^2}, \quad p\in[0,p_{\mathrm{max}}),
\end{equation}
and $\mathcal{T}(p)=\mathcal{S}(p)=\infty$ otherwise.
Even though the arguments in this section may also work for this case up to some adaptation, we simply avoid that technicality.

In practice, elasticity laws may be time-varying.
For example, for a parametrically-forced string that models active biological tissues, $S(p,t) = a + b\sin (\omega t)$ with $|b|<a$ \cite{cortez2004parametric}.
For such models, it suffices to freeze the elasticity law and validate the assumption for each time slice.
We may additionally require the bounds to be uniform in time, so that error estimates apply to all time.
We leave the technical discussion to interested readers. 

To this end, we state a useful lemma that roughly claims that $F_Y$ behaves like $Y_{ss}$ in regularity.
This can be viewed as a generalization of the obvious fact in the Hookean elasticity case where $F_Y = k_0 Y_{ss}$.

\begin{lemma}[Estimates for $F$]\label{lemma: H^1/2 estimate of F}
Assume $Y\in H^{2+\theta}(\mathbb{T})$ satisfies the well-stretched condition \eqref{eqn: well-stretched condition} with constant $\lambda$.
Let $F_Y(s)$ be defined by \eqref{eqn: rewriting F using S}. 
For $\theta\in[0,1)$, under the Assumption \ref{assumption: C 1,1 regularity assumption} on $S$, 
\begin{align}
|F_Y(s)|\leq &\;C\mu(|Y'(s)|+|Y''(s)|),\label{eqn: pointwise bound of F}\\
\|F_Y\|_{\dot{H}^{\theta}(\mathbb{T})}\leq &\;C_\theta\mu\|Y\|_{\dot{H}^{2+\theta}(\mathbb{T})},\label{eqn: H^theta estimate of F}
\end{align}
where
\begin{equation}\label{eqn: choice of mu in the static estimate}
\mu = \mu(\lambda, c\|Y''\|_{L^2(\mathbb{T})})
\end{equation}
is defined by \eqref{eqn: C0 and C1 bound for S} and \eqref{eqn: C11 bound for S}.
Here $c\geq 1$ is a universal constant such that $\|Y'\|_{L^\infty}\leq c\|Y''\|_{L^2}$.
With abuse of notations, we understand $\dot{H}^0(\mathbb{T})$ as $L^2(\mathbb{T})$.

\end{lemma}
Its proof is a straightforward calculation.
We leave it to Appendix \ref{section: proof of lemma on regularity of F}.

\subsection{The Contour Dynamic Formulations}\label{section: contour dynamic formulation}
Assume $\phi$, the profile of the regularized $\delta$-function in \eqref{eqn: def of delta_eps}, satisfies the conditions in Theorem \ref{thm: error estimates of the regularized problem}.
We also assume that $Y$ at least has $H^2$-regularity, and satisfies the well-stretched condition \eqref{eqn: well-stretched condition} with constant $\lambda$.

We first recall the contour dynamic formulation of the original Stokes immersed boundary problem \eqref{eqn: Stokes equation}-\eqref{eqn: motion of the membrane and initial configuration} \cite{lin2017solvability}.
Given a string configuration $Y$, thanks to the stationary Stokes equation, the flow field is instantaneously determined by the force exerted on the fluid.
Combining \eqref{eqn: Stokes equation}-\eqref{eqn: elastic forcing}, we formally derive that
\begin{equation}
\begin{split}
u(x) = &\;\int_{\mathbb{R}^2}G(x-y)f(y)\,dy \\
=&\; \int_{\mathbb{R}^2}G(x-y)\int_{\mathbb{T}}F_Y(s)\delta(y-Y(s'))\,ds'\\
=&\; \int_{\mathbb{T}}G(x-Y(s'))F_Y(s')\,ds',
\end{split}
\label{eqn: formal derivation of velocity field in the un-regularized case}
\end{equation}
where $G$ is the fundamental solution defined in \eqref{eqn: stokeslet}.
Take $x = Y(s)$.
By \eqref{eqn: rewriting F using S} and integration by parts, 
\begin{equation}
U_Y(s) =u(Y(s)) = \mathrm{p.v.} \int_{\mathbb{T}} -\partial_{s'}[G(Y(s)-Y(s'))]SY'(s')\,ds'.
\label{eqn: contour dynamic representation for string velocity the singular problem}
\end{equation}
Here $S=S(|Y'(s)|,s)$.

It is shown in \cite{lin2017solvability} that if $S \equiv 1$ and $Y\in H^2(\mathbb{T})$ satisfying the well-stretched condition \eqref{eqn: well-stretched condition}, the above derivation can be made rigourous.
We claim that, following an almost identical argument, \eqref{eqn: contour dynamic representation for string velocity the singular problem} can be justified for general $S$ that satisfies the Assumptions \ref{assumption: C 1,1 regularity assumption} 
and for all $Y\in H^2(\mathbb{T})$ satisfying \eqref{eqn: well-stretched condition}.
See \cite[Section 2]{lin2017solvability} for more details.
\begin{remark}
In \eqref{eqn: contour dynamic representation for string velocity the singular problem}, we treat the restriction of $u$ on $Y(\mathbb{T})$ as the string velocity, which is not obviously valid.
In fact, it is noted in \cite[Proposition 1.1 and Lemma 2.2]{tong2018thesis} that higher space-time regularity of $Y$ is needed in order to uniquely define transport of a lower dimensional object $Y(\mathbb{T})$ as well as the restriction of $u$ on it.
Since in this section we only focus on static error estimates, we shall avoid this subtlety, but simply assume \eqref{eqn: contour dynamic representation for string velocity the singular problem} is a valid formula for the string velocity. 
\end{remark}

In the $\varepsilon$-regularized problem, for $Y\in H^2(\mathbb{T})$, it is not difficult to show that $f^\varepsilon$ defined as in \eqref{eqn: regularization of forcing} is as regular as $\delta_\varepsilon$; $u^\varepsilon$ is even more regular locally.
Hence, we can rigorously perform a derivation similar to \eqref{eqn: formal derivation of velocity field in the un-regularized case} and \eqref{eqn: contour dynamic representation for string velocity the singular problem}, and justify that \eqref{eqn: Stokes equation in the regularized IB problem}-\eqref{eqn: regularized motion of the membrane and initial configuration} gives 
\begin{equation}
U^\varepsilon_{Y}(s) = \int_{\mathbb{T}}G^\varepsilon(Y(s)-Y(s'))F_Y(s')\,ds' = \int_{\mathbb{T}}-\partial_{s'}[G^\varepsilon(Y(s)-Y(s'))]SY'(s')\,ds'.
\label{eqn: contour dynamic formulation of the regularized problem crude form}
\end{equation}
Here 
\begin{equation}
G^\varepsilon(x) = \int_{\mathbb{R}^2} \int_{\mathbb{R}^2} \delta_\varepsilon(x-y)G(y-z)\delta_\varepsilon(z)\,dydz = [(\delta_\varepsilon\ast\delta_\varepsilon)\ast G](x)
\end{equation}
is the regularized fundamental solution.
Define
\begin{equation}
\label{eqn: def of varphi}
\varphi(x) = \phi\ast \phi(x),\quad \varphi_\varepsilon(x) \triangleq \frac{1}{\varepsilon^2}\varphi\left(\frac{x}{\varepsilon}\right).
\end{equation}
Obviously, $\varphi$ is compactly supported, radially symmetric, and smooth; in addition, $\varphi_\varepsilon = \delta_\varepsilon\ast \delta_\varepsilon$.
Hence, we may write $G^\varepsilon = \varphi_\varepsilon\ast G$, and \eqref{eqn: contour dynamic formulation of the regularized problem crude form} becomes
\begin{equation}
\begin{split}
(U^\varepsilon_Y)_j(s) =&\;\int_\mathbb{T} -\partial_{s'}[G^\varepsilon_{jk}(Y(s)-Y(s'))]\cdot SY'_k(s')\,ds'\\
=&\; \int_\mathbb{T}SY_l'(s')Y_k'(s')[ \partial_l\varphi_\varepsilon\ast G_{jk}](Y(s)-Y(s'))\,ds'.
\end{split}
\label{eqn: contour dynamic formulation of the regularized problem crude form taking derivative}
\end{equation}
Here $G_{jk}$ is the $(j,k)$-entry of $G$.
Define the Fourier transform in $\mathbb{R}^d$ and its inverse as follows
\begin{equation}
\hat{f}(\xi) = \int_{\mathbb{R}^d} f(x)e^{-ix\cdot \xi}\,dx,\quad f(x) = \frac{1}{(2\pi)^d}\int_{\mathbb{R}^d} \hat{f}(\xi)e^{ix\cdot \xi}\,d\xi.
\end{equation}
With this definition,
\begin{equation}
\widehat{f\ast g} = \hat{f}\hat{g}.
\label{eqn: Fourier transform makes convolution into product}
\end{equation}
%
Since $\partial_i\varphi_\varepsilon\ast G$ is sufficiently regular, we rewrite it in \eqref{eqn: contour dynamic formulation of the regularized problem crude form taking derivative} using its Fourier transform
\begin{equation}
\begin{split}
(U^\varepsilon_Y)_j(s)&=\frac{1}{(2\pi)^2}\int_\mathbb{T}ds'\,SY_l'(s')Y_k'(s')\int_{\mathbb{R}^2} e^{i(Y(s)-Y(s'))\cdot\xi} (\partial_l\varphi_\varepsilon\ast G_{jk})\hat{\,}(\xi)\,d\xi\\
&=\frac{1}{4\pi^2}\int_\mathbb{T}ds'\,S\int_{\mathbb{R}^2} e^{i(Y(s)-Y(s'))\cdot\xi} \cdot i\xi\cdot Y'(s')
\frac{\hat{\varphi}_\varepsilon(\xi)}{|\xi|^2}\left(Y_j'(s')-\frac{Y'(s')\cdot \xi}{|\xi|^2}\xi_j\right)\,d\xi.
\end{split}
\label{eqn: crude fourier representation of the velocity}
\end{equation}
In the last line, we used the fact that 
\begin{equation}
[\partial_l \varphi_\varepsilon *G]\hat{\,}(\xi) = \frac{i\xi_l \hat{\varphi}_\varepsilon(\xi)}{|\xi|^2}\left(Id-\frac{\xi\otimes \xi}{|\xi|^2}\right).
\end{equation}

In order to simplify \eqref{eqn: crude fourier representation of the velocity}, we introduce a new variable $\eta \triangleq (\eta_1,\eta_2)\in\mathbb{R}^2$ to replace $\xi$, such that
\begin{equation}
\xi = \frac{Y(s)-Y(s')}{|Y(s)-Y(s')|}\cdot \eta_1+\left(\frac{Y(s)-Y(s')}{|Y(s)-Y(s')|}\right)^\perp \cdot \eta_2.
\end{equation}
Note that when $s'\not = s$, $|Y(s')-Y(s)|\not = 0$ due to \eqref{eqn: well-stretched condition}.
Here $\perp$ means rotating a vector in $\mathbb{R}^2$ counter-clockwise by $\pi/2$.
We also denote
\begin{align}
& A(s,s') = Y'(s')\cdot (Y(s')-Y(s)),\label{eqn: def of A(s,s')}\\
& B(s,s') = Y'(s')\cdot (Y(s')-Y(s))^\perp,\label{eqn: def of B(s,s')}\\
& D(s,s') = |Y(s')-Y(s)|.\label{eqn: def of D(s,s')}
\end{align}
Hence,
\begin{equation}
Y'(s') = \frac{1}{D^2}\left[A(Y(s')-Y(s))+B(Y(s')-Y(s))^\perp\right],
\label{eqn: representation of Y'(s') using secants}
\end{equation}
and
\begin{equation}
Y'(s')\cdot \xi = -\frac{A\eta_1 +B\eta_2}{D}.
\end{equation}
Then \eqref{eqn: crude fourier representation of the velocity} becomes
\begin{equation}
\begin{split}
U^\varepsilon_Y(s)
&=\frac{1}{4\pi^2}\int_\mathbb{T}ds'\,S\int_{\mathbb{R}^2} d\eta\,e^{iD\eta_1} \cdot \frac{-i(A\eta_1+B\eta_2)}{D}\cdot
\frac{\hat{\varphi}_\varepsilon(\eta)}{\eta_1^2+\eta_2^2}\\
&\qquad \cdot\left(\frac{1}{D^2}\left[A(Y(s')-Y(s))+B(Y(s')-Y(s))^\perp\right] \right.\\
&\qquad \quad \left.-\frac{A\eta_1+B\eta_2}{(\eta_1^2+\eta_2^2)D^2}[(Y(s')-Y(s))\eta_1+(Y(s')-Y(s))^\perp\eta_2]\right).
\end{split}
\end{equation}
Here we used  $\hat{\varphi}_\varepsilon(\xi) = \hat{\varphi}_\varepsilon(\eta)$ thanks to the radial symmetry of $\hat{\varphi}_\varepsilon$.
We further simplify this using the fact that $\hat{\varphi}_\varepsilon(\eta_1,\eta_2)$ is even in $\eta_2$.
\begin{equation}
\begin{split}
&\;U^\varepsilon_Y(s)\\
=&\;-\frac{1}{4\pi^2}\int_\mathbb{T}ds'\,\frac{S}{D^3}(Y(s')-Y(s))\int_{\mathbb{R}^2} d\eta\,e^{iD\eta_1} \cdot i(A\eta_1+B\eta_2)(A\eta_2 - B\eta_1)\eta_2
\frac{\hat{\varphi}_\varepsilon(\eta)}{(\eta_1^2+\eta_2^2)^2}\\
&\;+\frac{1}{4\pi^2}\int_\mathbb{T}ds'\,\frac{S}{D^3}(Y(s')-Y(s))^\perp \int_{\mathbb{R}^2} d\eta\,e^{iD\eta_1} \cdot i(A\eta_1+B\eta_2)(A\eta_2 - B\eta_1)\eta_1
\frac{\hat{\varphi}_\varepsilon(\eta)}{(\eta_1^2+\eta_2^2)^2}\\
=&\;\frac{1}{4\pi}\int_\mathbb{T}ds'\,\frac{S(A^2 - B^2)}{D^3}(Y(s')-Y(s))\cdot \frac{1}{\pi}\int_{\mathbb{R}^2} d\eta\,e^{iD\eta_1}
\frac{-i\eta_1\eta_2^2}{(\eta_1^2+\eta_2^2)^2}\hat{\varphi}_\varepsilon(\eta)\\
&\;+\frac{1}{4\pi}\int_\mathbb{T}ds'\,\frac{SAB}{D^3}(Y(s')-Y(s))^\perp \cdot \frac{1}{\pi}\int_{\mathbb{R}^2} d\eta\,e^{iD\eta_1}
\frac{-i(\eta_1^2-\eta_2^2)\eta_1}{(\eta_1^2+\eta_2^2)^2}\hat{\varphi}_\varepsilon(\eta).
\end{split}
\label{eqn: fourier representation of the velocity}
\end{equation}
Here the inner integral is absolutely integrable so that Fubini's theorem can be applied to omit terms that are odd in $\eta_2$.
For $x\in \mathbb{R}$, define
\begin{align}
f_1(x) &= \frac{x}{\pi}\int_{\mathbb{R}^2} e^{ix\eta_1} \frac{-i\eta_1\eta_2^2}{(\eta_1^2+\eta_2^2)^2}\hat{\varphi}(\eta)\,d\eta,\label{eqn: def of f_1}\\
f_2(x) &= \frac{x}{\pi}\int_{\mathbb{R}^2} e^{ix\eta_1} \frac{-i\eta_1(\eta_1^2-\eta_2^2)}{(\eta_1^2+\eta_2^2)^2}\hat{\varphi}(\eta)\,d\eta.\label{eqn: def of f_2}
\end{align}
Since $\hat{\varphi}_{\varepsilon}(\eta) = \hat{\varphi}(\varepsilon\eta)$, \eqref{eqn: fourier representation of the velocity} becomes
\begin{equation}
\begin{split}
U^\varepsilon_Y(s)
=&\;\frac{1}{4\pi}\int_\mathbb{T}ds'\,\frac{S(A^2 - B^2)}{D^4}(Y(s')-Y(s))\cdot f_1\left(\frac{D}{\varepsilon}\right)\\
&\;+\frac{1}{4\pi}\int_\mathbb{T}ds'\,\frac{SAB}{D^4}(Y(s')-Y(s))^\perp \cdot f_2\left(\frac{D}{\varepsilon}\right).
\end{split}
\label{eqn: string motion in the regularized case}
\end{equation}

On the other hand, we rewrite $U_Y$ in \eqref{eqn: contour dynamic representation for string velocity the singular problem} as
\begin{equation}
\begin{split}
U_Y(s) & = \mathrm{p.v.} \int_{\mathbb{T}} -\partial_{s'}[G(Y(s)-Y(s'))]SY'(s')\,ds'\\
& = \frac{1}{4\pi}\mathrm{p.v.} \int_{\mathbb{T}} S\cdot\left[-\frac{|Y'(s')|^2}{|Y(s')-Y(s)|^2} + \frac{2[(Y(s')-Y(s))\cdot Y'(s')]^2}{|Y(s')-Y(s)|^4}\right](Y(s')-Y(s))\,ds'\\
& = \frac{1}{4\pi}\mathrm{p.v.} \int_{\mathbb{T}} \frac{S(A^2-B^2)}{D^4}(Y(s')-Y(s))\,ds'.
\end{split}
\label{eqn: string motion in the singular case}
\end{equation}
Combining \eqref{eqn: string motion in the regularized case} and \eqref{eqn: string motion in the singular case}, we obtain a representation of the regularization error in the string velocity
\begin{equation}
\begin{split}
U^\varepsilon_Y(s)-U_Y(s)
=&\;\frac{1}{4\pi}\mathrm{p.v.}\int_\mathbb{T}\frac{S(A^2 - B^2)}{D^4}(Y(s')-Y(s))\cdot f_3\left(\frac{D}{\varepsilon}\right)\,ds'\\
&\;+\frac{1}{4\pi}\int_\mathbb{T}\frac{SAB}{D^4}(Y(s')-Y(s))^\perp \cdot f_2\left(\frac{D}{\varepsilon}\right)\,ds',
\end{split}
\label{eqn: difference between regularized and singular string motion}
\end{equation}
where
\begin{equation}
f_3(x) \triangleq f_1(x)-1.
\label{eqn: def of f_3}
\end{equation}

Concerning $f_2$ and $f_3$, we can show that
\begin{lemma}[Estimates for $f_2$ and $f_3$]\label{lemma: decay estimate of f_1 f_2 f_3}
Let $f_2$ and $f_3$ be defined by \eqref{eqn: def of f_1}, \eqref{eqn: def of f_2}, and \eqref{eqn: def of f_3}.
For $k = 0,1,2$,
\begin{equation}\label{eqn: estimate for f_2 f_3}
|f_2^{(k)}(x)|+|f_3^{(k)}(x)| \leq \frac{C}{1+|x|^{k+2}}.
\end{equation}
Here $f_2^{(k)}$ and $f_3^{(k)}$ denote $k$-th derivatives of $f_2$ and $f_3$, respectively.
Let $m_2$ be defined in \eqref{eqn: def of M_2}.
If in addition $m_2 = 0$, then for $k = 0,1,2$,
\begin{equation}\label{eqn: improved estimate for f_2 f_3}
|f_2^{(k)}(x)|+|f_3^{(k)}(x)| \leq \frac{C}{1+x^4}.
\end{equation}
\end{lemma}

\begin{lemma}[Integrals of $f_2$ and $f_3$]\label{lemma: integrals of f_2 and f_3}
Let $m_1$ be defined in \eqref{eqn: def of M_1}.
Then
\begin{equation}
\int_\mathbb{R} f_2(x)\,dx =  4 m_1,\quad \int_\mathbb{R} f_3(x)\,dx = -4 m_1.
\label{eqn: integrals of f_2 and f_3}
\end{equation}
\end{lemma}
Their proofs involve repeated integration by parts; we leave them to Appendix \ref{section: study of the auxiliary functions f_i}.

For future use, define
\begin{equation}\label{eqn: def of f_4 and f_5}
f_4 \triangleq f_2-f_3,\quad f_5 \triangleq f_2-2f_3.
\end{equation}
Estimates for $f_4$ and $f_5$ follow from Lemma \ref{lemma: decay estimate of f_1 f_2 f_3}.

\subsection{Statements of the static error estimates}\label{section: statements of static error estimates}

For clarity, we present static error estimates for the string velocities as the main results of this section.
Their proofs are left to Sections \ref{section: preliminary estimates}-\ref{section: static H^1 error estimates}.

On the $L^2$-estimates for the regularization error, we have
\begin{proposition}[Static $L^2$-error estimate]\label{prop: L^2 error estimate of the string velocity}
Assume $\phi$, the profile of the regularized $\delta$-function, satisfies the assumption in Theorem \ref{thm: error estimates of the regularized problem}.
Let $m_1$ and $m_2$ be defined in \eqref{eqn: def of M_1} and \eqref{eqn: def of M_2}, respectively.

Suppose $Y\in H^{2+\theta}(\mathbb{T})$ with $\theta\in[1/4,1)$, satisfying the well-stretched condition \eqref{eqn: well-stretched condition} with $\lambda>0$.
Let $F_Y(s)$ be defined by \eqref{eqn: rewriting F using S}, 
with $S$ satisfying the Assumption \ref{assumption: C 1,1 regularity assumption}. 
Given $\varepsilon>0$, let $U^\varepsilon_Y$ and $U_Y$ be the string velocities corresponding to $Y$ in the $\varepsilon$-regularized problem and the original problem, defined by \eqref{eqn: string motion in the regularized case} and
\eqref{eqn: string motion in the singular case}, respectively.
Provided that $\varepsilon\ll \lambda$,
\begin{equation}
\begin{split}
\|U^\varepsilon_Y-U_Y\|_{L^2(\mathbb{T})}
\leq &\;\frac{m_1 \varepsilon}{\pi}\left\| \frac{F_Y(s)\cdot Y'(s)}{|Y'(s)|^2}\right\|_{L^2(\mathbb{T})}+\frac{C\mu\varepsilon^{1+\theta} \ln^\theta(\lambda /\varepsilon)}{\lambda^{1+\theta}}\|Y\|_{\dot{H}^{2+\theta}(\mathbb{T})}\\
&\; +\frac{C\mu\varepsilon^2 \ln(\lambda/\varepsilon)}{\lambda^5}\|Y'\|_{L^\infty(\mathbb{T})}^2 \|Y''\|^2_{L^4(\mathbb{T})}.
\end{split}
\label{eqn: L^2 error estimate final version}
\end{equation}
Here $C$ is a universal constant depending on $\theta$, 
and $\mu$ is defined in \eqref{eqn: choice of mu in the static estimate}.

If, in addition, $m_2 = 0$, the logarithmic factors above can be removed.
\end{proposition}

Note that we require $\theta\geq \frac{1}{4}$ in Proposition \ref{prop: L^2 error estimate of the string velocity} because the estimate involves $\|Y''\|_{L^4}$ and $H^{2+\frac{1}{4}}(\mathbb{T})\hookrightarrow W^{2,4}(\mathbb{T})$.

\begin{corollary}[Improved static $L^2$-error estimate for the normal velocity]\label{coro: L^2 error estimate of the string normal velocity}
Under the assumptions of Proposition \ref{prop: L^2 error estimate of the string velocity}, the regularization error of the string normal velocity satisfies 
\begin{equation}
\begin{split}
&\;\left\|(U^\varepsilon_Y(s)-U_Y(s))\cdot \frac{Y'(s)^\perp}{|Y'(s)|}\right\|_{L^2(\mathbb{T})}\\
\leq &\;\frac{C\mu\varepsilon^{1+\theta} \ln^\theta(\lambda /\varepsilon)}{\lambda^{1+\theta}}\|Y\|_{\dot{H}^{2+\theta}(\mathbb{T})}+\frac{C\mu\varepsilon^2 \ln(\lambda/\varepsilon)}{\lambda^5}\|Y'\|_{L^\infty(\mathbb{T})}^2 \|Y''\|^2_{L^4(\mathbb{T})}.
\end{split}
\label{eqn: L^2 error estimate normal direction}
\end{equation}
If, in addition, $m_2 = 0$, the logarithmic factors above can be removed.
\end{corollary}

We will prove Proposition \ref{prop: L^2 error estimate of the string velocity} and Corollary \ref{coro: L^2 error estimate of the string normal velocity} in Section \ref{section: L^2 static error}.

Proposition \ref{prop: L^2 error estimate of the string velocity} implies that in general, $\|U_Y^\varepsilon-U_Y\|_{L^2}$ is of order $O(\varepsilon)$,
while Corollary \ref{coro: L^2 error estimate of the string normal velocity} shows that such $O(\varepsilon)$-error arises only from the tangential component of the string velocity.
Improved error bounds may be achieved when $(F_Y\cdot Y')/|Y'|\equiv 0$ or $m_1=0$.
We should highlight that this has a clear physical interpretation, which will be discussed in Section \ref{section: discussion}.
At this point, we remark that the former condition 
is true only when $S(p,s) = kp^{-1}$, with $k\geq 0$ being a constant.
In such case, the normal force in the Eulerian coordinate is proportional to the local curvature of the string,
which is the situation when describing moving interfaces with surface tension between two fluid domains  \cite{solonnikov1986solvability,tanaka1993global,solonnikov2014theory}.
The latter condition $m_1 = 0$ may be achieved by suitably choosing the regularized $\delta$-function. 
Find more discussions on this in Section \ref{section: discussion} as well.

Our analysis also indicates that another moment-type condition $m_2 = 0$ may only benefit the error bound by a logarithmic factor in higher-order terms.

The next result is concerned with $H^1$-estimates of $U_Y^\varepsilon-U_Y$, whose proof will be provided in Section \ref{section: static H^1 error estimates}.
\begin{proposition}[Static $H^1$-error estimate]\label{prop: H^1 error estimate of the string velocity}
Under the assumptions of Proposition \ref{prop: L^2 error estimate of the string velocity},
\begin{equation}
\|U^\varepsilon_Y-U_Y\|_{\dot{H}^1(\mathbb{T})}\leq \frac{C\mu\varepsilon^\theta}{\lambda^\theta}\|Y\|_{\dot{H}^{2+\theta}(\mathbb{T})}+\frac{C\mu \varepsilon}{\lambda^4}\|Y'\|_{L^\infty(\mathbb{T})}^2\|Y''\|_{L^4(\mathbb{T})}^2.
\label{eqn: H^1 error estimate of the string velocity}
\end{equation}
\end{proposition}

Proposition \ref{prop: static regularization error estimate for Hookean case} follows from Proposition \ref{prop: L^2 error estimate of the string velocity} and Proposition \ref{prop: H^1 error estimate of the string velocity} by taking $F_Y = Y_{ss}$ and $\mu = 1$.

\subsection{Preliminary estimates}\label{section: preliminary estimates}
In what follows, we shall write $s' = s+\tau$.
Let $\mathcal{M}$ be the Hardy-Littlewood maximal operator on $\mathbb{T}$.
Define
\begin{equation}\label{eqn: def of Q and R}
Q = \|Y'\|_{L^\infty(\mathbb{T})}, \quad R(s,\tau) =|\mathcal{M}Y''(s+\tau)|+ |Y''(s+\tau)|.
\end{equation}
We shall omit the arguments of $R$ whenever it is convenient.
As before, $SY'(s)\triangleq S(|Y'(s)|)Y'(s)$.

\begin{lemma}\label{lemma: preliminary estimates as building blocks}
Suppose $Y\in H^2(\mathbb{T})$ and $\tau \in [-\pi,\pi]$.
Let $A$, $B$, $D$ and $S$ be defined in \eqref{eqn: def of A(s,s')}, \eqref{eqn: def of B(s,s')}, \eqref{eqn: def of D(s,s')} and \eqref{eqn: definition of the generalized stiffness coefficient}, respectively.
Then
\begin{enumerate}
\item We have
\begin{align}
|Y(s+\tau)-Y(s)-\tau Y'(s)|\leq &\;C\tau^2R,\label{eqn: estimate of Y(s')-Y(s)-tau Y'(s)}\\
|Y(s+\tau)-Y(s)-\tau Y'(s+\tau)|\leq &\;C\tau^2R.\label{eqn: estimate of Y(s')-Y(s)-tau Y'(s')}
\end{align}
Hence,
\begin{align}
|D(s,s+\tau) -|\tau||Y'(s)||\leq&\; C\tau^2R,\label{eqn: approximation of D version 1}\\
|D(s,s+\tau) -|\tau||Y'(s+\tau)||\leq&\; C\tau^2R.\label{eqn: approximation of D version 2}
\end{align}

\item
\begin{equation}
|B(s,s+\tau)| \leq C \tau^2 QR.\label{eqn: improved estimate for B}
\end{equation}

\item Under the Assumption \ref{assumption: C 1,1 regularity assumption} on $S$,
\begin{equation}
|SY'(s+\tau)-SY'(s)|\leq C\mu |\tau|(Q+R),
\label{eqn: estimate of difference between SY'}
\end{equation}
with $\mu$ defined in \eqref{eqn: choice of mu in the static estimate}. 
\end{enumerate}

\begin{proof}
We simply calculate that
\begin{equation}
\begin{split}
|Y(s+\tau)-Y(s)-\tau Y'(s)|=&\;\left|\int_0^\tau Y'(s+\eta)-Y'(s)\,d\eta\right|\\
\leq &\;\int_0^{|\tau|} |Y'(s+\eta)-Y'(s+\tau)|+|Y'(s)-Y'(s+\tau)|\,d\eta\\
\leq &\;C\tau^2|\mathcal{M}Y''(s+\tau)|.
\end{split}
\end{equation}
\eqref{eqn: estimate of Y(s')-Y(s)-tau Y'(s')} can be shown in a similar manner; then \eqref{eqn: approximation of D version 1} and \eqref{eqn: approximation of D version 2} follow immediately.

\eqref{eqn: improved estimate for B} can be proved by observing
\begin{equation}\label{eqn: improved formula for B}
B(s,s+\tau) = (Y(s+\tau)-Y(s)-\tau Y'(s+\tau))\cdot Y'(s+\tau)^\perp.
\end{equation}

Finally, by Assumption \ref{assumption: C 1,1 regularity assumption}, 
\begin{equation}
\begin{split}
&\;|SY'(s+\tau)-SY'(s)|\\
\leq &\;|S(|Y'(s+\tau)|,s+\tau)|\cdot|Y'(s+\tau)-Y'(s)|\\
&\;+|S(|Y'(s+\tau)|,s+\tau)-S(|Y'(s)|,s)|\cdot |Y'(s)|\\
\leq &\;|Y'(s+\tau)-Y'(s)|\cdot \left(\mu+\frac{\mu|Y'(s)|}{\|Y'\|_{L^\infty(\mathbb{T})}}\right)+\mu|\tau||Y'(s)|\\
\leq &\;C\mu |\tau|(|\mathcal{M}Y''(s+\tau)|+|Y'(s)|).
\end{split}
\end{equation}
This proves \eqref{eqn: estimate of difference between SY'}.
\end{proof}
\end{lemma}

\begin{lemma}\label{lemma: preliminary estimates of derivatives as building blocks}
Assume $Y\in H^{2+\theta}(\mathbb{T})$ with $\theta\in(0,1)$, and $\tau \in [-\pi,\pi]$.
Let $A$, $B$, and $D$ be defined in \eqref{eqn: def of A(s,s')}-
\eqref{eqn: def of D(s,s')}, respectively.
Then
\begin{align}
|\partial_s A(s,s+\tau)|\leq &\; C|\tau|QR,\label{eqn: estimate of A'}\\
|\partial_s B(s,s+\tau)|\leq &\; C|\tau|QR,\label{eqn: estimate of B'}\\
|\partial_s D(s,s+\tau)|\leq &\; C|\tau|R.\label{eqn: estimate of D'}
\end{align}
Moreover, 
for $\forall\,\beta\in (0,\theta]$,
\begin{equation}\label{eqn: estimate of B' involving H2.5 norm}
\begin{split}
|\partial_sB(s,s+\tau)|\leq &\;C\tau^2R^2\\
&\;+C|\tau|^{1+\beta}|Y'(s+\tau)|\left(\int_0^{|\tau|} \frac{|Y''(s+\tau-\eta)-Y''(s+\tau)|^2}{|\eta|^{1+2\beta}}\,d\eta\right)^{1/2}.
\end{split}
\end{equation}
\begin{proof}
We calculate that
\begin{align}
\partial_s A(s,s+\tau) = &\;Y''(s+\tau)\cdot (Y(s+\tau)-Y(s))+Y'(s+\tau)\cdot (Y'(s+\tau)-Y'(s)),\label{eqn: formula for derivative of A}\\
\partial_s B(s,s+\tau) = &\;Y''(s+\tau)\cdot (Y(s+\tau)-Y(s))^\perp+Y'(s+\tau)\cdot (Y'(s+\tau)-Y'(s))^\perp,\label{eqn: formula for derivative of B}\\
\partial_s D(s,s+\tau) = &\;\frac{(Y'(s+\tau)-Y'(s))\cdot (Y(s+\tau)-Y(s))}{|Y(s+\tau)-Y(s)|}.\label{eqn: formula for derivative of D}
\end{align}
Then \eqref{eqn: estimate of A'}-\eqref{eqn: estimate of D'} follow immediately from $|Y'(s+\tau)-Y'(s)|\leq |\tau||\mathcal{M}Y''(s+\tau)|$.
To prove \eqref{eqn: estimate of B' involving H2.5 norm}, we rewrite
\begin{equation}\label{eqn: calculation of B'}
\begin{split}
\partial_sB(s,s+\tau)
= &\;Y''(s+\tau)\cdot  (Y(s+\tau)-Y(s)-\tau Y'(s+\tau))^\perp \\
&\;+Y'(s+\tau)\cdot(Y'(s+\tau)-Y'(s)-\tau Y''(s+\tau))^\perp.
\end{split}
\end{equation}
Hence, by Lemma \ref{lemma: preliminary estimates as building blocks},
\begin{equation}
\begin{split}
&\;|\partial_sB(s,s+\tau)|\\
\leq &\;C|Y''(s+\tau)|\cdot \tau^2R+|Y'(s+\tau)|\left|\int_{0}^\tau Y''(s+\eta)-Y''(s+\tau)\,d\eta\right|.
\end{split}
\label{eqn: preliminary bound for derivative of B}
\end{equation}
With $\beta\in (0,\theta]$,
\begin{equation}
\begin{split}
&\;|\partial_sB(s,s+\tau)|\\
\leq &\;C\tau^2R^2\\
&\;+|Y'(s+\tau)|\left|\int_{0}^{|\tau|} |\eta-\tau|^{1+2\beta}\,d\eta\right|^{1/2}\left(\int_0^{|\tau| } \frac{|Y''(s+\eta)-Y''(s+\tau)|^2}{|\eta-\tau|^{1+2\beta}}\,d\eta\right)^{1/2}\\
\leq &\;C\tau^2R^2\\
&\;+C|\tau|^{1+\beta}|Y'(s+\tau)|\left(\int_0^{|\tau|} \frac{|Y''(s+\tau-\eta)-Y''(s+\tau)|^2}{|\eta|^{1+2\beta}}\,d\eta\right)^{1/2}.
\end{split}
\end{equation}
\end{proof}
\end{lemma}

\subsection{Proof of the static $L^2$-error estimate}\label{section: L^2 static error}

\begin{proof}[Proof of Proposition \ref{prop: L^2 error estimate of the string velocity}]

\setcounter{step}{0}
\begin{step}[Splitting the regularization error]
We start from rewriting \eqref{eqn: difference between regularized and singular string motion}.
By \eqref{eqn: representation of Y'(s') using secants} and \eqref{eqn: def of f_4 and f_5},
\begin{equation}
\begin{split}
&\;4\pi[U^\varepsilon_Y(s)-U_Y(s)]\\
=&\;\mathrm{p.v.}\int_{\mathbb{T}}S\cdot \frac{B^2}{D^2}\cdot \frac{Y(s+\tau)-Y(s)}{D^2} \cdot \left[f_2\left(\frac{D}{\varepsilon}\right)-2f_3\left(\frac{D}{\varepsilon}\right)\right]\,d\tau\\
&\;+\mathrm{p.v.}\int_{\mathbb{T}}S \left[\frac{A^2}{D^4}(Y(s+\tau)-Y(s))+\frac{AB}{D^4}(Y(s+\tau)-Y(s))^\perp \right]\cdot f_3\left(\frac{D}{\varepsilon}\right)\,d\tau\\
&\;+\mathrm{p.v.}\int_{\mathbb{T}}S \left[\frac{AB}{D^4}(Y(s+\tau)-Y(s))^\perp - \frac{B^2}{D^4}(Y(s+\tau)-Y(s))\right] \cdot \left[f_2\left(\frac{D}{\varepsilon}\right)-f_3\left(\frac{D}{\varepsilon}\right)\right]\,d\tau\\
=&\;\mathrm{p.v.}\int_{\mathbb{T}}S\cdot \frac{B^2}{D^2}\cdot \frac{Y(s+\tau)-Y(s)}{D^2} \cdot f_5\left(\frac{D}{\varepsilon}\right)\,d\tau\\
&\;+\mathrm{p.v.}\int_{\mathbb{T}} \frac{A}{D^2}\cdot SY'(s+\tau)\cdot f_3\left(\frac{D}{\varepsilon}\right)\,d\tau\\
&\;+\mathrm{p.v.}\int_{\mathbb{T}}\frac{SB}{D^2}\cdot Y'(s+\tau)^\perp\cdot  f_4\left(\frac{D}{\varepsilon}\right)\,d\tau\\
\triangleq &\;E_1(s)+E_2(s)+E_3(s).
\end{split}
\label{eqn: splitting the string velocity error into E_k}
\end{equation}
\end{step}

\begin{step}[Estimate for $E_1$]
By \eqref{eqn: C0 and C1 bound for S}, Lemma \ref{lemma: decay estimate of f_1 f_2 f_3} and Lemma \ref{lemma: preliminary estimates as building blocks},
\begin{equation}\label{eqn: pointwise bound for E_1}
\begin{split}
|E_1(s)|\leq &\; C\int_{\mathbb{T}}\mu\cdot \left(\frac{\tau^2 QR}{\lambda|\tau|}\right)^2\cdot \frac{1}{\lambda|\tau|}\cdot \frac{1}{1+\left(\frac{\lambda\tau}{\varepsilon}\right)^2}\,d\tau\\
\leq &\; \frac{C\mu Q^2}{\lambda^3}\int_{\mathbb{T}}R^2\cdot \frac{|\tau|}{1+\left(\frac{\lambda\tau}{\varepsilon}\right)^2}\,d\tau.
\end{split}
\end{equation}
By Minkowski inequality,
\begin{equation}\label{eqn: L^2 estimate of E_1}
\begin{split}
\|E_1\|_{L^2(\mathbb{T})}
\leq &\; \frac{C\mu Q^2}{\lambda^3}\int_{\mathbb{T}}\|R^2\|_{L^2_s(\mathbb{T})}\cdot \frac{|\tau|}{1+\left(\frac{\lambda\tau}{\varepsilon}\right)^2}\,d\tau\\
\leq &\;\frac{C\mu\varepsilon^2 \ln(\lambda/\varepsilon)}{\lambda^5}\|Y'\|_{L^\infty(\mathbb{T})}^2\|Y''\|^2_{L^4(\mathbb{T})}.
\end{split}
\end{equation}
\end{step}

\begin{step}[Partial estimate for $E_2$]
We further split $E_2(s)$.
\begin{equation}
\begin{split}
E_2(s)
= &\;SY'(s)\cdot \mathrm{p.v.}\int_{\mathbb{T}} \frac{A}{D^2}\cdot f_3\left(\frac{D}{\varepsilon}\right)\,d\tau\\
&\;+\mathrm{p.v.}\int_{\mathbb{T}}\left[\frac{A}{D^2}-\frac{1}{\tau}\right]\cdot [SY'(s+\tau)-SY'(s)]\cdot f_3\left(\frac{D}{\varepsilon}\right)\,d\tau\\
&\;+\mathrm{p.v.}\int_{\mathbb{T}} \frac{1}{\tau}\cdot [SY'(s+\tau)-SY'(s)]\cdot \left[f_3\left(\frac{D}{\varepsilon}\right)-f_3\left(\frac{|Y'(s)||\tau|}{\varepsilon}\right)\right]\,d\tau\\
&\;+\mathrm{p.v.}\int_{\mathbb{T}} \frac{SY'(s+\tau)-SY'(s)}{\tau}\cdot f_3\left(\frac{|Y'(s)||\tau|}{\varepsilon}\right)\,d\tau\\
\triangleq &\; E_{2,1}(s)+E_{2,2}(s)+E_{2,3}(s)+E_{2,4}(s).
\end{split}
\label{eqn: splitting E_2}
\end{equation}
It would be clear later that $E_{2,4}$ accounts for the most singular part in $E_2$.

We claim that $E_{2,1} = 0$.
Indeed, since $\frac{dD}{d\tau} = \frac{A}{D}$,
\begin{equation}
\begin{split}
E_{2,1}(s) = &\;SY'(s)\cdot \left[\lim_{\eta \rightarrow 0^+} \int_{\eta}^{\pi} \frac{A}{D^2}\cdot f_3\left(\frac{D}{\varepsilon}\right) \,d\tau +\int_{-\pi}^{-\eta} \frac{A}{D^2}\cdot f_3\left(\frac{D}{\varepsilon}\right) \,d\tau\right]\\
= &\;SY'(s)\cdot \lim_{\eta \rightarrow 0^+}\left[ \int_{D(s,s+\eta)}^{D(s,s+\pi)} \frac{1}{D}\cdot f_3\left(\frac{D}{\varepsilon}\right) \,dD -\int_{D(s, s-\eta)}^{D(s,s-\pi)} \frac{1}{D}\cdot f_3\left(\frac{D}{\varepsilon}\right) \,dD\right]\\
= &\;SY'(s)\cdot \left[ \int_{D(s,s-\pi)}^{D(s,s+\pi)} \frac{1}{D}\cdot f_3\left(\frac{D}{\varepsilon}\right) \,dD -\lim_{\eta \rightarrow 0^+}\int_{D(s, s-\eta)}^{D(s,s+\eta)} \frac{1}{D}\cdot f_3\left(\frac{D}{\varepsilon}\right) \,dD\right].
\end{split}
\end{equation}
The first term is zero since $D(s,s+\pi)=D(s,s-\pi)$.
For the second term, 
\begin{equation*}
\begin{split}
&\;|D(s,s+\eta)-D(s,s-\eta)|\\
\leq &\; |(Y(s+\eta)-Y(s))-(Y(s)-Y(s-\eta))|\\
\leq &\; \left|\int_0^\eta Y'(s+\omega)-Y'(s+\omega-\eta)\,d\omega\right|\\
\leq &\; C|\eta|^{3/2}\|Y'\|_{C^{1/2}(\mathbb{T})}\leq C|\eta|^{3/2}\|Y\|_{\dot{H}^2(\mathbb{T})}.
\end{split}
\end{equation*}
Since $D(s,s\pm\eta)\geq \lambda |\eta|$ by \eqref{eqn: well-stretched condition}, and $|f_3|\leq C$ by Lemma \ref{lemma: decay estimate of f_1 f_2 f_3}, we have that
\begin{equation}
|E_{2,1}(s)| \leq \mu|Y'(s)|\cdot \lim_{\eta \rightarrow 0^+}C|\eta|^{3/2}\|Y\|_{\dot{H}^{2}(\mathbb{T})}\cdot \frac{1}{\lambda |\eta|}\cdot 1 = 0.
\label{eqn: estimate of E_21}
\end{equation}

For $E_{2,2}$, by Lemma \ref{lemma: decay estimate of f_1 f_2 f_3} and Lemma \ref{lemma: preliminary estimates as building blocks},
\begin{equation}
\begin{split}
|E_{2,2}(s)|\leq &\;C\int_{\mathbb{T}} \frac{|D^2 - \tau A|}{D^2 }\cdot \frac{|SY'(s+\tau)-SY'(s)|}{|\tau|} \cdot\frac{1}{1+\left(\frac{\lambda \tau}{\varepsilon}\right)^2 }\,d\tau\\
\leq &\;C\int_{\mathbb{T}} \frac{|D||Y(s+\tau)-Y(s) - \tau Y'(s+\tau)|}{D^2}\cdot \mu (Q+R) \cdot\frac{1}{1+\left(\frac{\lambda \tau}{\varepsilon}\right)^2 }\,d\tau\\
\leq &\;\frac{C\mu}{\lambda}\int_{\mathbb{T}} \frac{|\tau|R(Q+R)}{1+\left(\frac{\lambda \tau}{\varepsilon}\right)^2 }\,d\tau.
\end{split}
\label{eqn: estimate of E_22}
\end{equation}

For $E_{2,3}$, by Lemma \ref{lemma: preliminary estimates as building blocks} and the mean value theorem,
\begin{equation}
\begin{split}
|E_{2,3}(s)|\leq &\;C\int_{\mathbb{T}}\mu (Q+R)\cdot \left|f_3\left(\frac{D(s,s+\tau)}{\varepsilon}\right)-f_3\left(\frac{|Y'(s)||\tau|}{\varepsilon}\right)\right|\,d\tau\\
\leq &\;C\mu\int_{\mathbb{T}}(Q+R)\cdot\frac{|D(s,s+\tau)-|\tau||Y'(s)||}{\varepsilon}\left|f_3'\left(\frac{\xi(s,\tau)}{\varepsilon}\right)\right|\,d\tau,
\end{split}
\label{eqn: estimate of E_23 mean value theorem}
\end{equation}
where $\xi(s,\tau)\in \mathbb{R}$ is between $D(s,s+\tau)$ and $|\tau||Y'(s)|$.
It is clear that $|\xi(s,\tau)|\geq \lambda|\tau|$.
Again by Lemma \ref{lemma: decay estimate of f_1 f_2 f_3} and Lemma \ref{lemma: preliminary estimates as building blocks},
\begin{equation}
|E_{2,3}(s)|
\leq C\mu\int_{\mathbb{T}}(Q+R)\cdot\frac{\tau^2R}{\varepsilon}\cdot \frac{1}{1+\left(\frac{\lambda|\tau|}{\varepsilon}\right)^3}\,d\tau\leq \frac{C\mu}{\lambda}\int_{\mathbb{T}} \frac{|\tau|R^2}{1+\left(\frac{\lambda \tau}{\varepsilon}\right)^2 }\,d\tau.
\label{eqn: estimate of E_23}
\end{equation}
Here we used the fact that
\begin{equation*}
\frac{|x|^n}{1+|x|^{n+2}}\leq \frac{C}{1+x^2},\quad \forall\, n\in\mathbb{N}.
\end{equation*}

In order to bound $E_{2,4}$, we recall that $\mathcal{P}_K$ is defined in \eqref{eqn: projection operator}.
For convenience, define $\mathcal{Q}_K = Id-\mathcal{P}_K$; both $\mathcal{P}_K$ and $\mathcal{Q}_K$ commute with differentiation.
With $K\in\mathbb{N}_+$ to be chosen,
\begin{equation}
\begin{split}
&\;\left|E_{2,4}(s)-\mathcal{P}_K[(SY')'](s) \int_{\mathbb{T}} f_3\left(\frac{|Y'(s)||\tau|}{\varepsilon}\right)\,d\tau\right|\\
\leq &\;\left|\mathrm{p.v.}\int_{\mathbb{T}}\left(\frac{\mathcal{P}_K [SY'](s+\tau)-\mathcal{P}_K [SY'](s)}{\tau}-\mathcal{P}_K [(SY')'](s)\right) \cdot f_3\left(\frac{|Y'(s)||\tau|}{\varepsilon}\right)\,d\tau\right|\\
&\;+\left|\mathrm{p.v.}\int_{\mathbb{T}}\frac{\mathcal{Q}_K[SY'](s+\tau)-\mathcal{Q}_K [SY'](s)}{\tau}\cdot f_3\left(\frac{|Y'(s)||\tau|}{\varepsilon}\right)\,d\tau\right|\\
\leq &\;C\int_{\mathbb{T}}[|\tau||\mathcal{M}(\mathcal{P}_K [SY']'')(s)|+|\mathcal{M}(\mathcal{Q}_K [SY']')(s)|]\cdot \frac{1}{1+\left(\frac{\lambda\tau}{\varepsilon}\right)^2}\,d\tau\\
\leq &\;\frac{C\varepsilon^2 \ln(\lambda/\varepsilon)}{\lambda^2}|\mathcal{M}(\mathcal{P}_K F'_Y)(s)|
+\frac{C\varepsilon}{\lambda}|\mathcal{M}(\mathcal{Q}_K F_Y)(s)|.
\end{split}
\label{eqn: estimate of E_24}
\end{equation}

Combining \eqref{eqn: estimate of E_21}, \eqref{eqn: estimate of E_22}, \eqref{eqn: estimate of E_23} and \eqref{eqn: estimate of E_24}, we apply Minkowski and Sobolev inequalities to find
\begin{equation}
\begin{split}
&\;\left\|E_2
-\mathcal{P}_K F_Y(s) \int_{\mathbb{T}} f_3\left(\frac{|Y'(s)||\tau|}{\varepsilon}\right)\,d\tau
\right\|_{L^2(\mathbb{T})}\\
\leq &\;\frac{C\mu}{\lambda}\int_{\mathbb{T}} \|R(Q+R)\|_{L_s^2(\mathbb{T})}\cdot\frac{|\tau|}{1+\left(\frac{\lambda \tau}{\varepsilon}\right)^2 }\,d\tau
+\frac{C\varepsilon^2 \ln(\lambda/\varepsilon)}{\lambda^2}\|\mathcal{P}_K F'_Y\|_{L^2(\mathbb{T})}+\frac{C\varepsilon}{\lambda}\|\mathcal{Q}_K F_Y\|_{L^2(\mathbb{T})}\\
\leq &\;\frac{C\mu\varepsilon^2\ln(\lambda/\varepsilon)}{\lambda^3 }\|Y''\|_{L^4(\mathbb{T})}^2+\left(\frac{C\varepsilon^2 \ln(\lambda/\varepsilon)}{\lambda^2}\cdot K^{1-\theta}+\frac{C\varepsilon}{\lambda}\cdot K^{-\theta}\right)\|F_Y\|_{\dot{H}^{\theta}(\mathbb{T})}.
\end{split}
\label{eqn: a partial estimate of E_2 involving K}
\end{equation}
By Lemma \ref{lemma: H^1/2 estimate of F}, and taking $K \sim \frac{\lambda}{\varepsilon\ln(\frac{\lambda}{\varepsilon})}$, 
\begin{equation}
\begin{split}
&\;\left\|E_2
-\mathcal{P}_K F_Y \int_{\mathbb{T}} f_3\left(\frac{|Y'(s)||\tau|}{\varepsilon}\right)\,d\tau
\right\|_{L^2(\mathbb{T})}\\
\leq &\;\frac{C\mu\varepsilon^2\ln(\lambda/\varepsilon)}{\lambda^3 }\|Y''\|_{L^4(\mathbb{T})}^2+\frac{C\mu\varepsilon^{1+\theta} \ln^\theta(\lambda/\varepsilon)}{\lambda^{1+\theta}}\|Y\|_{\dot{H}^{2+\theta}(\mathbb{T})}.
\end{split}
\label{eqn: L^2 estimate of E_2 minus an additional part}
\end{equation}
The extra term on the left hand side of \eqref{eqn: L^2 estimate of E_2 minus an additional part} will be handled after we bound $E_3$.

\end{step}

\begin{step}[Partial estimate for $E_3$]

Again we further split $E_3(s)$.
Thanks to \eqref{eqn: improved formula for B}, 
\begin{equation}
\begin{split}
E_3(s) 
=&\;
\mathrm{p.v.}\int_{\mathbb{T}}\frac{SB}{D^2}\cdot (Y'(s+\tau)-Y'(s))^\perp\cdot f_4\left(\frac{D}{\varepsilon}\right)\,d\tau\\
&\;+\mathrm{p.v.}\int_{\mathbb{T}}\frac{SB}{D^2}\cdot Y'(s)^\perp  \left[f_4\left(\frac{D}{\varepsilon}\right)-f_4\left(\frac{|Y'(s)||\tau|}{\varepsilon}\right)\right]\,d\tau\\
&\;+\mathrm{p.v.}\int_{\mathbb{T}}SB\cdot \left(\frac{1}{D^2} -\frac{1}{\tau^2 |Y'(s)|^2}\right)\cdot Y'(s)^\perp  \cdot  f_4\left(\frac{|Y'(s)||\tau|}{\varepsilon}\right)\,d\tau\\
&\;+\mathrm{p.v.}\int_{\mathbb{T}}S(|Y'(s')|,s')\cdot (Y(s')-Y(s)-\tau Y'(s'))^\perp\cdot (Y'(s')-Y'(s)) \\
&\;\qquad\qquad \cdot \frac{Y'(s)^\perp}{\tau^2 |Y'(s)|^2} f_4\left(\frac{|Y'(s)||\tau|}{\varepsilon}\right)\,d\tau\\
&\;+\mathrm{p.v.}\int_{\mathbb{T}} \left[\left(\int_{0}^\tau S(|Y'(s+\tau)|,s+\tau)Y'(s+\eta)-SY'(s+\eta)\,d\eta\right)^\perp\cdot Y'(s)\right]\\
&\;\qquad\qquad \cdot \frac{Y'(s)^\perp}{\tau^2 |Y'(s)|^2} f_4\left(\frac{|Y'(s)||\tau|}{\varepsilon}\right)\,d\tau\\
&\;+\mathrm{p.v.}\int_{\mathbb{T}} \left[\left(\int_{0}^\tau SY'(s+\eta)-S Y'(s+\tau)\,d\eta\right)^\perp\cdot Y'(s)\right]\\
&\;\qquad\qquad \cdot \frac{Y'(s)^\perp}{\tau^2 |Y'(s)|^2} f_4\left(\frac{|Y'(s)||\tau|}{\varepsilon}\right)\,d\tau\\
\triangleq &\; \sum_{i = 1}^6 E_{3,i}(s).
\end{split}
\label{eqn: splitting E_3}
\end{equation}
In $E_{3,5}$ and $E_{3,6}$, we used the identity that
\begin{equation}
\begin{split}
&\;S(|Y'(s')|,s')\cdot [Y(s')-Y(s)-\tau Y'(s'))]\\
= &\; \int_{0}^\tau [S(|Y'(s')|,s')-S(|Y'(s+\eta)|,s+\eta)]Y'(s+\eta)\,d\eta \\
&\;+ \int_0^\tau  SY'(s+\eta)- SY'(s')\,d\eta.
\end{split}
\end{equation}

We bound $E_{3,i}$ one by one.
By \eqref{eqn: C0 and C1 bound for S}, Lemma \ref{lemma: decay estimate of f_1 f_2 f_3} and Lemma \ref{lemma: preliminary estimates as building blocks}, $E_{3,1}$, $E_{3,2}$, $E_{3,3}$ and $E_{3,4}$ can be bounded as follows.
\begin{equation}
\begin{split}
|E_{3,1}(s)|\leq &\;C\mu\int_{\mathbb{T}} \frac{\tau^2QR}{\lambda^2\tau^2}\cdot|\tau|R\cdot \frac{1}{1+\left(\frac{\lambda \tau}{\varepsilon}\right)^2}\,d\tau\\
\leq &\;\frac{C\mu}{\lambda^2}\int_{\mathbb{T}} QR^2\cdot \frac{|\tau|}{1+\left(\frac{\lambda \tau}{\varepsilon}\right)^2}\,d\tau.
\end{split}
\label{eqn: estimate of E_31}
\end{equation}
\begin{equation}
\begin{split}
|E_{3,2}(s)|\leq &\;C\mu\int_{\mathbb{T}} \frac{\tau^2 QR}{\lambda^2\tau^2}\cdot Q\cdot \frac{\tau^2 R}{\varepsilon}\cdot \frac{1}{1+\left(\frac{\lambda |\tau|}{\varepsilon}\right)^3}\,d\tau\\
\leq &\;\frac{C\mu}{\lambda^3}\int_{\mathbb{T}} Q^2R^2\cdot \frac{|\tau|}{1+\left(\frac{\lambda \tau}{\varepsilon}\right)^2}\,d\tau. 
\end{split}
\label{eqn: estimate of E_32}
\end{equation}
\begin{equation}
\begin{split}
|E_{3,3}(s)|\leq &\;C\mu\int_{\mathbb{T}}\tau^2 QR\cdot
\frac{(|D|+|\tau||Y'(s)|)|D-|\tau|| Y'(s)||}{|D|^2 \tau^2 |Y'(s)|}\cdot \frac{1}{1+\left(\frac{\lambda \tau}{\varepsilon}\right)^2}\,d\tau\\
\leq &\;C\mu\int_{\mathbb{T}}\tau^2 QR\cdot
\frac{\tau^2 R}{\lambda^2|\tau|^3}\cdot \frac{1}{1+\left(\frac{\lambda \tau}{\varepsilon}\right)^2}\,d\tau\\
\leq &\;\frac{C\mu}{\lambda^2}\int_{\mathbb{T}}QR^2\cdot \frac{|\tau|}{1+\left(\frac{\lambda \tau}{\varepsilon}\right)^2}\,d\tau.
\end{split}
\label{eqn: estimate of E_33}
\end{equation}
\begin{equation}
|E_{3,4}(s)|\leq C\mu\int_{\mathbb{T}}|\tau|^3 R^2\cdot \frac{1}{\tau^2 \lambda}\cdot \frac{1}{1+\left(\frac{\lambda \tau}{\varepsilon}\right)^2}\,d\tau\leq \frac{C\mu}{\lambda}\int_{\mathbb{T}}R^2\cdot \frac{|\tau|}{1+\left(\frac{\lambda \tau}{\varepsilon}\right)^2}\,d\tau.
\label{eqn: estimate of E_34}
\end{equation}
Note that for $E_{3,2}$, we applied the mean value theorem and proceeded as in \eqref{eqn: estimate of E_23 mean value theorem} and \eqref{eqn: estimate of E_23}. 

To handle $E_{3,5}$, we purposefully put an extra $Y'(s)$ into the integral without changing its value, i.e.,
\begin{equation}
\begin{split}
&\;E_{3,5}(s)\\
= &\;\mathrm{p.v.}\int_{\mathbb{T}} \left(\int_{0}^\tau [S(|Y'(s+\tau)|,s+\tau)-S(|Y'(s+\eta)|,s+\eta)]\cdot (Y'(s+\eta)-Y'(s))\,d\eta\right)^\perp\\
&\;\qquad\qquad\cdot Y'(s)\cdot \frac{Y'(s)^\perp}{\tau^2 |Y'(s)|^2}\cdot f_4\left(\frac{|Y'(s)||\tau|}{\varepsilon}\right)\,d\tau.
\end{split}
\end{equation}
Since
\begin{equation}
|Y'(s+\eta)-Y'(s)|\leq |Y'(s+\eta)-Y'(s+\tau)| + |Y'(s+\tau)-Y'(s)|\leq C|\tau||\mathcal{M}Y''(s+\tau)|,
\end{equation}
by Lemma \ref{lemma: decay estimate of f_1 f_2 f_3} and Assumption \ref{assumption: C 1,1 regularity assumption}, 
\begin{equation}
\begin{split}
&\;|E_{3,5}(s)| \\
\leq
&\;C\int_{\mathbb{T}}\left[\int_{0}^{|\tau|} \left(\mu |\tau|+\frac{\mu|\tau||\mathcal{M}Y''(s+\tau)|}{\|Y'\|_{L^{\infty}(\mathbb{T})}}\right)\cdot|\tau||\mathcal{M}Y''(s+\tau)|\,d\eta\right]\cdot \frac{1}{\tau^2}\cdot \frac{1}{1+\left(\frac{\lambda \tau}{\varepsilon}\right)^2}\,d\tau\\
\leq &\;\frac{C\mu}{\lambda}\int_{\mathbb{T}} (Q+R)R\cdot\frac{|\tau|}{1+\left(\frac{\lambda \tau}{\varepsilon}\right)^2}\,d\tau. 
\end{split}
\label{eqn: estimate of E_35}
\end{equation}
Finally, for $E_{3,6}$, we take the same strategy as in estimating $E_{2,4}$,
choosing the same $K$ as before. 
Indeed,
\begin{equation}
\begin{split}
E_{3,6}(s) = &\;\mathrm{p.v.}\int_{\mathbb{T}} \left[\left(\int_{0}^\tau \mathcal{P}_K [SY'](s+\eta)-\mathcal{P}_K [SY'](s+\tau)\,d\eta +\frac{\tau^2}{2}\mathcal{P}_K[(SY')'](s)\right)^\perp\cdot Y'(s)\right]\\
&\;\qquad\qquad\cdot \frac{Y'(s)^\perp}{\tau^2 |Y'(s)|^2}f_4\left(\frac{|Y'(s)||\tau|}{\varepsilon}\right)\,d\tau\\
&\;+\frac{1}{2}\mathcal{P}_K[(SY')'](s)\cdot Y'(s)^\perp \frac{Y'(s)^\perp}{|Y'(s)|^2} \int_{\mathbb{T}} f_4\left(\frac{|Y'(s)||\tau|}{\varepsilon}\right)\,d\tau\\
&\;+\mathrm{p.v.}\int_{\mathbb{T}} \left[\left(\int_{0}^\tau \mathcal{Q}_K [SY'](s+\eta)-\mathcal{Q}_K [SY'](s+\tau)\,d\eta\right)^\perp\cdot Y'(s)\right]\\
&\;\qquad\qquad\cdot \frac{Y'(s)^\perp}{\tau^2 |Y'(s)|^2}f_4\left(\frac{|Y'(s)||\tau|}{\varepsilon}\right)\,d\tau.
\end{split}
\label{eqn: splitting E_36}
\end{equation}
We derive that
\begin{equation}
\begin{split}
&\;\left|\int_{0}^\tau \mathcal{P}_K [SY'](s+\eta)-\mathcal{P}_K [SY'](s+\tau)\,d\eta +\frac{\tau^2}{2}\mathcal{P}_K[(SY')'](s)\right|\\
=&\;\left|\int_{0}^\tau \int_\eta ^\tau \mathcal{P}_K[(SY')'](s) - \mathcal{P}_K [(SY')'](s+\zeta)\,d\zeta d\eta\right|\\
\leq &\;
C|\tau|^3|\mathcal{M}(\mathcal{P}_K F'_Y)(s)|.
\end{split}
\end{equation}
Combine this with \eqref{eqn: splitting E_36} and we find that
\begin{equation}
\begin{split}
&\;\left|E_{3,6}(s)-\frac{1}{2}\mathcal{P}_KF_Y(s)\cdot Y'(s)^\perp \frac{Y'(s)^\perp}{|Y'(s)|^2} \int_{\mathbb{T}} f_4\left(\frac{|Y'(s)||\tau|}{\varepsilon}\right)\,d\tau\right|\\
\leq &\;C\int_{\mathbb{T}}|\tau|^3|\mathcal{M}(\mathcal{P}_KF'_Y)(s)|
\cdot\frac{1}{\tau^2}\cdot \frac{1}{1+\left(\frac{\lambda \tau}{\varepsilon}\right)^2}\,d\tau+C\int_{\mathbb{T}}|\tau|^2|\mathcal{M}(\mathcal{Q}_K F_Y)(s)|
\cdot \frac{1}{\tau^2}\cdot \frac{1}{1+\left(\frac{\lambda \tau}{\varepsilon}\right)^2}\,d\tau\\
\leq &\;\frac{C\varepsilon^2 \ln(\lambda/\varepsilon)}{\lambda^2}|\mathcal{M}(\mathcal{P}_KF'_Y)(s)|+\frac{C\varepsilon }{\lambda}|\mathcal{M}(\mathcal{Q}_KF_Y)(s)|.
\end{split}
\label{eqn: estimate of E_36}
\end{equation}
Combining \eqref{eqn: estimate of E_31}-\eqref{eqn: estimate of E_34}, \eqref{eqn: estimate of E_35}, and \eqref{eqn: estimate of E_36}, we argue as in \eqref{eqn: a partial estimate of E_2 involving K} and \eqref{eqn: L^2 estimate of E_2 minus an additional part} to obtain that
\begin{equation}\label{eqn: L^2 estimate of E_3 minus an additional part}
\begin{split}
&\;\left\|E_3(s)-\frac{1}{2}\mathcal{P}_K F_Y(s)\cdot Y'(s)^\perp\cdot \frac{Y'(s)^\perp}{|Y'(s)|^2}\int_{\mathbb{T}} f_4\left(\frac{|Y'(s)||\tau|}{\varepsilon}\right)\,d\tau\right\|_{L^2(\mathbb{T})}\\
\leq &\;\frac{C\mu\varepsilon^2 \ln(\lambda/\varepsilon)}{\lambda^5}\|Y'\|_{L^\infty(\mathbb{T})}^2 \|Y''\|^2_{L^4(\mathbb{T})}+\frac{C\varepsilon^2 \ln(\lambda/\varepsilon)}{\lambda^2}K^{1-\theta}\|\mathcal{P}_KF_Y\|_{\dot{H}^{\theta}(\mathbb{T})}\\
&\;+\frac{C\varepsilon }{\lambda}K^{-\theta}\|\mathcal{Q}_KF_Y\|_{\dot{H}^{\theta}(\mathbb{T})}\\
\leq &\;\frac{C\varepsilon^2 \ln(\lambda/\varepsilon)}{\lambda^5}\|Y'\|_{L^\infty(\mathbb{T})}^2 \|Y''\|^2_{L^4(\mathbb{T})}+\frac{C\mu\varepsilon^{1+\theta} \ln^\theta(\lambda/\varepsilon)}{\lambda^{1+\theta}}\|Y\|_{\dot{H}^{2+\theta}(\mathbb{T})}.
\end{split}
\end{equation}
\end{step}

\begin{step}[Estimates of the extra terms]
Combining \eqref{eqn: L^2 estimate of E_1}, \eqref{eqn: L^2 estimate of E_2 minus an additional part} and \eqref{eqn: L^2 estimate of E_3 minus an additional part},
\begin{equation}
\begin{split}
&\;\left\|4\pi[U^\varepsilon_Y(s)-U_Y(s)]-\mathcal{P}_K F_Y(s)\cdot \int_{\mathbb{T}} f_3\left(\frac{|Y'(s)||\tau|}{\varepsilon}\right)\,d\tau\right.\\
&\;\quad\left.-\frac{1}{2}\mathcal{P}_K F_Y(s)\cdot Y'(s)^\perp\cdot \frac{Y'(s)^\perp}{|Y'(s)|^2}\int_{\mathbb{T}} \left[f_2\left(\frac{|Y'(s)||\tau|}{\varepsilon}\right)-f_3\left(\frac{|Y'(s)||\tau|}{\varepsilon}\right)\right]\,d\tau\right\|_{L^2(\mathbb{T})} \\
\leq &\;\frac{C\mu\varepsilon^2 \ln(\lambda/\varepsilon)}{\lambda^5}\|Y'\|_{L^\infty(\mathbb{T})}^2 \|Y''\|^2_{L^4(\mathbb{T})}+\frac{C\mu\varepsilon^{1+\theta} \ln^\theta(\lambda/\varepsilon)}{\lambda^{1+\theta}}\|Y\|_{\dot{H}^{2+\theta}(\mathbb{T})}
\end{split}
\label{eqn: L^2 error estimates with the extra terms}
\end{equation}

To this end, we shall handle the extra terms on the left hand side.
By Lemma \ref{lemma: integrals of f_2 and f_3},
\begin{equation}
\begin{split}
\int_{\mathbb{T}}f_3\left(\frac{|Y'(s)||\tau|}{\varepsilon}\right)\,d\tau 
=&\; \frac{\varepsilon}{|Y'(s)|}\int_{\mathbb{R}}f_3(\omega)\,d\omega - \frac{\varepsilon}{|Y'(s)|}\int_{|\omega|>|Y'(s)|\pi/\varepsilon}f_3(\omega)\,d\omega\\
=&\; -\frac{4 m_1\varepsilon}{|Y'(s)|}  - \frac{\varepsilon}{|Y'(s)|}\int_{|\omega|>|Y'(s)|\pi/\varepsilon}f_3(\omega)\,d\omega,
\end{split}
\end{equation}
where
\begin{equation}
\left|\frac{\varepsilon}{|Y'(s)|}\int_{|\omega|>|Y'(s)|\pi/\varepsilon}f_3(\omega)\,d\omega\right| \leq \frac{C\varepsilon}{|Y'(s)|}\int_{|Y'(s)|\pi/\varepsilon}^\infty\frac{1}{1+\omega^2}\,d\omega
\leq \frac{C\varepsilon^2 }{\lambda ^2}.
\label{eqn: bound of the tail in the O(epsilon) term}
\end{equation}
Similarly, 
\begin{equation}
\int_{\mathbb{T}} f_4\left(\frac{|Y'(s)||\tau|}{\varepsilon}\right)\,d\tau= \frac{8 m_1\varepsilon}{|Y'(s)|}  - \frac{\varepsilon}{|Y'(s)|}\int_{|\omega|>|Y'(s)|\pi/\varepsilon}[f_2(\omega )- f_3(\omega)]\,d\omega.
\label{eqn: calculation of f_4 eps integral on torus}
\end{equation}
where the second term can be bounded as in \eqref{eqn: bound of the tail in the O(epsilon) term}.
Hence, we combine \eqref{eqn: L^2 error estimates with the extra terms}-\eqref{eqn: calculation of f_4 eps integral on torus} to obtain that
\begin{equation}
\begin{split}
&\;\left\|4\pi[U^\varepsilon_Y(s)-U_Y(s)]+\left(\mathcal{P}_K F_Y(s)\cdot \frac{4 m_1 \varepsilon}{|Y'(s)|}-\frac{1}{2}\mathcal{P}_K F_Y(s)\cdot Y'(s)^\perp\cdot \frac{Y'(s)^\perp}{|Y'(s)|^2}\cdot \frac{8 m_1 \varepsilon}{|Y'(s)|}\right)\right\|_{L^2(\mathbb{T})} \\
\leq &\;\frac{C\mu\varepsilon^2 \ln(\lambda/\varepsilon)}{\lambda^5}\|Y'\|_{L^\infty(\mathbb{T})}^2 \|Y''\|^2_{L^4(\mathbb{T})}+\frac{C\mu\varepsilon^{1+\theta} \ln^\theta(\lambda/\varepsilon)}{\lambda^{1+\theta}}\|Y\|_{\dot{H}^{2+\theta}(\mathbb{T})}.
\end{split}
\label{eqn: L^2 error estimate with P_K}
\end{equation}
Note that 
\begin{equation}
\begin{split}
&\;\mathcal{P}_K F_Y(s)\cdot \frac{4 m_1 \varepsilon}{|Y'(s)|}-\frac{1}{2}\mathcal{P}_K F_Y(s)\cdot Y'(s)^\perp\cdot \frac{Y'(s)^\perp}{|Y'(s)|^2}\cdot \frac{8 m_1 \varepsilon}{|Y'(s)|}\\
=&\;\frac{4 m_1 \varepsilon}{|Y'(s)|}\cdot \frac{\mathcal{P}_K F_Y(s)\cdot Y'(s)}{|Y'(s)|^2}Y'(s)\\
=&\;\frac{4 m_1 \varepsilon}{|Y'(s)|}\cdot \frac{F_Y(s)\cdot Y'(s)}{|Y'(s)|^2}Y'(s)-\frac{4 m_1 \varepsilon}{|Y'(s)|}\cdot \frac{\mathcal{Q}_K F_Y(s)\cdot Y'(s)}{|Y'(s)|^2}Y'(s).
\end{split}
\end{equation}
By Lemma \ref{lemma: H^1/2 estimate of F},
\begin{equation}
\left\|\frac{4 m_1 \varepsilon}{|Y'(s)|}\cdot \frac{\mathcal{Q}_K F_Y(s)\cdot Y'(s)}{|Y'(s)|^2}Y'(s)\right\|_{L^2(\mathbb{T})}\\
\leq \frac{C\varepsilon}{\lambda}\cdot K^{-\theta}\|F_Y\|_{\dot{H}^{\theta}(\mathbb{T})}\leq \frac{C\mu\varepsilon^{1+\theta}\ln^\theta(\lambda/\varepsilon)}{\lambda^{1+\theta}}\|Y\|_{\dot{H}^{2+\theta}(\mathbb{T})}.
\end{equation}
Therefore, \eqref{eqn: L^2 error estimate with P_K} implies that
\begin{equation}
\begin{split}
&\;\left\|[U^\varepsilon_Y(s)-U_Y(s)]+\frac{m_1 \varepsilon}{\pi|Y'(s)|}\cdot \frac{F_Y(s)\cdot Y'(s)}{|Y'(s)|^2}Y'(s)\right\|_{L^2(\mathbb{T})} \\
\leq &\;\frac{C\mu\varepsilon^2 \ln(\lambda/\varepsilon)}{\lambda^5}\|Y'\|_{L^\infty(\mathbb{T})}^2 \|Y''\|^2_{L^4(\mathbb{T})}+\frac{C\mu\varepsilon^{1+\theta} \ln^\theta(\lambda/\varepsilon)}{\lambda^{1+\theta}}\|Y\|_{\dot{H}^{2+\theta}(\mathbb{T})},
\end{split}
\label{eqn: L^2 error estimate with tangent force on the left hand side}
\end{equation}
which proves the desired estimate in both Proposition \ref{prop: L^2 error estimate of the string velocity} and Corollary \ref{coro: L^2 error estimate of the string normal velocity}.

Thanks to Lemma \ref{lemma: decay estimate of f_1 f_2 f_3}, if $m_2 = 0$, $f_3$, $f_4$, $f_5$, and their first derivatives will enjoy improved decay at $\infty$.
In this case, it is not difficult to verify that all the logarithmic factors in this proof can be removed.
\end{step}
\end{proof}

\begin{remark}\label{rmk: get rid of pv integral in the components of the error}
All the principal value integrals in the proof, except those for $E_{2,1}$, 
may be replaced by the usual integrals.
\end{remark}

\subsection{Proof of the static $H^1$-error estimate}\label{section: static H^1 error estimates}

\begin{proof}[Proof of Proposition \ref{prop: H^1 error estimate of the string velocity}]
Let $E_1(s)$, $E_2(s)$ and $E_3(s)$ be defined as in \eqref{eqn: splitting the string velocity error into E_k}.
\setcounter{step}{0}

\begin{step}[Estimate for $E_1'$]
Recall that 
\begin{equation}
\label{eqn: E_1 and definition of e_1}
E_1(s) = \int_{\mathbb{T}}SB^2\cdot \frac{Y(s+\tau)-Y(s)}{D^4} \cdot f_5\left(\frac{D}{\varepsilon}\right)\,d\tau.
\end{equation}
As mentioned in Remark \ref{rmk: get rid of pv integral in the components of the error}, we can get rid of the principal value integral.
Denote its integrand to be $e_1(s,\tau)$.

We first give an estimate of $\int_\mathbb{T}|\partial_s e_1(s,\tau)|\,d\tau$.
For clarity, we start from some simpler estimates.
By Lemma \ref{lemma: decay estimate of f_1 f_2 f_3} and Lemma \ref{lemma: preliminary estimates as building blocks}, 
\begin{equation}
|SB^2|\leq C\mu |\tau|^2 QR\cdot |D|Q,
\label{eqn: bound for e_1 part 1}
\end{equation}
and
\begin{equation}
\left|\frac{Y(s+\tau)-Y(s)}{D^4} \cdot f_5\left(\frac{D}{\varepsilon}\right)\right|\leq \frac{C}{|D|^3}\cdot \frac{1}{1+\left(\frac{\lambda\tau}{\varepsilon}\right)^2}.
\label{eqn: bound for e_1 part 2}
\end{equation}
Also by Lemma \ref{lemma: preliminary estimates of derivatives as building blocks},
\begin{equation}
\begin{split}
|\partial_s (SB^2)|\leq &\; |\partial_sS(|Y'(s+\tau)|,s+\tau)||B^2|\\
&\; +\left|\partial_pS(|Y'(s+\tau)|,s+\tau)\cdot \frac{Y''\cdot Y'(s+\tau)}{|Y'(s+\tau)|}\right||B^2|\\
&\; + 2|S||B||\partial_s B|\\
\leq &\; C\mu\cdot (\tau^2QR)^2+ C\frac{\mu R}{\|Y'\|_{L^\infty(\mathbb{T})}}\cdot\tau^2 QR\cdot |\tau|Q^2 + C\mu\cdot \tau^2QR\cdot|\tau|QR\\
\leq &\; C\mu|\tau|^3Q^2R^2.
\end{split}
\label{eqn: bound for e_1 part 1 derivative}
\end{equation}
Lastly,
\begin{equation}
\begin{split}
&\;\partial_s \left(\frac{Y(s+\tau)-Y(s)}{D^4} \cdot f_5\left(\frac{D}{\varepsilon}\right)\right)\\
\leq &\;\frac{Y'(s+\tau)-Y'(s)}{D^4} \cdot f_5\left(\frac{D}{\varepsilon}\right)+\frac{Y(s+\tau)-Y(s)}{D^4}\cdot \left(\frac{1}{\varepsilon}f_5'\left(\frac{D}{\varepsilon}\right)-\frac{4}{D}f_5\left(\frac{D}{\varepsilon}\right)\right)\partial_s D.
\end{split}
\label{eqn: calculate e_1 part 2 derivative}
\end{equation}
Hence,
\begin{equation}
\begin{split}
&\;\left|\partial_s \left(\frac{Y(s+\tau)-Y(s)}{D^4} \cdot f_5\left(\frac{D}{\varepsilon}\right)\right)\right|\\
\leq &\; \frac{C|\tau|R}{|D|^4}\cdot \frac{1}{1+\left(\frac{\lambda\tau}{\varepsilon}\right)^2}+\frac{C}{|D|^3} \left(\frac{1}{\varepsilon}\cdot \frac{1}{1+\left(\frac{\lambda |\tau|}{\varepsilon}\right)^3}+\frac{1}{|D|}\cdot\frac{1}{1+\left(\frac{\lambda\tau}{\varepsilon}\right)^2}\right)|\tau|R\\
\leq &\; \frac{CR}{\lambda|D|^3}\cdot \frac{1}{1+\left(\frac{\lambda\tau}{\varepsilon}\right)^2}.
\end{split}
\label{eqn: bound for e_1 part 2 derivative}
\end{equation}
Combining \eqref{eqn: bound for e_1 part 1}-\eqref{eqn: bound for e_1 part 2 derivative}, we find that
\begin{equation}
\begin{split}
&\;\int_{\mathbb{T}}|\partial_s e_1(s,\tau)|\,d\tau \\
\leq &\;\int_{\mathbb{T}} |\partial_s (SB^2)|\left|\frac{Y(s+\tau)-Y(s)}{D^4} \cdot f_5\left(\frac{D}{\varepsilon}\right)\right|+|SB^2|\left|\partial_s \left(\frac{Y(s+\tau)-Y(s)}{D^4} \cdot f_5\left(\frac{D}{\varepsilon}\right)\right)\right|\,d\tau\\
\leq &\;\frac{C\mu}{\lambda^3} Q^2\int_{\mathbb{T}} \frac{R^2}{1+\left(\frac{\lambda\tau}{\varepsilon}\right)^2}\,d\tau. 
\end{split}
\end{equation}
By Minkowski inequality and Sobolev inequality,
\begin{equation}
\left\|\int_{\mathbb{T}}|\partial_s e_1(s,\tau)|\,d\tau \right\|_{L^2(\mathbb{T})}\leq \frac{C\mu\varepsilon}{\lambda^4}\|Y'\|_{L^\infty(\mathbb{T})}^2 \|Y''\|_{L^4(\mathbb{T})}^2.
\label{eqn: L^2 estimate of integral of e_1'}
\end{equation}

To this end, we claim that $E_1'(s) = \int_\mathbb{T}\partial_s e_1(s,\tau)\,d\tau$.
Indeed, it is clear from the estimate \eqref{eqn: pointwise bound for E_1} of $E_1(s)$ that $e_1\in L^1(\mathbb{T}\times \mathbb{T})$; so is $\partial_s e_1$ by \eqref{eqn: L^2 estimate of integral of e_1'}.
Hence, if we take an arbitrary $\psi(s)\in C^\infty(\mathbb{T})$, by Fubini's Theorem and integration by parts,
\begin{equation}
\begin{split}
\int_{\mathbb{T}}ds\,\psi'(s)E_1(s) = \int_{\mathbb{T}}d\tau\,\int_{\mathbb{T}}-\psi(s)\partial_s e_1(s,\tau)\,ds = \int_{\mathbb{T}}-\psi(s)\int_{\mathbb{T}}\partial_s e_1(s,\tau)\,d\tau ds.
\end{split}
\label{eqn: integral of e_1' is the derivative of E_1}
\end{equation}
This proves the claim, and hence \eqref{eqn: L^2 estimate of integral of e_1'} is also a bound for $\|E_1'(s)\|_{L^2(\mathbb{T})}$.
\end{step}

\begin{step}[Estimate for $E_2'$]
By the proof in Section \ref{section: L^2 static error},
\begin{equation}
\label{eqn: E_2 and definition of e_2}
E_2(s) = \int_{\mathbb{T}} \frac{SY'(s+\tau)-SY'(s)}{\tau} \cdot \frac{\tau A}{D^2}\cdot f_3\left(\frac{D}{\varepsilon}\right)\,d\tau.
\end{equation}
Denote its integrand to be $e_2(s,\tau)$.
We shall bound $\int_{\mathbb{T}}|\partial_s e_{2}|\,d\tau$ first.
Again, we derive some simple estimates.

Aiming at a sharper estimate of the leading term (see \eqref{eqn: bound of e_2'} below), we purposefully split the term
\begin{equation}
\frac{\tau A}{D^2}\cdot f_3\left(\frac{D}{\varepsilon}\right) = \left[\frac{\tau A}{D^2}\cdot f_3\left(\frac{D}{\varepsilon}\right)-f_3\left(\frac{|Y'(s)|\tau}{\varepsilon}\right)\right]+f_3\left(\frac{|Y'(s)|\tau}{\varepsilon}\right),
\end{equation}
where
\begin{equation}
\begin{split}
&\;\left|\frac{\tau A}{D^2}\cdot f_3\left(\frac{D}{\varepsilon}\right)-f_3\left(\frac{|Y'(s)|\tau}{\varepsilon}\right)\right|\\
\leq
&\;\left|\frac{\tau A-D^2}{D^2}\cdot f_3\left(\frac{D}{\varepsilon}\right)\right| +\left|f_3\left(\frac{D}{\varepsilon}\right)-f_3\left(\frac{|Y'(s)|\tau}{\varepsilon}\right)\right|\\
\leq &\;\frac{CD |Y(s+\tau)-Y(s)-\tau Y'(s+\tau)|}{D^2}\cdot \frac{1}{1+\left(\frac{\lambda\tau}{\varepsilon}\right)^2}\\
&\;+\frac{C|Y(s+\tau)-Y(s)-\tau Y'(s)|}{\varepsilon}\cdot \frac{1}{1+\left(\frac{\lambda|\tau|}{\varepsilon}\right)^3}\\
\leq &\;\frac{C\tau^2 R}{\lambda |\tau|}\cdot \frac{1}{1+\left(\frac{\lambda\tau}{\varepsilon}\right)^2}+\frac{C\tau^2 R}{\varepsilon}\cdot \frac{1}{1+\left(\frac{\lambda|\tau|}{\varepsilon}\right)^3}\\
\leq &\;\frac{C|\tau| R}{\lambda}\cdot \frac{1}{1+\left(\frac{\lambda\tau}{\varepsilon}\right)^2}.
\end{split}
\label{eqn: bound for e_2 part 2}
\end{equation}
By Lemma \ref{lemma: H^1/2 estimate of F},
\begin{equation}
\left|\partial_s\left(\frac{SY'(s+\tau)-SY'(s)}{\tau} \right)\right| \leq \frac{|F(s+\tau)|+|F(s)|}{\tau}\leq\frac{C\mu}{|\tau|}(Q+R+|Y''(s)|).
\label{eqn: bound for e_2 part 1 derivative}
\end{equation}
Lastly,
\begin{equation}\label{eqn: calculate e_2 part 2 derivative}
\begin{split}
&\;\partial_s\left(\frac{\tau A}{D^2}\cdot f_3\left(\frac{D}{\varepsilon}\right)\right)\\
= &\;\frac{\tau \partial_s A}{D^2}\cdot f_3\left(\frac{D}{\varepsilon}\right) +\frac{\tau A}{D^2} \left(\frac{1}{\varepsilon}f_3'\left(\frac{D}{\varepsilon}\right)-\frac{2}{D}f_3\left(\frac{D}{\varepsilon}\right)\right)\partial_s D.
\end{split}
\end{equation}
By Lemma \ref{lemma: decay estimate of f_1 f_2 f_3}, Lemma \ref{lemma: preliminary estimates as building blocks}, and Lemma \ref{lemma: preliminary estimates of derivatives as building blocks},
\begin{equation}\label{eqn: bound for e_2 part 2 derivative}
\begin{split}
&\;\left|\partial_s\left(\frac{\tau A}{D^2}\cdot f_3\left(\frac{D}{\varepsilon}\right)\right)\right|\\
\leq &\;\frac{C|\tau|^2 QR}{\lambda^2\tau^2}\cdot \frac{1}{1+\left(\frac{\lambda\tau}{\varepsilon}\right)^2}+
\frac{ C|D|\cdot |\tau|Q}{|D|^2}\cdot \left(\frac{1}{\varepsilon}\cdot\frac{1}{1+\left(\frac{\lambda|\tau|}{\varepsilon}\right)^3}+\frac{1}{\lambda|\tau|}\cdot\frac{1}{1+\left(\frac{\lambda \tau}{\varepsilon}\right)^2}\right)|\tau|R\\
\leq &\;\frac{CQR}{\lambda^2}\cdot \frac{1}{1+\left(\frac{\lambda\tau}{\varepsilon}\right)^2}.
\end{split}
\end{equation}

Combining \eqref{eqn: estimate of difference between SY'} and \eqref{eqn: E_2 and definition of e_2}-\eqref{eqn: bound for e_2 part 2 derivative}, we find that
\begin{equation}
\begin{split}
&\;\int_{\mathbb{T}} |\partial_s e_2|\,d\tau \\
\leq &\; C\int_{\mathbb{T}} \left|\partial_s\left(\frac{SY'(s+\tau)-SY'(s)}{\tau} \right)\right|\left|f_3\left(\frac{|Y'(s)|\tau}{\varepsilon}\right)\right|\,d\tau\\
&\;+C\int_{\mathbb{T}} \left|\partial_s\left(\frac{SY'(s+\tau)-SY'(s)}{\tau} \right)\right|\left|\frac{\tau A}{D^2}\cdot f_3\left(\frac{D}{\varepsilon}\right)-f_3\left(\frac{|Y'(s)|\tau}{\varepsilon}\right)\right|\,d\tau\\
&\;+C\int_{\mathbb{T}}\left|\frac{SY'(s+\tau)-SY'(s)}{\tau} \right|\left|\partial_s\left(\frac{\tau A}{D^2}\cdot f_3\left(\frac{D}{\varepsilon}\right)\right)\right|\,d\tau\\
\leq &\; C\int_{\mathbb{T}} \frac{|F(s+\tau)-F(s)|}{|\tau|}\cdot \frac{1}{1+\left(\frac{\lambda\tau}{\varepsilon}\right)^2}\,d\tau \\
&\;+ \frac{C\mu}{\lambda^2}\int_{\mathbb{T}}(Q+R+|Y''(s)|)QR\cdot \frac{1}{1+\left(\frac{\lambda\tau}{\varepsilon}\right)^2}\,d\tau.
\end{split}
\label{eqn: bound of e_2'}
\end{equation}
By Cauchy-Schwarz inequality,
\begin{equation*}
\begin{split}
&\;\int_{\mathbb{T}} \frac{|F(s+\tau)-F(s)|}{|\tau|}\cdot \frac{1}{1+\left(\frac{\lambda\tau}{\varepsilon}\right)^2}\,d\tau\\
\leq &\;\left(\int_{\mathbb{T}} \frac{|F(s+\tau)-F(s)|^2}{|\tau|^{1+2\theta}}\,d\tau\right)^{1/2}
\left(\int_{\mathbb{T}} \tau^{2\theta-1}\left(1+\left(\frac{\lambda\tau}{\varepsilon}\right)^2\right)^{-2}\,d\tau\right)^{1/2}\\
\leq &\;\frac{C\varepsilon^\theta}{\lambda^\theta}\left(\int_{\mathbb{T}} \frac{|F(s+\tau)-F(s)|^2}{|\tau|^{1+2\theta}}\,d\tau\right)^{1/2}.
\end{split}
\end{equation*}
Therefore, by Minkowski inequality, Sobolev inequality, and Lemma \ref{lemma: H^1/2 estimate of F},
\begin{equation}
\begin{split}
\left\|\int_{\mathbb{T}} |\partial_s e_2|\,d\tau \right\|_{L^2(\mathbb{T})}\leq &\; \frac{C\mu\varepsilon}{\lambda^3}\|Y'\|_{L^\infty(\mathbb{T})}\|Y''\|_{L^4(\mathbb{T})}^2
+\frac{C\varepsilon^{\theta}}{\lambda^{\theta}}\|F\|_{\dot{H}^{\theta}(\mathbb{T})}\\
\leq &\; \frac{C\mu\varepsilon}{\lambda^3}\|Y'\|_{L^\infty(\mathbb{T})}\|Y''\|_{L^4(\mathbb{T})}^2 +\frac{C\mu\varepsilon^\theta}{\lambda^\theta}\|Y\|_{\dot{H}^{2+\theta}(\mathbb{T})}.
\end{split}
\label{eqn: L^2 estimate of integral of e_2'}
\end{equation}
Then we argue as in \eqref{eqn: integral of e_1' is the derivative of E_1} to show that $E_{2}'(s) = \int_{\mathbb{T}}\partial_s e_{2}\,d\tau$; thus
$\|E_{2}'(s)\|_{L^2(\mathbb{T})}$ enjoys the bound in \eqref{eqn: L^2 estimate of integral of e_2'}.
\end{step}

\begin{step}[Estimate for $E_3'$]
Recall 
\begin{equation}
\label{eqn: E_3 and definition of e_3}
E_3(s) = \int_{\mathbb{T}} SY'(s+\tau)^\perp \cdot \frac{B}{D^2} \cdot f_4\left(\frac{D}{\varepsilon}\right)\,d\tau.
\end{equation}
We denote its integrand to be $e_3(s,\tau)$.
By Lemma \ref{lemma: decay estimate of f_1 f_2 f_3}, Lemma \ref{lemma: preliminary estimates as building blocks} and Lemma \ref{lemma: preliminary estimates of derivatives as building blocks}, we calculate that
\begin{equation}
\left|\frac{B}{D^2} \cdot f_4\left(\frac{D}{\varepsilon}\right)\right|
\leq \frac{CQR}{\lambda^2} \cdot \frac{1}{1+\left(\frac{\lambda\tau}{\varepsilon}\right)^2},
\label{eqn: bound for e_3 part 2}
\end{equation}
and 
\begin{equation}
\begin{split}
&\;\left|\partial_s\left(\frac{B}{D^2} \cdot f_4\left(\frac{D}{\varepsilon}\right)\right)\right|\\
\leq &\; \frac{|\partial_s B|}{D^2}\cdot \left|f_4\left(\frac{D}{\varepsilon}\right)\right|+\frac{|B|}{D^2}\cdot \left(\frac{1}{\varepsilon}\left|f_4'\left(\frac{D}{\varepsilon}\right)\right|+\frac{2}{|D|}\left|f_4\left(\frac{D}{\varepsilon}\right)\right|\right)|\partial_s D|\\
\leq &\; \frac{|\partial_s B|}{\tau^2|Y'(s+\tau)|^2}\cdot \left|f_4\left(\frac{D}{\varepsilon}\right)\right|+\left|\frac{1}{D^2}-\frac{1}{\tau^2|Y'(s+\tau)|^2}\right|\cdot |\partial_s B|\cdot\left| f_4\left(\frac{D}{\varepsilon}\right)\right|\\
&\;+\frac{C\tau^2 QR}{\lambda^2\tau^2}\cdot \left(\frac{1}{\varepsilon}\cdot \frac{1}{1+\left(\frac{\lambda|\tau|}{\varepsilon}\right)^3}+\frac{2}{\lambda |\tau|}\cdot\frac{1}{1+\left(\frac{\lambda\tau}{\varepsilon}\right)^2}\right)|\tau|R.
\end{split}
\label{eqn: bound for e_3 part 2 derivative}
\end{equation}
We apply \eqref{eqn: estimate of B' involving H2.5 norm} to the first term to bound $\partial_s B$ while apply \eqref{eqn: estimate of B'} to the second term 
\begin{equation}\label{eqn: bound for e_3 part 2 derivative fractional theta}
\begin{split}
&\;\left|\partial_s\left(\frac{B}{D^2} \cdot f_4\left(\frac{D}{\varepsilon}\right)\right)\right|\\
\leq &\; \frac{C\tau^2R^2}{\tau^2\lambda^2}\cdot \frac{1}{1+\left(\frac{\lambda\tau}{\varepsilon}\right)^2}\\
&\;+\frac{C|\tau|^{1+\theta}|Y'(s+\tau)|}{\tau^2 |Y'(s+\tau)|^2}\left(\int_0^\pi \frac{|Y''(s+\tau-\eta)-Y''(s+\tau)|^2}{|\eta|^{1+2\theta}}\,d\eta\right)^{1/2}
\cdot \frac{1}{1+\left(\frac{\lambda\tau}{\varepsilon}\right)^2}\\
&\;+\frac{C(D+|\tau||Y'(s+\tau)|)|Y(s+\tau)-Y(s)-\tau Y'(s+\tau)|}{D^2 \tau^2|Y'(s+\tau)|^2}\cdot|\tau|QR
\cdot \frac{1}{1+\left(\frac{\lambda\tau}{\varepsilon}\right)^2}\\
&\;+\frac{CQR^2}{\lambda^3}\cdot \frac{1}{1+\left(\frac{\lambda\tau}{\varepsilon}\right)^2}\\
\leq &\;C|Y'(s+\tau)|^{-1}\left(\int_0^\pi \frac{|Y''(s+\tau-\eta)-Y''(s+\tau)|^2}{|\eta|^{1+2\theta}}\,d\eta\right)^{1/2}
\cdot \frac{|\tau|^{\theta-1}}{1+\left(\frac{\lambda\tau}{\varepsilon}\right)^2}\\
&\;+\frac{CQR^2}{\lambda^3}\cdot \frac{1}{1+\left(\frac{\lambda\tau}{\varepsilon}\right)^2}.
\end{split}
\end{equation}
Hence, by Lemma \ref{lemma: H^1/2 estimate of F},
\begin{equation}
\begin{split}
&\;\int_{\mathbb{T}}|\partial_s e_3|\,d\tau \\
\leq &\;\int_{\mathbb{T}} |SY'(s+\tau)|\left|\partial_s\left(\frac{B}{D^2} \cdot f_4\left(\frac{D}{\varepsilon}\right)\right)\right|+|F(s+\tau)|\left|\frac{B}{D^2} \cdot f_4\left(\frac{D}{\varepsilon}\right)\right|\,d\tau\\
\leq &\;C\mu\int_{\mathbb{T}} \left(\int_0^\pi \frac{|Y''(s+\tau-\eta)-Y''(s+\tau)|^2}{|\eta|^{1+2\theta}}\,d\eta\right)^{1/2}
\cdot \frac{|\tau|^{\theta-1}}{1+\left(\frac{\lambda\tau}{\varepsilon}\right)^2}\,d\tau\\
&\;+\frac{C\mu}{\lambda^3}\int_{\mathbb{T}}Q^2R^2
\cdot \frac{1}{1+\left(\frac{\lambda\tau}{\varepsilon}\right)^2}\,d\tau+C\mu\int_{\mathbb{T}} (Q+R)\cdot \frac{QR}{\lambda^2}\cdot \frac{1}{1+\left(\frac{\lambda\tau}{\varepsilon}\right)^2}\,d\tau.
\end{split}
\label{eqn: bound of e_3'}
\end{equation}
By Minkowski inequality and Sobolev inequality,
\begin{equation}
\left\|\int_{\mathbb{T}}|\partial_s e_3|\,d\tau \right\|_{L^2(\mathbb{T})}\leq \frac{C\mu \varepsilon}{\lambda^4}\|Y'\|_{L^\infty(\mathbb{T})}^2\|Y''\|_{L^4(\mathbb{T})}^2+\frac{C\mu\varepsilon^\theta}{\lambda^\theta}\|Y\|_{\dot{H}^{2+\theta}(\mathbb{T})}.
\label{eqn: L^2 estimate of integral of e_3'}
\end{equation}
Then we argue as in \eqref{eqn: integral of e_1' is the derivative of E_1} that $E_{3}'(s) = \int_{\mathbb{T}}\partial_s e_{3}\,d\tau$; and
$\|E_{3}'(s)\|_{L^2(\mathbb{T})}$ enjoys the bound in \eqref{eqn: L^2 estimate of integral of e_3'}.
\end{step}

Combining \eqref{eqn: L^2 estimate of integral of e_1'}, \eqref{eqn: L^2 estimate of integral of e_2'} and \eqref{eqn: L^2 estimate of integral of e_3'}, we complete the proof of Proposition \ref{prop: H^1 error estimate of the string velocity}.
\end{proof}

\section{Singular Limit and Dynamic Error Estimates}\label{section: proof of convergence and error estimates}

In this section, we come back to the case of Hookean elasticity, and prove Theorem \ref{thm: error estimates of the regularized problem} and Theorem \ref{thm: convergence and error estimates of the low-frequency regularized IB method}.

\subsection{Proof of Theorem \ref{thm: error estimates of the regularized problem}}\label{section: proof of dynamic error estimate assuming boundedness}
We will follow the blueprint sketched in Section \ref{section: main results} to prove Theorem \ref{thm: error estimates of the regularized problem}.
Recall that $X^\varepsilon-X$ satisfies the equation
\begin{equation}
\partial_t (X^\varepsilon-X) = \mathcal{L}(X^\varepsilon-X) + (g_{X^\varepsilon}-g_X)+(U_{X^\varepsilon}^\varepsilon-U_{X^\varepsilon}),\quad (X^\varepsilon-X)(0) = 0,
\label{eqn: difference between contour dynamic equation between X and X_eps recap}
\end{equation}
where
\begin{equation}\label{eqn: def of g_Y}
g_Y =U_Y -\mathcal{L}Y= U_Y +\frac{1}{4}(-\Delta)^{1/2}Y.
\end{equation}
In order to bound $X^\varepsilon-X$, we will use estimates for $g_{X^\varepsilon}-g_X$ and $U_{X^\varepsilon}^\varepsilon - U_{X^\varepsilon}$.

Although many estimates for $g_Y$ have been established in \cite[Section 3]{lin2017solvability},
we need improved ones in the proofs of Theorem \ref{thm: error estimates of the regularized problem} and Theorem \ref{thm: convergence and error estimates of the low-frequency regularized IB method}. 

\begin{lemma}\label{lemma: improved H2 estimate for g_X}
Let $Y\in H^{2+\theta}(\mathbb{T})$ satisfy \eqref{eqn: well-stretched condition}, with $\theta\in[\frac{1}{4},1)$.
Then
$\|g_{Y}\|_{\dot{H}^2(\mathbb{T})}\leq C\lambda^{-3}\|Y\|_{\dot{H}^{9/4}(\mathbb{T})}^4$.
\end{lemma}

\begin{lemma}
\label{lemma: improved L2 estimate for g_X1-g_X2}
Let $Y_1,Y_2\in H^{2+\theta}(\mathbb{T})$ both satisfy \eqref{eqn: well-stretched condition}, with $\theta\in[\frac{1}{4},1)$.
For $\beta$ satisfying \eqref{eqn: admissible range of beta},
\begin{equation}
\|g_{Y_1}-g_{Y_2}\|_{L^2(\mathbb{T})}\leq C\lambda^{-2}(\|Y_1\|_{\dot{H}^{2+\theta}}+\|Y_2\|_{\dot{H}^{2+\theta}})^2 \|\delta Y\|_{\dot{H}^{\frac{1}{2}-\beta}(\mathbb{T})}.
\label{eqn: improved L2 estimate for g_X1-g_X2}
\end{equation}
\end{lemma}

\begin{lemma}
\label{lemma: improved H1 estimate for g_X1-g_X2}
Let $Y_1,Y_2\in H^{2+\theta}(\mathbb{T})$ both satisfy \eqref{eqn: well-stretched condition}, with $\theta\in[\frac{1}{4},1)$. 
For $\beta$ satisfying \eqref{eqn: admissible range of beta larger},
\begin{equation}
\|g_{Y_1}-g_{Y_2}\|_{\dot{H}^1(\mathbb{T})}\leq
C\lambda^{-3}(\|Y_1\|_{\dot{H}^{2+\theta}}+\|Y_2\|_{\dot{H}^{2+\theta}})^3\|\delta Y\|_{\dot{H}^{\frac{3}{2}-\beta}(\mathbb{T})}.
\label{eqn: improved H1 estimate for g_X1-g_X2}
\end{equation}
\end{lemma}
We leave their lengthy proofs to Appendix \ref{section: a priori estimate for g_Y}.

\begin{proof}[Proof of Theorem \ref{thm: error estimates of the regularized problem}]
Recall that $\tilde{\varepsilon} = \varepsilon/\lambda$, and $X^\varepsilon$ and $X$ satisfy assumptions \ref{assumption: uniform boundedness} and \ref{assumption: uniform stretching}.
\setcounter{step}{0}
\begin{step}[Error estimate in $H^2$- or higher-order norms]\label{step: dynamic error estimate for H 2 and higher norm}
Thanks to Proposition \ref{prop: static regularization error estimate for Hookean case} and Lemma \ref{lemma: improved H1 estimate for g_X1-g_X2}, for all $t\in[0,T]$,
\begin{align}
\|U_{X^\varepsilon}^\varepsilon - U_{X^\varepsilon}\|_{\dot{H}^1(\mathbb{T})}\leq &\;C\tilde{\varepsilon}^\theta M+C\tilde{\varepsilon}\lambda^{-3}M^4,\label{eqn: estimate for static error in the proof of thm}\\
\|g_{X^\varepsilon}-g_X\|_{\dot{H}^1(\mathbb{T})}\leq &\;C\lambda^{-3}M^3\|X^\varepsilon-X\|_{\dot{H}^{3/2}(\mathbb{T})}.
\label{eqn: estiamte for difference of g in the proof of thm}
\end{align}
With $T_0>0$ to be determined, we apply the energy estimate of \eqref{eqn: difference between contour dynamic equation between X and X_eps recap} to obtain that
\begin{equation}
\begin{split}
&\;
\|X^\varepsilon-X\|_{C_{[0,T_0]}\dot{H}^{3/2}\cap L^2_{T_0}\dot{H}^2(\mathbb{T})}\\
\leq &\;C_0T_0^{1/2}(\lambda^{-3}M^3\|X^\varepsilon-X\|_{L^\infty_{[0,T_0]}\dot{H}^{3/2}(\mathbb{T})}+\tilde{\varepsilon}^{\theta} M+\tilde{\varepsilon}\lambda^{-3}M^4).
\end{split}
\end{equation}
Taking $T_0$ be sufficiently small such that $C_0T_0^{1/2}\lambda^{-3}M^3\leq \frac{1}{2}$, we obtain that
\begin{equation}
\|X^\varepsilon-X\|_{C_{[0,T_0]}\dot{H}^{3/2}\cap L^2_{T_0}\dot{H}^2(\mathbb{T})}\leq C\tilde{\varepsilon}^\theta\lambda^3 M^{-2}+C\tilde{\varepsilon}M\leq C\tilde{\varepsilon}^\theta M.
\end{equation}
Here we used the assumption $\tilde{\varepsilon}\ll 1$ and the fact $\lambda \leq CM$.
If $T\geq T_0$, we may repeat this argument for $[T_0,2T_0]$, $[2T_0,3T_0]$, $\cdots$, until the time interval $[0,T]$ is fully covered.
For instance, for $[T_0, 2T_0]$, the energy estimate of \eqref{eqn: difference between contour dynamic equation between X and X_eps recap} writes
\begin{equation}
\begin{split}
&\;\|X^\varepsilon-X\|_{C_{[T_0,2T_0]}\dot{H}^{3/2}\cap L^2_{[T_0,2T_0]}\dot{H}^2(\mathbb{T})}\\
\leq &\;C\|(X^\varepsilon-X)(T_0)\|_{\dot{H}^{3/2}(\mathbb{T})}\\
&\;+ C_0T_0^{1/2}(\lambda^{-3}M^3\|X^\varepsilon-X\|_{L^\infty_{[T_0,2T_0]}\dot{H}^{3/2}(\mathbb{T})}+\tilde{\varepsilon}^\theta M+\tilde{\varepsilon}\lambda^{-3}M^4).
\end{split}
\end{equation}
With the same $T_0$ chosen earlier,
\begin{equation}
\|X^\varepsilon-X\|_{C_{[T_0,2T_0]}\dot{H}^{3/2}\cap L^2_{[T_0,2T_0]}\dot{H}^2(\mathbb{T})}\leq C\|X^\varepsilon-X\|_{C_{[0,T_0]}\dot{H}^{3/2}(\mathbb{T})}+C\tilde{\varepsilon}^\theta M\leq C\tilde{\varepsilon}^\theta M.
\end{equation}
We conclude that
\begin{equation}
\|X^\varepsilon-X\|_{C_{[0,T]}\dot{H}^{3/2}\cap L^2_T\dot{H}^2(\mathbb{T})}\leq C\tilde{\varepsilon}^\theta M.
\label{eqn: H 3/2 estiamte over 0 to T with assumption on the top norm}
\end{equation}
Here $C = C(\theta,\lambda^{-1}M,T)$ depends exponentially on $T$.

Now take $N\gg 1$ to be determined.
In what follows, we shall bound $\mathcal{P}_N(X^\varepsilon-X)$ and $\mathcal{Q}_N(X^\varepsilon-X)$ separately.
By assumption, for all $t\in[0,T]$,
\begin{equation}\label{eqn: H^2 estimate for high frequency component of error}
\|\mathcal{Q}_N(X^\varepsilon-X)\|_{\dot{H}^2(\mathbb{T})}\leq CN^{-\theta}\|X^\varepsilon-X\|_{\dot{H}^{2+\theta}(\mathbb{T})}\leq CN^{-\theta}M.
\end{equation}
We apply $\mathcal{P}_N$ to both sides of \eqref{eqn: difference between contour dynamic equation between X and X_eps recap}.
By Lemma \ref{lemma: a priori estimate of nonlocal eqn}, \eqref{eqn: estimate for static error in the proof of thm}, \eqref{eqn: estiamte for difference of g in the proof of thm}, and \eqref{eqn: H 3/2 estiamte over 0 to T with assumption on the top norm},
\begin{equation}
\begin{split}
&\;\|\mathcal{P}_N(X^\varepsilon-X)\|_{C_{[0,T]}\dot{H}^2(\mathbb{T})}\\
\leq &\;C(\ln N)^{1/2}(\lambda^{-3}M^3\|X^\varepsilon-X\|_{C_{[0,T]}\dot{H}^{3/2}(\mathbb{T})}+\tilde{\varepsilon}^\theta M+\tilde{\varepsilon}\lambda^{-3}M^4)\\
\leq &\;C(\theta,\lambda^{-1} M, T)(\ln N)^{1/2}\tilde{\varepsilon}^\theta M.
\end{split}
\label{eqn: H^2 estimate for low frequency component of error}
\end{equation}
Combining \eqref{eqn: H^2 estimate for high frequency component of error} and \eqref{eqn: H^2 estimate for low frequency component of error}, and taking $N \sim \tilde{\varepsilon}^{-1}$, we finally obtain that
\begin{equation}
\|X^\varepsilon-X\|_{C_{[0,T]}\dot{H}^2(\mathbb{T})}\leq C(\theta,\lambda^{-1} M, T)|\ln \tilde{\varepsilon}|^{\frac{1}{2}}\tilde{\varepsilon}^\theta M.
\label{eqn: H^2 error estimate}
\end{equation}
By interpolation with $\|X^\varepsilon-X\|_{C_{[0,T]}\dot{H}^{2+\theta}}\leq CM$, we prove \eqref{eqn: error estimate H gamma greater than 2}.
\end{step}

\begin{step}[Error estimate in the $H^{1+\gamma}$-norms]
Error estimates in lower-order norms can be derived as in previous steps with minor adaptation.
By Proposition \ref{prop: static regularization error estimate for Hookean case} and Lemma \ref{lemma: improved L2 estimate for g_X1-g_X2}, for all $t\in[0,T]$, 
%
\begin{align}
\|U^\varepsilon_{X^\varepsilon}-U_{X^\varepsilon}\|_{L^2(\mathbb{T})}
\leq &\;C m_1 \tilde{\varepsilon} M+C\tilde{\varepsilon}^{1+\theta} |\ln\tilde{\varepsilon}|^\theta M+C\tilde{\varepsilon}^2|\ln \tilde{\varepsilon}|\lambda^{-3}M^4, \label{eqn: L2 estimate for static error in the proof of thm}\\
\|g_{X^\varepsilon}-g_{X}\|_{L^2(\mathbb{T})}\leq &\;C\lambda^{-2}M^2\|X^\varepsilon-X\|_{\dot{H}^{1/2}(\mathbb{T})}.\label{eqn: L2 estiamte for difference of g in the proof of thm}
\end{align}
If, in addition, $m_2 = 0$, the logarithmic factors in \eqref{eqn: L2 estimate for static error in the proof of thm} can be removed.
By a similar argument as in the previous step, we can show that 
\begin{equation}
\|X^\varepsilon-X\|_{C_{[0,T]}\dot{H}^{1/2}\cap L^2_{T}\dot{H}^1(\mathbb{T})}\leq C(\theta,\lambda^{-1}M,T) (m_1 \tilde{\varepsilon} +\tilde{\varepsilon}^{1+\theta} |\ln\tilde{\varepsilon}|^\theta)M.
\label{eqn: H 1/2 estiamte over 0 to T with assumption on the top norm}
\end{equation}

To this end, we take $N \sim \tilde{\varepsilon}^{-1}$, and derive that
\begin{equation}\label{eqn: H^1 estimate for high frequency component of error}
\|\mathcal{Q}_N(X^\varepsilon-X)\|_{\dot{H}^1(\mathbb{T})}\leq CN^{-1-\theta}\|X^\varepsilon-X\|_{\dot{H}^{2+\theta}(\mathbb{T})}\leq C\tilde{\varepsilon}^{1+\theta} M.
\end{equation}
On the other hand, by \eqref{eqn: L2 estimate for static error in the proof of thm}-
\eqref{eqn: H 1/2 estiamte over 0 to T with assumption on the top norm} and Lemma \ref{lemma: a priori estimate of nonlocal eqn},
\begin{equation}
\begin{split}
&\;\|\mathcal{P}_N(X^\varepsilon-X)\|_{C_{[0,T]}\dot{H}^1(\mathbb{T})}\\
\leq &\;C(\ln N)^{1/2}\lambda^{-2}M^2\|X^\varepsilon-X\|_{C_{[0,T]}\dot{H}^{1/2}(\mathbb{T})}\\
&\;+C(\theta,\lambda^{-1}M)(\ln N)^{1/2}(m_1 \tilde{\varepsilon}+\tilde{\varepsilon}^{1+\theta} |\ln\tilde{\varepsilon}|^\theta) M\\
\leq &\;
C(\theta,\lambda^{-1}M,T)(\ln N)^{1/2}(m_1 \tilde{\varepsilon} +\tilde{\varepsilon}^{1+\theta} |\ln\tilde{\varepsilon}|^\theta)M.
\end{split}
\label{eqn: H^1 estimate for low frequency component of error}
\end{equation}
Combining \eqref{eqn: H^1 estimate for high frequency component of error} and \eqref{eqn: H^1 estimate for low frequency component of error}, we conclude that
\begin{equation}
\|X^\varepsilon-X\|_{C_{[0,T]}\dot{H}^1(\mathbb{T})}\leq C(\theta,\lambda^{-1} M,T)|\ln \tilde{\varepsilon}|^{\frac{1}{2}}(m_1 \tilde{\varepsilon} +\tilde{\varepsilon}^{1+\theta} |\ln\tilde{\varepsilon}|^\theta)M.
\label{eqn: H^1 error estimate}
\end{equation}
Interpolation between \eqref{eqn: H^2 error estimate} and \eqref{eqn: H^1 error estimate} gives \eqref{eqn: error estimate H gamma between 1 and 2}.

\end{step}
\begin{step}[Error estimate in the $H^{\gamma}$-norms]
By \eqref{eqn: L2 estimate for static error in the proof of thm}-
\eqref{eqn: H 1/2 estiamte over 0 to T with assumption on the top norm} and Lemma \ref{lemma: a priori estimate of nonlocal eqn},
\begin{equation}
\begin{split}
&\;\|X^\varepsilon-X\|_{C_{[0,T]}L^2(\mathbb{T})}\\
\leq &\;C\lambda^{-2}M^2\|X^\varepsilon-X\|_{L^1_T\dot{H}^{1}(\mathbb{T})}+C T m_1 \tilde{\varepsilon} M+CT\tilde{\varepsilon}^{1+\theta}|\ln\tilde{\varepsilon}|^\theta M+CT\tilde{\varepsilon}^2|\ln \tilde{\varepsilon}|\lambda^{-3}M^4\\
\leq &\;C(\theta,\lambda^{-1}M,T) (m_1 \tilde{\varepsilon}+\tilde{\varepsilon}^{1+\theta} |\ln\tilde{\varepsilon}|^\theta)M.
\end{split}
\label{eqn: L^2 error estimate}
\end{equation}
Then 
\eqref{eqn: error estimate H gamma greater than 1/2} follow from interpolation between \eqref{eqn: H 1/2 estiamte over 0 to T with assumption on the top norm}, \eqref{eqn: H^1 error estimate}, and \eqref{eqn: L^2 error estimate}.
\end{step}
\begin{step}[Weak-* convergence in the top regularity]
In order to prove \eqref{eqn: weak star convergence with assumption on the uniform bound}, it suffices to show that the limiting point of the sequence $\{X^\varepsilon-X\}_{\varepsilon}$ in the weak-* topology of $C_{[0,T]}H^{2+\theta}(\mathbb{T})$ is unique.
By \eqref{eqn: difference between contour dynamic equation between X and X_eps recap}, \eqref{eqn: estimate for static error in the proof of thm}, \eqref{eqn: estiamte for difference of g in the proof of thm}, and \eqref{eqn: H 3/2 estiamte over 0 to T with assumption on the top norm}, it is not difficult to show that $\|\partial_t(X^\varepsilon-X)\|_{L^2_T\dot{H}^1(\mathbb{T})}$ is uniformly bounded.
By Aubin-Lions Lemma \cite{temam1984navier}, $\{X^\varepsilon-X\}_{\varepsilon}$ is compact in $L_T^2H^2(\mathbb{T})$.
This implies that any weak-* limiting point of $\{X^\varepsilon-X\}_\varepsilon$ in $C_{[0,T]}H^{2+\theta}(\mathbb{T})$ must be a limiting point of $\{X^\varepsilon-X\}_\varepsilon$ in the strong topology of $L_T^2H^2(\mathbb{T})$.
However, it has been proved that the latter can only be zero.
Hence, \eqref{eqn: weak star convergence with assumption on the uniform bound} is proved.
\end{step}
This completes the proof of Theorem \ref{thm: error estimates of the regularized problem}.
\end{proof}

\subsection{Proof of Theorem \ref{thm: convergence and error estimates of the low-frequency regularized IB method}}
\label{section: proof of singular limit of low-frequency regularized problem}

We first remark on a nice property of the $(\varepsilon,N)$-regularized problem, which is of independent interest. 
\begin{remark}
The $(\varepsilon,N)$-regularized problem is volume-preserving, i.e., area of the domain enclosed by the string is invariant in time.
Note that in the un-regularized problem or the $\varepsilon$-regularized, the volume conservation is a direct consequence of the flow field (or the regularized flow field) being divergence-free.

We derive as follows.
The area of the domain enclosed by the string is given by
\begin{equation}
V(t) = \frac{1}{2}\int_{\mathbb{T}}X^{\varepsilon,N}(s,t)\times (X^{\varepsilon,N})'(s,t)\,ds.
\end{equation}
Taking $t$-derivative and doing integration by parts,
\begin{equation}
V'(t) = \int_{\mathbb{T}}\partial_t X^{\varepsilon,N}(s,t)\times (X^{\varepsilon,N})'(s,t)\,ds.
\end{equation}
This can be justified rigorously since we will show $X^{\varepsilon,N}(s,t)$ is sufficiently smooth.
By \eqref{eqn: low-frequency regularized motion of the membrane},
\begin{equation}
\begin{split}
V'(t) = &\;\int_{\mathbb{T}}\mathcal{P}_N \left[\int_{\mathbb{R}^2} u^{\varepsilon,N}(x,t)\delta_\varepsilon(X^{\varepsilon,N}(s,t)-x)\,dx\right]\times (X^{\varepsilon,N})'(s,t)\,ds\\
= &\;\int_{\mathbb{T}} \left[\int_{\mathbb{R}^2} u^{\varepsilon,N}(x,t)\delta_\varepsilon(X^{\varepsilon,N}(s,t)-x)\,dx\right]\times (X^{\varepsilon,N})'(s,t)\,ds\\
=&\;0.
\end{split}
\end{equation}
In the second equation, we used the fact that $(X^{\varepsilon,N})'= \mathcal{P}_N(X^{\varepsilon,N})'$; in the last equation, we applied the divergence theorem and noticed that
\begin{equation}
\tilde{u}^{\varepsilon,N}(y,t) = \int_{\mathbb{R}^2} u^{\varepsilon,N}(x,t)\delta_\varepsilon(y-x)\,dx
\end{equation}
is divergence-free.

\end{remark}

In the proof of Theorem \ref{thm: convergence and error estimates of the low-frequency regularized IB method}, we will need estimates for $\mathcal{Q}_N X$, the high-frequency portion of $X$, where $\mathcal{Q}_N = Id-\mathcal{P}_N$ as defined in the proof of Proposition \ref{prop: L^2 error estimate of the string velocity}.
In fact, given the assumptions of Theorem \ref{thm: convergence and error estimates of the low-frequency regularized IB method}, a naive one would be $\|\mathcal{Q}_N X\|_{\dot{H}^\gamma(\mathbb{T})}\leq CN^{-2-\theta+\gamma}\|X\|_{\dot{H}^{2+\theta}(\mathbb{T})}$ for all $\gamma\leq 2+\theta$ and $N\geq 1$.
Yet, we shall derive an improved one as follows. 

\begin{lemma}[An improved estimate for $\mathcal{Q}_N X$]\label{lemma: improved estimate for Q_N X}
Under the assumptions on $X(s,t)$ in Theorem \ref{thm: convergence and error estimates of the low-frequency regularized IB method},
for all $\gamma\leq 2+\theta$, and $t\in [0,T_*]$,
\begin{equation}
\|\mathcal{Q}_N X(t)\|_{\dot{H}^{\gamma}(\mathbb{T})}
\leq Ce^{-tN/4}N^{\gamma -2-\theta}\|\mathcal{Q}_NX_0\|_{\dot{H}^{2+\theta}(\mathbb{T})}+CN^{\gamma-3}\lambda^{-3}\|X_0\|_{\dot{H}^{2+\theta}(\mathbb{T})}^4.
\end{equation}
Here the constants $C$ depend on $\theta$ and $\gamma$, but not on $T_*$ or $N$.
\begin{proof}
Consider the equation of $\mathcal{Q}_N X$,
\begin{equation}
\partial_t \mathcal{Q}_N X = \mathcal{L} \mathcal{Q}_N X+ \mathcal{Q}_N g_X,\quad \mathcal{Q}_N X(0) = \mathcal{Q}_N X_0.
\end{equation}
By Lemma \ref{lemma: improved H2 estimate for g_X} and Lemma \ref{lemma: a priori estimate of nonlocal eqn high freq}, for all $\gamma\leq 2+\theta$,
\begin{equation}
\begin{split}
&\;\|\mathcal{Q}_N X(t)\|_{\dot{H}^{\gamma}(\mathbb{T})}\\
\leq &\;e^{-tN/4}N^{\gamma -2-\theta}\|\mathcal{Q}_NX_0\|_{\dot{H}^{2+\theta}(\mathbb{T})}+CN^{\gamma-3}\|\mathcal{Q}_N g_X\|_{L_{T_*}^\infty \dot{H}^2}\\
\leq &\;e^{-tN/4}N^{\gamma -2-\theta}\|\mathcal{Q}_NX_0\|_{\dot{H}^{2+\theta}(\mathbb{T})}+CN^{\gamma-3}\lambda^{-3}\|X_0\|_{\dot{H}^{2+\theta}(\mathbb{T})}^4.
\end{split}
\end{equation}

\end{proof}
\end{lemma}

\begin{proof}[Proof of Theorem \ref{thm: convergence and error estimates of the low-frequency regularized IB method}]
Recall that $\tilde{\varepsilon} = \varepsilon/\lambda$.
\setcounter{step}{0}

\begin{step}[Well-posedness]
The proof of the global well-posedness of the $(\varepsilon,N)$-regularized problem is exactly the same as that in Section \ref{section: well-posedness of regularized problem}, as $\mathcal{P}_N$ is a bounded linear operator in all $H^\gamma(\mathbb{T})$-spaces and the energy estimate remains unchanged in spite of the presence of the projection.
We omit the details, but only note that $X^{\varepsilon,N}$ is continuous from $[0,+\infty)$ to $H^{2+\theta}(\mathbb{T})$.
Indeed, thanks to \eqref{eqn: H beta bound of U_Y},
\begin{equation}
\|U^{\varepsilon,N}_{X^{\varepsilon,N}}\|_{H^{2+\theta}(\mathbb{T})}\leq C(\|X_0\|_{\dot{H}^1(\mathbb{T})},\varepsilon,\theta)\|X^{\varepsilon,N}\|_{\dot{H}^{2+\theta}(\mathbb{T})}.
\end{equation}
This implies the continuity.
\end{step}

\begin{step}[Uniform estimates for $X^{\varepsilon,N}$]

Since $\mathcal{P}_N X_0 \rightarrow X_0$ in $H^{2+\theta}(\mathbb{T})$ and $X_0$ satisfies the well-stretched condition with constant $\lambda$, whenever $N\gg 1$, $\mathcal{P}_N X_0$ satisfies the well-stretched condition with constant $\lambda/2$.
For given $\varepsilon$ and $N$, by the continuity of $X^{\varepsilon,N}$, there exists a maximal $T_{\varepsilon,N}>0$, such that for all $t\in [0,T_{\varepsilon,N}]$,
\begin{equation}
\|X^{\varepsilon,N}(\cdot,t)\|_{\dot{H}^{2+\theta}(\mathbb{T})}\leq 2C_*M_0,
\label{eqn: bound in the maximal time interval}
\end{equation}
with $C_*$ and $M_0$ defined in the statement of Theorem \ref{thm: convergence and error estimates of the low-frequency regularized IB method}, and
\begin{equation}\label{eqn: well-stretched condition in the maximal time interval}
X^{\varepsilon,N}(\cdot,t)\mbox{  satisfies the well-stretched condition with constant }\lambda/4.
\end{equation}
By the maximality, we mean that for any $T'>T_{\varepsilon,N}$, there exists $t\in [0, T']$, such that at least one of \eqref{eqn: bound in the maximal time interval} and \eqref{eqn: well-stretched condition in the maximal time interval} is false.

We shall prove that there exists $c_*>0$ and $N_*>0$, such that for any $\tilde{\varepsilon}\ll 1$ and $N_*\leq  N\leq c_*\tilde{\varepsilon}^{-1}$, we must have $T_{\varepsilon,N}\geq T_*$, with $T_*$ given in  Theorem \ref{thm: convergence and error estimates of the low-frequency regularized IB method}.

Assume otherwise.
Fix $c_*$ small, which will be chosen latter, and we have $T_{\varepsilon,N}< T_*$ when $1\ll  N\leq c_*\tilde{\varepsilon}^{-1}$.
We start with $H^1$- and $H^2$-estimates for $X^{\varepsilon,N}-\mathcal{P}_N X$ by following the proof of Theorem \ref{thm: error estimates of the regularized problem}.
By \eqref{eqn: Stokes equation in the low-frequency regularized IB problem}-\eqref{eqn: low-frequency regularized initial configuration},
\begin{equation}
\begin{split}
\partial_t (X^{\varepsilon,N}-\mathcal{P}_N X) =&\; \mathcal{P}_N (U^\varepsilon_{X^{\varepsilon,N}}-U_X)\\
=&\; \mathcal{P}_N (U_{X^{\varepsilon,N}}-U_X)+\mathcal{P}_N (U^\varepsilon_{X^{\varepsilon,N}}-U_{X^{\varepsilon,N}})\\
=&\;
\mathcal{L}(X^{\varepsilon,N}-\mathcal{P}_N X) + \mathcal{P}_N(g_{X^{\varepsilon,N}}-g_X)+\mathcal{P}_N (U^\varepsilon_{X^{\varepsilon,N}}-U_{X^{\varepsilon,N}}).
\end{split}
\label{eqn: equation for the difference}
\end{equation}
Consider arbitrary $t\in [0,T_{\varepsilon,N}]$.
By \eqref{eqn: bound for benchmark solution in the maximal time interval}, \eqref{eqn: well-stretched condition of benchmark solution in the maximal time interval}, \eqref{eqn: bound in the maximal time interval}, \eqref{eqn: well-stretched condition in the maximal time interval}, and Lemma \ref{lemma: improved H1 estimate for g_X1-g_X2}, 
for any $\beta$ satisfying \eqref{eqn: admissible range of beta larger},
\begin{equation}
\begin{split}
&\;\|\mathcal{P}_{N}(g_{X^{\varepsilon,N}}-g_X) \|_{\dot{H}^1(\mathbb{T})}\\
\leq &\;C\lambda^{-3}M_0^3\|X^{\varepsilon,N}-X\|_{\dot{H}^{\frac{3}{2}-\beta}(\mathbb{T})}\\
\leq &\;C\lambda^{-3}M_0^3(\|X^{\varepsilon,N}-\mathcal{P}_N X\|_{\dot{H}^{\frac{3}{2}-\beta}(\mathbb{T})}+\|\mathcal{Q}_N X\|_{\dot{H}^{\frac{3}{2}-\beta}(\mathbb{T})})\\
\leq &\;C(\|X^{\varepsilon,N}-\mathcal{P}_N X\|_{\dot{H}^{\frac{3}{2}-\beta}(\mathbb{T})}+N^{-\frac{1}{2}-\theta-\beta}M_0),
\end{split}
\label{eqn: H 1 semi norm of low freq of g-g}
\end{equation}
where $C=C(\beta,\lambda^{-1}M_0)$.
Here we used the naive estimate $\|\mathcal{Q}_N X\|_{\dot{H}^{\frac{3}{2}-\beta}}\leq CN^{-\frac{1}{2}-\theta-\beta}\|X\|_{\dot{H}^{2+\theta}}$.
Similarly, for any $\beta'$ satisfying \eqref{eqn: admissible range of beta}, 
\begin{equation}
\|\mathcal{P}_{N}(g_{X^{\varepsilon,N}}-g_X) \|_{L^2(\mathbb{T})}
\leq 
C(\|X^{\varepsilon,N}-\mathcal{P}_N X\|_{\dot{H}^{\frac{1}{2}-\beta'}(\mathbb{T})}+N^{-\frac{3}{2}-\theta-\beta'}M_0).
\end{equation}
On the other hand, by Proposition \ref{prop: static regularization error estimate for Hookean case},
\begin{align}
\|\mathcal{P}_N(U_{X^{\varepsilon,N}}^\varepsilon-U_{X^{\varepsilon,N}})\|_{\dot{H}^1(\mathbb{T})}\leq &\;C\tilde{\varepsilon}^\theta M_0+C \tilde{\varepsilon}\lambda^{-3}M_0^4,
\label{eqn: H 1 semi norm of low freq of regularization error}\\
\|\mathcal{P}_N(U_{X^{\varepsilon,N}}^\varepsilon-U_{X^{\varepsilon,N}})\|_{L^2(\mathbb{T})}
\leq &\;C m_1 \tilde{\varepsilon}M_0+C\tilde{\varepsilon}^{1+\theta}|\ln\tilde{\varepsilon}|^\theta M_0 +C\tilde{\varepsilon}^2 |\ln \tilde{\varepsilon}|\lambda^{-3} M_0^4.
\label{eqn: L^2 bound of low freq regularization error}
\end{align}
We argue as in the proof of Theorem \ref{thm: error estimates of the regularized problem} to obtain that 
\begin{align}
\|X^{\varepsilon,N}-\mathcal{P}_N X\|_{C_{[0,T_{\varepsilon,N}]}\dot{H}^2(\mathbb{T})}
\leq &\;C(\ln N)^{\frac{1}{2}}(N^{ -\frac{1}{2}-\theta-\beta}+\tilde{\varepsilon}^\theta)M_0,
\label{eqn: H 2 estimate in T eps interval}\\
\|X^{\varepsilon,N}-\mathcal{P}_N X\|_{C_{[0,T_{\varepsilon,N}]}\dot{H}^1(\mathbb{T})}
\leq &\;
C(\ln N)^{\frac{1}{2}}(N^{-\frac{3}{2}-\theta-\beta'}+m_1\tilde{\varepsilon}+\tilde{\varepsilon}^{1+\theta}|\ln \tilde{\varepsilon}|^\theta)M_0,
\label{eqn: H 1 estimate in T eps interval}\\
\|X^{\varepsilon,N}-\mathcal{P}_N X\|_{C_{[0,T_{\varepsilon,N}]}H^{1/2}(\mathbb{T})}
\leq &\;
C(N^{-\frac{3}{2}-\theta-\beta'}+m_1\tilde{\varepsilon}+\tilde{\varepsilon}^{1+\theta}|\ln \tilde{\varepsilon}|^\theta)M_0.
\label{eqn: H 1/2 estimate in T eps interval}
\end{align}
By interpolation between \eqref{eqn: H 2 estimate in T eps interval} and \eqref{eqn: H 1 estimate in T eps interval},
\begin{equation}
\|X^{\varepsilon,N}-\mathcal{P}_N X\|_{C_{[0,T_{\varepsilon,N}]}\dot{H}^{\frac{3}{2}-\beta}(\mathbb{T})} \leq
C(\ln N)^{\frac{1}{2}}(N^{-\frac{1}{2}-\theta-\beta}+\tilde{\varepsilon}^\theta)^{\frac{1}{2}-\beta}(N^{-\frac{3}{2}-\theta-\beta'}+\tilde{\varepsilon})^{\frac{1}{2}+\beta}M_0.
\label{eqn: H 3/2 estimate in T eps interval}
\end{equation}

To this end, consider the equation for
$E_{\varepsilon,N}\triangleq \partial_s(X^{\varepsilon,N}-\mathcal{P}_N X)$.
By \eqref{eqn: equation for the difference}, 
\begin{align}
\partial_t E_{\varepsilon,N}= &\; \mathcal{L}E_{\varepsilon,N}+ \mathcal{P}_N\partial_s(g_{X^{\varepsilon,N}}-g_X)+\mathcal{P}_N\partial_s(U_{X^{\varepsilon,N}}^\varepsilon-U_{X^{\varepsilon,N}}),\label{eqn: difference between contour dynamic equation between X and X_eps N}\\
E_{\varepsilon,N}(0) = &\;0.\label{eqn: initial data for difference between contour dynamic equation between X and X_eps N}
\end{align}
Let $E_{\varepsilon,N} = E_{\varepsilon,N}^{(1)}+E_{\varepsilon,N}^{(2)}$, where $E_{\varepsilon,N}^{(1)}$ and $E_{\varepsilon,N}^{(2)}$ solve
\begin{align}
\partial_t E_{\varepsilon,N}^{(1)}= &\; \mathcal{L}E_{\varepsilon,N}^{(1)}+ \mathcal{P}_N\partial_s(g_{X^{\varepsilon,N}}-g_X),\quad E_{\varepsilon,N}^{(1)}(0) = 0,\\
\partial_t E_{\varepsilon,N}^{(2)}= &\; \mathcal{L}E_{\varepsilon,N}^{(2)}+ \mathcal{P}_N\partial_s(U_{X^{\varepsilon,N}}^\varepsilon-U_{X^{\varepsilon,N}}),\quad E_{\varepsilon,N}^{(2)}(0) = 0,
\end{align}
respectively.

First we derive an estimate for $E_{\varepsilon,N}^{(2)}$.
By \eqref{eqn: H 1 semi norm of low freq of regularization error},
for all $1\leq n\leq N\leq c_*\tilde{\varepsilon}^{-1}$,
\begin{equation}
\|\mathcal{P}_n\partial_s(U_{X^{\varepsilon,N}}^\varepsilon-U_{X^{\varepsilon,N}})\|_{\dot{H}^{\theta}(\mathbb{T})}\leq C(\theta, \lambda^{-1}M_0) M_0 \cdot \tilde{\varepsilon}^\theta n^\theta.
\end{equation}
Hence, by 
Lemma \ref{lemma: a priori estimate of nonlocal eqn} and \eqref{eqn: H 1 semi norm of low freq of regularization error},
\begin{equation}
\begin{split}
\|E_{\varepsilon,N}^{(2)}\|_{C_{[0,T_{\varepsilon,N}]} L^2(\mathbb{T})}
\leq &\;\|\mathcal{P}_N\partial_s(U_{X^{\varepsilon,N}}^\varepsilon-U_{X^{\varepsilon,N}})\|_{L^1_{T_{\varepsilon,N}}L^2(\mathbb{T})}\\
\leq &\;C(\theta, \lambda^{-1}M_0,T_*)M_0 \cdot\tilde{\varepsilon}^\theta.
\end{split}
\end{equation}
and for all $n\in \mathbb{Z}_+$,
\begin{equation}
\begin{split}
\|\mathcal{P}_{n,2n}E_{\varepsilon,N}^{(2)}\|_{C_{[0,T_{\varepsilon,N}]} \dot{H}^{1+\theta}(\mathbb{T})}
\leq &\;\|\mathcal{P}_{2n}\mathcal{P}_N\partial_s(U_{X^{\varepsilon,N}}^\varepsilon-U_{X^{\varepsilon,N}})\|_{L^\infty_{T_{\varepsilon,N}}\dot{H}^{\theta}(\mathbb{T})}\\
\leq &\;C(\theta, \lambda^{-1}M_0) M_0 \cdot \tilde{\varepsilon}^\theta n^\theta.
\end{split}
\end{equation}
Suppose $N \in (2^{k_*},2^{k_*+1}]$ with some $k_*\geq 1$.
Then by Parsevel's identity,
\begin{equation}
\begin{split}
&\;\|E_{\varepsilon,N}^{(2)}\|_{C_{[0,T_{\varepsilon,N}]} \dot{H}^{1+\theta}(\mathbb{T})}^2\\
\leq &\;\|E_{\varepsilon,N}^{(2)}\|_{C_{[0,T_{\varepsilon,N}]} L^2(\mathbb{T})}^2+ \sum_{k = 0}^{k_*}\|\mathcal{P}_{2^k,2^{k+1}}E_{\varepsilon,N}^{(2)}\|_{C_{[0,T_{\varepsilon,N}]} \dot{H}^{1+\theta}(\mathbb{T})}^2\\
\leq &\;C(\theta, \lambda^{-1}M_0,T_*) M_0^2  \sum_{k = 0}^{k_*} \tilde{\varepsilon}^{2\theta} 2^{2k\theta}\\
\leq &\;C(\theta, \lambda^{-1}M_0,T_*) M_0^2  \cdot \tilde{\varepsilon}^{2\theta} 2^{2k_*\theta}\sum_{k = 0}^{k_*}2^{-2\theta(k_*-k)} \\
\leq &\;C(\theta, \lambda^{-1}M_0,T_*) M_0^2  \cdot (\tilde{\varepsilon} N)^{2\theta}.
\end{split}
\label{eqn: bound for E_2 in the top regularity}
\end{equation}
Next we consider $E_{\varepsilon,N}^{(1)}$.
Combining \eqref{eqn: H 1 semi norm of low freq of g-g} and \eqref{eqn: H 3/2 estimate in T eps interval}, for all $t\in [0,T_{\varepsilon,N}]$,
\begin{equation}
\begin{split}
&\;\|\mathcal{P}_{N}\partial_s(g_{X^{\varepsilon,N}}-g_X) \|_{L^2(\mathbb{T})}\\
\leq &\;C\left[(\ln N)^{\frac{1}{2}}(N^{-\frac{1}{2}-\theta-\beta}+\tilde{\varepsilon}^\theta)^{\frac{1}{2}-\beta}(N^{-\frac{3}{2}-\theta-\beta'}+\tilde{\varepsilon})^{\frac{1}{2}+\beta}+N^{-\frac{1}{2}-\theta-\beta}\right]M_0, 
\end{split}
\label{eqn: crude H_1 bound for g X eps - g X low freq}
\end{equation}
where $C=C(\beta,\theta, \lambda^{-1}M_0, T_*)$.
Moreover, for all $n\in\mathbb{Z}_+$,
\begin{equation*}
\|\mathcal{P}_n\partial_s(g_{X^{\varepsilon,N}}-g_X) \|_{\dot{H}^\theta(\mathbb{T})}
\leq Cn^\theta\|\mathcal{P}_{N}\partial_s(g_{X^{\varepsilon,N}}-g_X) \|_{L^2(\mathbb{T})}.
\end{equation*}
Then we argue as above to derive that
%
\begin{equation}
\|E_{\varepsilon,N}^{(1)} \|_{C_{[0,T_{\varepsilon,N}]}\dot{H}^{1+\theta}(\mathbb{T})}
\leq CN^\theta \|\mathcal{P}_{N}\partial_s(g_{X^{\varepsilon,N}}-g_X) \|_{L^\infty _{T_{\varepsilon,N}}L^2(\mathbb{T})}.
\label{eqn: bound for E_1 in the top regularity}
\end{equation}
Combining  \eqref{eqn: bound for E_2 in the top regularity}-
\eqref{eqn: bound for E_1 in the top regularity}, 
\begin{equation}
\begin{split}
&\;\|X^{\varepsilon,N}-\mathcal{P}_NX\|_{C_{[0,T_{\varepsilon,N}]} \dot{H}^{2+\theta}(\mathbb{T})}\\
\leq &\;C \left[(\ln N)^{\frac{1}{2}}(N^{-\frac{1}{2}-\beta}+(\tilde{\varepsilon}N)^\theta)^{\frac{1}{2}-\beta}(N^{-\frac{3}{2}-\beta'}+\tilde{\varepsilon}N^\theta)^{\frac{1}{2}+\beta}+N^{-\frac{1}{2}-\beta}+(\tilde{\varepsilon} N)^{\theta}\right]M_0\\
\leq &\;C (N^{-\frac{1}{2}-\beta}+(\tilde{\varepsilon} N)^{\theta})M_0,
\end{split}
\label{eqn: difference between the solution X eps N and P_N X}
\end{equation}
where $C = C(\beta, \theta, \lambda^{-1}M_0,T_*)$.
Here we simplify the estimate by fixing $\beta'$ to be any small number, and using the assumption that $\tilde{\varepsilon}\leq c_*N^{-1}$.

By Lemma \ref{lemma: improved estimate for Q_N X},
\begin{equation}
\|\mathcal{Q}_N X(t)\|_{\dot{H}^{\gamma}(\mathbb{T})}
\leq C(e^{-tN/4}\|\mathcal{Q}_N X_0\|_{\dot{H}^{2+\theta}(\mathbb{T})}+N^{-1+\theta}M_0),
\end{equation}
which implies
\begin{equation}
\lim_{N\rightarrow +\infty}\|X-\mathcal{P}_NX\|_{C_{[0,T_*]}\dot{H}^{2+\theta}(\mathbb{T})}=0.
\label{eqn: difference between X and P_N X}
\end{equation}
Combining \eqref{eqn: bound for benchmark solution in the maximal time interval}, \eqref{eqn: well-stretched condition of benchmark solution in the maximal time interval}, \eqref{eqn: difference between the solution X eps N and P_N X}, \eqref{eqn: difference between X and P_N X},
we may take $N$ suitably large, which depends on $\lambda$, $M_0$, and $X_0$, and then assume $c_*$ to be suitably small, which also depends on $\lambda$ and $M_0$,
such that $\|(X^{\varepsilon,N}-X)(t)\|_{\dot{H}^{2+\theta}}$ is sufficiently small for all $t\in [0,T_{\varepsilon,N}]$ and thus
\begin{equation}
\|X^{\varepsilon,N}(\cdot,t)\|_{\dot{H}^{2+\theta}(\mathbb{T})}\leq \frac{3}{2}C_*M_0,
\end{equation}
and
\begin{equation}
X^{\varepsilon,N}(\cdot,t)\mbox{  satisfies the well-stretched condition with constant }\lambda/3.
\end{equation}
This contradicts with the maximality of $T_{\varepsilon,N}$, because by the well-posedness this implies that the solution $X^{\varepsilon,N}$ can be extended to a longer time interval without violating \eqref{eqn: bound in the maximal time interval} and \eqref{eqn: well-stretched condition in the maximal time interval}.
Therefore, we prove that there exists $c_*>0$ and $N_*>0$, such that for all $N\in [N_*, c_* \tilde{\varepsilon}^{-1}]$, \eqref{eqn: bound in the maximal time interval} and \eqref{eqn: well-stretched condition in the maximal time interval} hold for all $t\in [0,T_*]$.
As a result, \eqref{eqn: H 2 estimate in T eps interval}, \eqref{eqn: H 1 estimate in T eps interval}, and \eqref{eqn: difference between the solution X eps N and P_N X} become estimates on $[0,T_*]$.
\end{step}

\begin{step}[Convergence and error estimates]
By Lemma \ref{lemma: improved estimate for Q_N X}, for $\gamma = \frac{1}{2}, 1,2, 2+\theta$ and $t\in [0,T_*]$,
\begin{equation}
\|\mathcal{Q}_N X(t)\|_{\dot{H}^{\gamma}(\mathbb{T})}
\leq CN^{\gamma-2-\theta}(e^{-tN/4}+N^{-1+\theta})M_0,
\end{equation}
where $C = C(\gamma,\theta,\lambda^{-1}M_0)$.
Combining this with \eqref{eqn: H 2 estimate in T eps interval}-\eqref{eqn: H 1/2 estimate in T eps interval} and \eqref{eqn: difference between the solution X eps N and P_N X},
we can prove \eqref{eqn: H 1/2 difference between the solution X eps N and X}
-\eqref{eqn: H 2 theta difference between the solution X eps N and X} and thus \eqref{eqn: convergence with assumption on the uniform bound eps N problem}. 
The weak-* convergence \eqref{eqn: weak star convergence with assumption on the uniform bound eps N problem} can be justified in the same way as in the proof of Theorem \ref{thm: error estimates of the regularized problem}.
\end{step}
\end{proof}

\section{Discussion}\label{section: discussion}

\subsection{Improved error estimates revisited}

Proposition \ref{prop: L^2 error estimate of the string velocity} proves improved $L^2$-static error estimate when $m_1 = 0$ or when the elastic force has zero tangential component,
which leads to improved error estimates in Theorem \ref{thm: error estimates of the regularized problem} and Theorem \ref{thm: convergence and error estimates of the low-frequency regularized IB method}.
Corollary \ref{coro: L^2 error estimate of the string normal velocity} implies that the $O(\varepsilon)$-leading term in the $L^2$-static error estimate occurs only in the tangential direction.
In the following, we shall explain that these results have a clear physical interpretation.

Consider a model problem, in which the string is represented by $Y(s)$.
Assume that $Y(0) = (0,0)$ and a local part of the elastic string 
coincides with the segment that connects $(-1,0)$ and $(1,0)$.
The rest part of the string is assumed to be far away from the origin.
This is a simplification of general cases.
Indeed, if we zoom in to any local part of a string with sufficiently regular configuration, the local string segment is always close to being a straight line segment.
Suppose that in the Eulerian coordinate, there is a constant force along that local string segment;
in other words, we assume $F_Y(s)/|Y'(s)| = (f_1,f_2)^\intercal$ in that local segment, where the factor $|Y'(s)|^{-1}$ is the Jacobian between the Eulerian and Lagrangian coordinates.
We consider the velocity field with no regularization around the origin.
The elastic force from the other part of the string always contributes to a smooth velocity field $u_{\mathrm{far}}$ around the origin.
The local string segment, however, generates a flow field $u_{\mathrm{loc}}$ that is not smooth.
In a small neighborhood of the origin, the tangential component behaves like a shear flow, with opposite shear rates on two sides of the horizontal axis.
Indeed, 
\begin{equation}\label{eqn: local velocity field around the origin}
\begin{split}
u_{\mathrm{loc}}(x_1,x_2) = &\;\int_{-1}^1 G((x_1,x_2)-(x'_1,0))\cdot(f_1,f_2)^\intercal\,dx'_1 \\
=&\; \frac{1}{4\pi}\int_{-1}^1 -\frac{(f_1,f_2)^\intercal}{2}\ln((x_1-x'_1)^2+x_2^2)\,dx'\\
&\;+ \frac{f_1}{4\pi}\int_{-1}^1 \frac{((x_1-x_1')^2,(x_1-x_1')x_2)^\intercal}{(x_1-x'_1)^2+x_2^2}\,dx'\\
&\;+ \frac{f_2}{4\pi}\int_{-1}^1 \frac{((x_1-x_1')x_2,x_2^2)^\intercal}{(x_1-x'_1)^2+x_2^2}\,dx'\\
\sim &\; f_1\left(\frac{1}{\pi}-\frac{|x_2|}{2},0\right)^\intercal+ f_2\left(0,\frac{1}{2\pi}\right)^\intercal + O(|f||x|^2)
\end{split}
\end{equation}
for $|x|=|(x_1,x_2)|\ll 1$.
This has been characterized by the jump condition of the tangential component of $\nabla_\mathbf{n} u$ across the immersed boundary when the elastic force has non-zero tangential component there \cite{peskin1993improved}.
In fact, \cite{peskin1993improved} derives the jump condition for the immersed boundary problem with Navier-Stokes equation, but the same argument applies to the stationary Stokes case as well.

Compare
\begin{equation}
U_Y(0) = u(0,0) = u_{\mathrm{loc}}(0,0)+u_{\mathrm{far}}(0,0)
\end{equation}
with $U_Y^\varepsilon(0)$ in the regularized case. 
It is known that (see \eqref{eqn: contour dynamic representation for string velocity the singular problem}-\eqref{eqn: def of varphi})
\begin{equation}
U_Y^\varepsilon(0) = [u*\varphi_\varepsilon](0,0) = [(u_{\mathrm{loc}}+u_{\mathrm{far}})*\varphi_\varepsilon](0,0).
\label{eqn: regularized string velocity around the origin}
\end{equation}
By \eqref{eqn: local velocity field around the origin}, the smoothness of $u_{\mathrm{far}}$, and the fact that $\varphi_\varepsilon$ is supported on a disc of radius $C\varepsilon$,
\begin{equation}
\begin{split}
(U_Y^\varepsilon-U_Y)(0) =&\;\left.\left [-\frac{|x_2|}{2}*\varphi_{\varepsilon}\right]\right|_{(0,0)}(f_1,0)^\intercal+O(|f|\varepsilon^2)\\
= &\;\left[-\frac{\varepsilon}{2}\int_{\mathbb{R}^2}|x_2|\varphi(x_1,x_2)\,dx_1dx_2\right](f_1,0)^\intercal +O(|f|\varepsilon^2)\\
= &\;-\frac{m_1 \varepsilon}{\pi}(f_1,0)^\intercal +O(|f|\varepsilon^2).
\end{split}
\label{eqn: velocity difference in model problem}
\end{equation}
In the last equation, we used \eqref{eqn: a symmetrization trick} which will be proved in the Appendix. 
We should highlight that the leading term in \eqref{eqn: velocity difference in model problem} agrees with that in \eqref{eqn: L^2 error estimate final version}; see also \eqref{eqn: L^2 error estimate with tangent force on the left hand side}.

The calculation \eqref{eqn: velocity difference in model problem} clearly shows where the $O(\varepsilon)$-error 
comes from.
Given non-zero tangential force at a point on the string, the local tangential flow in the un-regularized case has a velocity profile like an absolute value function in the transversal direction.
When mollifying such flow field and restricting that onto the string, an $O(\varepsilon)$-error is produced pointwise in the tangential component,
unless the mollifier $\varphi = \phi*\phi$ is $L^2(\mathbb{R}^2)$-orthogonal to that absolute value function in the transversal direction, which is $|x_2|$ in our case.
By the radial symmetry of $\varphi$, this orthogonality condition is equivalent to $m_1=0$.
Since the normal velocity field is smoother, $O(\varepsilon)$-error only occurs in the tangential component.

With this insight, it is natural to believe that when $\phi$ and $\varphi$ are not necessarily radially symmetric, the right condition for the improved accuracy should be
$\tilde{m}_1(\mathbf{v}) = 0$ for any unit vector $\mathbf{v}$ in $\mathbb{R}^2$, where
\begin{equation}
\tilde{m}_1(\mathbf{v})\triangleq \int_{\mathbb{R}^2}|\mathbf{v}\times (x_1,x_2)|\cdot\phi*\phi(x_1,x_2)\,dx_1dx_2.
\label{eqn: generalized m_1}
\end{equation}
Here $\mathbf{v}$ should be understood as the tangential direction of the string.
If $\phi$ is radially symmetric, this condition reduces to $m_1 = 0$.
However, it is not clear if this condition can be fulfilled by some $\phi$ that is not radially symmetric.

It is noteworthy that the condition $\tilde{m}_1(\mathbf{v}) = 0$ can be treated as a generalization of the one-sided first moment condition in 1-D \cite{beyer1992analysis}, which in the continuous setting requires
\begin{equation}
\int_{\mathbb{R}}\max\{0,x\}\phi(x)\,dx = 0.
\end{equation}
Here with abuse of notations, we use $\phi$ to denote profile of a 1-D regularized $\delta$-function.
Note that this is equivalent to
\begin{equation}
\int_{\mathbb{R}}|x|\phi(x)\,dx = 0
\end{equation}
if $\phi$ is orthogonal with $x$.
However, in our work, we propose the orthogonality condition for $\phi*\phi$ instead of $\phi$.

\subsection{Improved regularized $\delta$-functions}

To this end, we show that the condition $m_1 = 0$ for the improved accuracy is indeed achievable.

Take an arbitrary $\rho\in C_0^\infty(\mathbb{R}^2)$ such that $\rho\geq 0$; $\rho$ is radially symmetric; and $\rho$ is normalized.
Define $\rho_r(x) = r^{-2}\rho(x/r)$.
We shall look for $\phi$ in the form of
\begin{equation}
\phi(x) = \frac{1}{1-c}(\rho_{r}(x)-c\rho(x))
\end{equation}
with some $r\in(0,1)$ and $c\in (0,1)$ to be determined.
Obviously, such $\phi$ satisfies all the assumptions in Theorem \ref{thm: error estimates of the regularized problem}.

Given the ansatz, the condition $m_1=0$ for $\phi$ becomes
\begin{equation}
\int_{\mathbb{R}^2}|x|\cdot(\rho_{r}-c\rho)*(\rho_{r}-c\rho)\,dx = 0.
\label{eqn: condition m_1 vanishes}
\end{equation}
We simplify the left hand side as follows.
\begin{equation}
\begin{split}
&\;\int_{\mathbb{R}^2}|x|\cdot(\rho_{r}-c\rho)*(\rho_{r}-c\rho)\,dx\\
=&\;\int_{\mathbb{R}^2}|x|\cdot\rho_{r}*\rho_{r}\,dx-2c\int_{\mathbb{R}^2}|x|\cdot\rho_{r}*\rho\,dx+c^2\int_{\mathbb{R}^2}|x|\cdot\rho*\rho\,dx\\
=&\;r\int_{\mathbb{R}^2}|x|\cdot\rho*\rho\,dx-2c\int_{\mathbb{R}^2}|x|\cdot\rho_{r}*\rho\,dx+c^2\int_{\mathbb{R}^2}|x|\cdot\rho*\rho\,dx\\
\end{split}
\end{equation}
All the three integrals in the last line are positive due to the assumption $\rho\geq 0$.
Hence, \eqref{eqn: condition m_1 vanishes} admits a root $c\in (0,1)$ as long as $r\in(0,1)$ and
\begin{equation}
\left(\int_{\mathbb{R}^2}|x|\cdot\rho_{r}*\rho\,dx\right)^2>r\left(\int_{\mathbb{R}^2}|x|\cdot\rho*\rho\,dx\right)^2.
\label{eqn: condition for c exists}
\end{equation}
Observe that as $r\rightarrow 0$, the right hand side converges to $0$, while the left hand side
\begin{equation}
\lim_{r\rightarrow 0}\left(\int_{\mathbb{R}^2}|x|\cdot\rho_{r}*\rho\,dx\right)^2 = \left(\int_{\mathbb{R}^2}|x|\rho\,dx\right)^2>0.
\end{equation}
This implies that there exists $r\in(0,1)$ such that \eqref{eqn: condition for c exists} holds; in practice, such $r$ does not have to be extremely small.
Then we can solve for desired $c\in(0,1)$ so that \eqref{eqn: condition m_1 vanishes} holds, which completes the construction of $\phi$.

In the numerical immersed boundary method, choice of the regularized $\delta$-function 
plays a crucial role in many aspects.
Lots of efforts have been made to design a good regularized $\delta$-function,
or to better understand its effect in the accuracy of the immersed boundary method or some other problems with singular source terms \cite{peskin2002immersed, beyer1992analysis, liu2012properties, liu2014p, tornberg2004numerical, stockie1997analysis, brady1988dynamic,  roma1999adaptive,  bringley2008analysis, yang2009smoothing, bao2016gaussian}. 
In the continuous case, our analysis suggests that one can indeed attain improved accuracy by suitably choosing the regularized $\delta$-function.
It is then worthwhile to investigate if the such improvement is possible in the discrete case. 

\subsection{Future problems}

In the paper, we introduce the regularized immersed boundary problem to mimic its discrete counterpart  in the numerical immersed boundary method.
As the convergence and error estimates have been obtained in the continuous case,
it is natural to ask if it is possible to derive an error bound in the discrete and dynamic setting for the immersed boundary method, at least for problems with stationary Stokes equations.
Some static error estimates for the velocity field and pressure have been obtained in the discrete case by assuming regularity of the string and the force along it \cite{liu2012properties,liu2014p,mori2008convergence}. 
\cite{mori2008convergence} also studies a simplified dynamic model problem and obtains its time-dependent error bound. 
An a posteriori analysis of the time-dependent error is performed in a recent work \cite{stotsky2018posteriori}.
However, a complete theory of dynamic error estimate for the numerical immersed boundary method is far from being established.

Besides the numerical issues, 
one question in analysis that has not been answered in this work is whether or not there is convergence from $X^\varepsilon$ to $X$ without extra assumptions \ref{assumption: uniform boundedness} and \ref{assumption: uniform stretching}.
According to Remark \ref{rmk: artificial bound on H2.5 norm}, we may have to look more closely at high frequencies in the string configuration.
Justification of the convergence or a counterexample would be both very interesting.

In this work, we only focus on comparing the string dynamics in the regularized and the original Stokes immersed boundary problems.
It is also important to formulate convergence and error estimates for other quantities such as velocity field and pressure.
Such convergence should be expected in the $(\varepsilon,N)$-regularized problem as the convergence in the string motion has been established.
The subtlety lies in the choice of function spaces in which convergence is established, because the velocity field and pressure in the un-regularized problem is not smooth around the immersed string even though the string configuration and the string velocity are smooth.
It is also noteworthy that the condition $m_1 = 0$ may not necessarily lead to improved accuracy in the velocity field and pressure,
as it only aims at representing the velocity accurately along the string, but not in its neighborhood.
Yet, if the string motion is tracked with higher accuracy, it is possible to come up with a separate scheme to find out the velocity field and pressure more accurately,
which may involve a different regularization or other technicality.

Generalizing this work to the immersed boundary problem with Navier-Stokes equation might be challenging.
To the best of our knowledge, the well-posedness of the un-regularized problem in the Navier-Stokes case is still open,
although there have been many related results on the interface dynamics of two-fluid system \cite{solonnikov1986solvability,tanaka1993global,solonnikov2014theory}.

\appendix

\section{Proofs of Auxiliary Results}\label{section: proofs of auxiliary lemmas}

\subsection{Proof of Lemma \ref{lemma: H^1/2 estimate of F}}\label{section: proof of lemma on regularity of F}
\begin{proof}[Proof of Lemma \ref{lemma: H^1/2 estimate of F}]
In what follows, we use $\partial_s S(|Y'|,\cdot )$ to denote the partial derivative with respect to the second variable only, but we use $S'$ to denote total derivative in $s$.

We first show \eqref{eqn: pointwise bound of F}.
By the definition of $F_Y$ and the Assumption \ref{assumption: C 1,1 regularity assumption},
\begin{equation}
\begin{split}
&\;|(SY')'(s)| \\
\leq &\;|SY''(s)|+|\partial_s S(|Y'(s)|,s)||Y'(s)|+|\partial_pS(|Y'(s)|,s)||Y'(s)|\cdot \frac{|Y'(s)||Y''(s)|}{|Y'(s)|}\\
\leq &\;C\mu(|Y'(s)|+|Y''(s)|).
\end{split}
\label{eqn: proof of pointwise bound of F}
\end{equation}
\eqref{eqn: H^theta estimate of F} with $\theta = 0$ follows immediately.

To show \eqref{eqn: H^theta estimate of F} with $\theta\in(0,1)$, we calculate
\begin{equation}
\begin{split}
&\;F_Y(s+\tau)-F_Y(s)\\
= &\;S(|Y'(s+\tau)|,s+\tau)Y''(s+\tau)-S(|Y'(s)|,s)Y''(s)\\
&\;+\partial_s S(|Y'(s+\tau)|,s+\tau)Y'(s+\tau)-\partial_s S(|Y'(s)|,s)Y'(s)\\
&\;+\partial_p S(|Y'(s+\tau)|,s+\tau)\frac{Y'(s+\tau)\cdot Y''(s+\tau)}{|Y'(s+\tau)|}Y'(s+\tau)\\
&\;-\partial_p S(|Y'(s)|,s)\frac{Y'(s)\cdot Y''(s)}{|Y'(s)|}Y'(s).
\end{split}
\end{equation}
By the Assumption \ref{assumption: C 1,1 regularity assumption} and the choice of $\mu$ in \eqref{eqn: choice of mu in the static estimate}, it is not difficult to derive that
\begin{equation}
\begin{split}
&\;|F_Y(s+\tau)-F_Y(s)|\\
\leq &\;C\mu \left(|\tau|+\frac{|Y'(s+\tau)-Y'(s)|}{c\|Y''\|_{L^2}} \right)|Y''(s+\tau)|+C\mu|Y''(s+\tau)-Y''(s)|\\
&\;+C\mu \left(|\tau|+\frac{|Y'(s+\tau)-Y'(s)|}{c\|Y''\|_{L^2}} \right)|Y'(s+\tau)|+C\mu|Y'(s+\tau)-Y'(s)|\\
\leq &\;C\mu \left(|\tau|+\frac{\|Y'(\cdot+\tau)-Y'(\cdot)\|_{L^\infty}}{c\|Y''\|_{L^2}} \right)|Y''(s+\tau)|+C\mu|Y''(s+\tau)-Y''(s)|\\
&\;+C\mu |\tau||Y'(s+\tau)|+C\mu|Y'(s+\tau)-Y'(s)|.
\end{split}
\end{equation}
Hence, by Sobolev embedding,
\begin{equation}
\|F_Y(\cdot+\tau)-F_Y(\cdot)\|_{L^2_s(\mathbb{T})}\leq C\mu |\tau|\|Y''\|_{L^2}+C\mu\|Y''(\cdot+\tau)-Y''(\cdot)\|_{L^2}.
\end{equation}
By the equivalent definitions of the $\dot{H}^{\theta}(\mathbb{T})$-semi-norm, 
\begin{equation}
\begin{split}
\|F_Y\|_{\dot{H}^{\theta}(\mathbb{T})}^2 \leq &\;C_\theta\int_{-1}^1 \frac{\|F_Y(\cdot+\tau)-F_Y(\cdot)\|^2_{L^2_s}}{|\tau|^{1+2\theta}}\,d\tau
\leq C_\theta\mu^2 \|Y\|_{\dot{H}^{2+\theta}(\mathbb{T})}^2.
\end{split}
\end{equation}
This proves \eqref{eqn: H^theta estimate of F}.
\end{proof}

\subsection{Properties of $f_2$ and $f_3$}\label{section: study of the auxiliary functions f_i}
In this section, we shall prove estimates for the auxiliary functions $f_2$ and $f_3$. 
For convenience, we recall their definitions in Section \ref{section: contour dynamic formulation}.
\begin{align}
f_1(x) &= \frac{x}{\pi}\int_{\mathbb{R}^2} e^{ix\eta_1} \frac{-i\eta_1\eta_2^2}{(\eta_1^2+\eta_2^2)^2}\hat{\varphi}(\eta)\,d\eta,\label{eqn: def of f_1 appendix}\\
f_2(x) &= \frac{x}{\pi}\int_{\mathbb{R}^2} e^{ix\eta_1} \frac{-i\eta_1(\eta_1^2-\eta_2^2)}{(\eta_1^2+\eta_2^2)^2}\hat{\varphi}(\eta)\,d\eta,\label{eqn: def of f_2 appendix}\\
f_3(x) &= f_1(x)-1.\label{eqn: def of f_3 appendix}
\end{align}

\begin{proof}[Proof of Lemma \ref{lemma: decay estimate of f_1 f_2 f_3}]
The proof involves repeated integration by parts in the formulas of $f_1$, $f_2$, and $f_3$.

Since $\varphi$ is smooth and compactly supported, $\hat{\varphi}$ is smooth and decays sufficiently fast at $\infty$.
Hence, the integrals in \eqref{eqn: def of f_1 appendix} and \eqref{eqn: def of f_2 appendix} are absolutely integrable;
in addition, 
$$
\int_{\mathbb{R}} \frac{-i\eta_1\eta_2^2}{(\eta_1^2+\eta_2^2)^2}\hat{\varphi}(\eta)\,d\eta_2\quad \mbox{and}\quad \int_{\mathbb{R}} \frac{-i\eta_1(\eta_1^2-\eta_2^2)}{(\eta_1^2+\eta_2^2)^2}\hat{\varphi}(\eta)\,d\eta_2
$$
decays fast as $\eta_1 \rightarrow \pm\infty$.
This implies that $f_1(0)= f_2(0)= 0$, and $f_1$ and $f_2$ are smooth functions in $x$. 
Therefore, $f_3(0) = -1$ and $f_3$ is also smooth on $\mathbb{R}$.

We start by considering
\begin{equation}
g_1(x) = \frac{f_1(x)}{x}=\frac{1}{\pi}\int_{\mathbb{R}^2} e^{ix\eta_1} \frac{-i\eta_1\eta_2^2}{(\eta_1^2+\eta_2^2)^2}\hat{\varphi}(\eta)\,d\eta.
\end{equation}
Since
\begin{equation}
\frac{\eta_1\eta_2^2}{(\eta_1^2+\eta_2^2)^2} =\frac{\partial}{\partial \eta_2}\left(\frac{1}{2}\arctan \frac{\eta_2}{\eta_1}-\frac{\eta_1\eta_2}{2(\eta_1^2+\eta_2^2)}\right)
\label{eqn: antiderivative wrt eta_2 of the integrand of g_1}
\end{equation}
whenever $\eta_1\not = 0$, by integration by parts first in $\eta_2$ and then in $\eta_1$, we derive that
\begin{equation}
\begin{split}
g_1(x) = &\;\frac{i}{2\pi}\int_{\mathbb{R}^2\backslash\{\eta_1 = 0\}} e^{ix\eta_1}\left(\arctan \frac{\eta_2}{\eta_1}-\frac{\eta_1\eta_2}{\eta_1^2+\eta_2^2}\right)\partial_{\eta_2}\hat{\varphi}(\eta)\,d\eta\\
= &\;\lim_{M\rightarrow +\infty}\frac{1}{2\pi x}\int_{(B_M\backslash B_{1/M})\cap\{\eta_1 > 0\}} \left(\arctan \frac{\eta_2}{\eta_1}-\frac{\eta_1\eta_2}{\eta_1^2+\eta_2^2}\right)\partial_{\eta_2}\hat{\varphi}(\eta)\,de^{ix\eta_1}d\eta_2\\
&\;+\lim_{M\rightarrow +\infty}\frac{1}{2\pi x}\int_{(B_M\backslash B_{1/M})\cap\{\eta_1 < 0\}} \left(\arctan \frac{\eta_2}{\eta_1}-\frac{\eta_1\eta_2}{\eta_1^2+\eta_2^2}\right)\partial_{\eta_2}\hat{\varphi}(\eta)\,de^{ix\eta_1}d\eta_2\\
= &\;\frac{1}{2\pi x}\lim_{M\rightarrow +\infty}\int_{\partial B_M\backslash\{\eta_1 = 0\}}
\left(\arctan \frac{\eta_2}{\eta_1}-\frac{\eta_1\eta_2}{\eta_1^2+\eta_2^2}\right)\partial_{\eta_2}\hat{\varphi}(\eta)\cdot \frac{\eta_1}{M}\,d\sigma \\
&\;-\frac{1}{2\pi x}\lim_{M\rightarrow +\infty}\int_{\partial B_{1/M}\backslash\{\eta_1 = 0\}}
\left(\arctan \frac{\eta_2}{\eta_1}-\frac{\eta_1\eta_2}{\eta_1^2+\eta_2^2}\right)\partial_{\eta_2}\hat{\varphi}(\eta)\cdot \frac{\eta_1}{M^{-1}}\,d\sigma \\
&\;-\frac{1}{2\pi x}\lim_{M\rightarrow +\infty}\int_{[-M,-1/M]\cup [1/M,M]}
\frac{\pi}{2}\mathrm{sgn}(\eta_2)\cdot \partial_{\eta_2}\hat{\varphi}(0,\eta_2)\,d\eta_2\\
&\;+\frac{1}{2\pi x}\lim_{M\rightarrow +\infty}\int_{[-M,-1/M]\cup [1/M,M]}
\frac{\pi}{2}\mathrm{sgn}(-\eta_2)\cdot \partial_{\eta_2}\hat{\varphi}(0,\eta_2)\,d\eta_2\\
&\;-\lim_{M\rightarrow +\infty}\frac{1}{2\pi x}\int_{(B_M\backslash B_{1/M})\backslash\{\eta_1 = 0\}} e^{ix\eta_1}\cdot \frac{\partial}{\partial \eta_1}\left[\left(\arctan \frac{\eta_2}{\eta_1}-\frac{\eta_1\eta_2}{\eta_1^2+\eta_2^2}\right)\partial_{\eta_2}\hat{\varphi}(\eta)\right]\,d\eta_1 d\eta_2.
\end{split}
\label{eqn: processing g_1 integration by parts separately on two regions}
\end{equation}
Here we did integration by parts separately in the regions divided by $\{\eta_1= 0\}$ because $\arctan \frac{\eta_2}{\eta_1}$ is discontinuous across that line.
Indeed,
$$
\lim_{\eta_1\rightarrow 0^\pm}\left(\arctan \frac{\eta_2}{\eta_1}-\frac{\eta_1\eta_2}{\eta_1^2+\eta_2^2}\right) = \frac{\pi}{2}\mathrm{sgn}(\pm \eta_2),\quad \forall\,\eta_2\not = 0.
$$
By the fast decay of $\partial_{\eta_2}\hat{\varphi}$ at $\infty$, the first term is $0$.
The second term is also zero as the perimeter of the circle shrinks to zero while the integrand stays bounded.
Hence,
\begin{equation}
\begin{split}
g_1(x) = &\;-\frac{1}{2 x}\int_{\mathbb{R}}\mathrm{sgn}(\eta_2)\cdot \partial_{\eta_2}\hat{\varphi}(0,\eta_2)\,d\eta_2\\
&\;-\frac{1}{2\pi x}\int_{\mathbb{R}^2\backslash\{\eta_1 = 0\}} e^{ix\eta_1} \frac{\partial}{\partial \eta_1}\left[\left(\arctan \frac{\eta_2}{\eta_1}-\frac{\eta_1\eta_2}{\eta_1^2+\eta_2^2}\right)\partial_{\eta_2}\hat{\varphi}(\eta)\right]\,d\eta_1 d\eta_2\\
= &\;\frac{1}{x}-\frac{1}{2\pi x}\int_{\mathbb{R}^2\backslash\{\eta_1 = 0\}} e^{ix\eta_1} \left(\arctan \frac{\eta_2}{\eta_1}-\frac{\eta_1\eta_2}{\eta_1^2+\eta_2^2}\right)\partial_{\eta_1\eta_2}\hat{\varphi}(\eta)\,d\eta_1 d\eta_2\\
&\;+\frac{1}{\pi x}\int_{\mathbb{R}^2\backslash\{\eta_1 = 0\}} e^{ix\eta_1} \cdot\frac{\eta_2^3}{(\eta_1^2+\eta_2^2)^2}\cdot\partial_{\eta_2}\hat{\varphi}(\eta)\,d\eta_1 d\eta_2.
\end{split}
\label{eqn: g_1 is of order 1/x}
\end{equation}
Here we used the fact that
\begin{equation}
\int_{\mathbb{R}}\mathrm{sgn}(\eta_2)\cdot \partial_{\eta_2}\hat{\varphi}(0,\eta_2)\,d\eta_2 = \int_{\mathbb{R}_+}\partial_{\eta_2}\hat{\varphi}(0,\eta_2)\,d\eta_2  - \int_{\mathbb{R}_-}\partial_{\eta_2}\hat{\varphi}(0,\eta_2)\,d\eta_2  = -2\hat{\varphi}(0) = -2.
\end{equation}
Observing that
\begin{equation}
\frac{\eta_2^3}{(\eta_1^2+\eta_2^2)^2} = \frac{\partial}{\partial \eta_2}\left(\frac{1}{2}\ln(\eta_1^2+\eta_2^2)+\frac{1}{2}\frac{\eta_1^2}{\eta_1^2+\eta_2^2}\right),
\label{eqn: antiderivative wrt to y second round}
\end{equation}
whenever $(\eta_1,\eta_2)\not = 0$, we perform integration by parts in $\eta_2$ to obtain that
\begin{equation}
\begin{split}
f_3(x) = &\;xg_1(x) -1\\
= &\;-\frac{1}{2\pi }\int_{\mathbb{R}^2\backslash\{\eta_1 = 0\}} e^{ix\eta_1} \left(\arctan \frac{\eta_2}{\eta_1}-\frac{\eta_1\eta_2}{\eta_1^2+\eta_2^2}\right)\partial_{\eta_1\eta_2}\hat{\varphi}(\eta)\,d\eta_1 d\eta_2\\
&\;-\frac{1}{2\pi}\int_{\mathbb{R}^2\backslash\{\eta_1 = 0\}} e^{ix\eta_1} \left(\ln(\eta_1^2+\eta_2^2)+\frac{\eta_1^2}{\eta_1^2+\eta_2^2}\right)\partial_{\eta_2\eta_2}\hat{\varphi}(\eta)\,d\eta_1 d\eta_2.
\end{split}
\label{eqn: simplification of f_3}
\end{equation}
Since $\hat{\varphi}$ is symmetric with respect to the axis $\{\eta_1 = 0\}$, $\partial_{\eta_1\eta_2}\hat{\varphi}(\eta) = 0$ when $\eta_1 = 0$.
This implies that the integrand in \eqref{eqn: simplification of f_3} is continuous away from $(0,0)$.
Hence, by integration by parts in $\eta_1$,
\begin{equation}
\begin{split}
f_3(x)
= &\;-\frac{i}{2\pi x}\int_{\mathbb{R}^2\backslash\{\eta_1 = 0\}}e^{ix\eta_1} \frac{\partial}{\partial \eta_1}\left[ \left(\arctan \frac{\eta_2}{\eta_1}-\frac{\eta_1\eta_2}{\eta_1^2+\eta_2^2}\right)\partial_{\eta_1\eta_2}\hat{\varphi}(\eta)\right]\,d\eta_1 d\eta_2\\
&\;-\frac{i}{2\pi x}\int_{\mathbb{R}^2\backslash\{\eta_1 = 0\}} e^{ix\eta_1} \frac{\partial}{\partial \eta_1}\left[\left(\ln(\eta_1^2+\eta_2^2)+\frac{\eta_1^2}{\eta_1^2+\eta_2^2}\right)\partial_{\eta_2\eta_2}\hat{\varphi}(\eta)\right]
\,d\eta_1 d\eta_2\\
= &\;-\frac{i}{2\pi x}\int_{\mathbb{R}^2\backslash\{\eta_1 = 0\}}e^{ix\eta_1} \left(\arctan \frac{\eta_2}{\eta_1}-\frac{\eta_1\eta_2}{\eta_1^2+\eta_2^2}\right)\partial_{\eta_1\eta_1\eta_2}\hat{\varphi}(\eta)\,d\eta_1 d\eta_2\\
&\;-\frac{i}{2\pi x}\int_{\mathbb{R}^2\backslash\{\eta_1 = 0\}}e^{ix\eta_1}\cdot \frac{-2\eta_2^3}{(\eta_1^2+\eta_2^2)^2}\cdot \partial_{\eta_1\eta_2}\hat{\varphi}(\eta)\,d\eta_1 d\eta_2\\
&\;-\frac{i}{2\pi x}\int_{\mathbb{R}^2\backslash\{\eta_1 = 0\}} e^{ix\eta_1} \left(\ln(\eta_1^2+\eta_2^2)+\frac{\eta_1^2}{\eta_1^2+\eta_2^2}\right)\partial_{\eta_1\eta_2\eta_2}\hat{\varphi}(\eta)
\,d\eta_1 d\eta_2\\
&\;-\frac{i}{2\pi x}\int_{\mathbb{R}^2\backslash\{\eta_1 = 0\}} e^{ix\eta_1} \cdot\frac{2\eta_1^3+4\eta_1\eta_2^2}{(\eta_1^2+\eta_2^2)^2}\cdot\partial_{\eta_2\eta_2}\hat{\varphi}(\eta)
\,d\eta_1 d\eta_2.
\end{split}
\label{eqn: further simplification of f_3}
\end{equation}
Thanks to \eqref{eqn: antiderivative wrt to y second round}, the second term coincides with the third term.
Moreover,
$$
\frac{2\eta_1^3+4\eta_1\eta_2^2}{(\eta_1^2+\eta_2^2)^2} = \frac{\partial}{\partial \eta_2}\left(3\arctan \frac{\eta_2}{\eta_1}-\frac{\eta_1\eta_2}{\eta_1^2+\eta_2^2}\right).
$$
Hence,
\begin{equation}
\begin{split}
f_3(x)=&\;-\frac{i}{2\pi x}\int_{\mathbb{R}^2\backslash\{\eta_1 = 0\}}e^{ix\eta_1} \left(\arctan \frac{\eta_2}{\eta_1}-\frac{\eta_1\eta_2}{\eta_1^2+\eta_2^2}\right)\partial_{\eta_1\eta_1\eta_2}\hat{\varphi}(\eta)\,d\eta_1 d\eta_2\\
&\;-\frac{i}{\pi x}\int_{\mathbb{R}^2\backslash\{\eta_1 = 0\}} e^{ix\eta_1} \left(\ln(\eta_1^2+\eta_2^2)+\frac{\eta_1^2}{\eta_1^2+\eta_2^2}\right)\partial_{\eta_1\eta_2\eta_2}\hat{\varphi}(\eta)
\,d\eta_1 d\eta_2\\
&\;+\frac{i}{2\pi x}\int_{\mathbb{R}^2\backslash\{\eta_1 = 0\}} e^{ix\eta_1} \left(3\arctan \frac{\eta_2}{\eta_1}-\frac{\eta_1\eta_2}{\eta_1^2+\eta_2^2}\right)\partial_{\eta_2\eta_2\eta_2}\hat{\varphi}(\eta)
\,d\eta_1 d\eta_2.
\end{split}
\label{eqn: decay of f_3 intermediate step}
\end{equation}
Recall that $m_2$ is defined in \eqref{eqn: def of M_2}.
Since $\hat{\varphi}$ is radially symmetric,
\begin{equation}
\partial_{\eta_1\eta_1}\hat{\varphi}(0) = \partial_{\eta_2\eta_2}\hat{\varphi}(0) = -\int_{\mathbb{R}^2}x_1^2\cdot\varphi(x_1,x_2)\,dx_1dx_2 = -\frac{1}{2}m_2.
\label{eqn: relation between second derivative of hat phi with M_2}
\end{equation}
To this end, we perform further integration by parts to \eqref{eqn: decay of f_3 intermediate step} and proceed as in \eqref{eqn: processing g_1 integration by parts separately on two regions} and \eqref{eqn: g_1 is of order 1/x} to find that,
\begin{equation}
\begin{split}
f_3(x)
= &\;-\frac{1}{x^2}\partial_{\eta_1\eta_1}\hat{\varphi}(0)+\frac{3}{ x^2}\partial_{\eta_2\eta_2}\hat{\varphi}(0)\\
&\;+\frac{1}{2\pi x^2}\int_{\mathbb{R}^2\backslash\{\eta_1 = 0\}}  e^{ix\eta_1} \frac{\partial}{\partial\eta_1}\left[\left(\arctan \frac{\eta_2}{\eta_1}-\frac{\eta_1\eta_2}{\eta_1^2+\eta_2^2}\right)\partial_{\eta_1\eta_1\eta_2}\hat{\varphi}(\eta)\right]\,d\eta_1 d\eta_2\\
&\;+\frac{1}{\pi x^2}\int_{\mathbb{R}^2\backslash\{\eta_1 = 0\}}e^{ix\eta_1} \frac{\partial}{\partial\eta_1}\left[\left(\ln(\eta_1^2+\eta_2^2)+\frac{\eta_1^2}{\eta_1^2+\eta_2^2}\right)\partial_{\eta_1\eta_2\eta_2}\hat{\varphi}(\eta)
\right]\,d\eta_1 d\eta_2\\
&\;-\frac{1}{2\pi x^2}\int_{\mathbb{R}^2\backslash\{\eta_1 = 0\}}e^{ix\eta_1} \frac{\partial}{\partial\eta_1}\left[\left(3\arctan \frac{\eta_2}{\eta_1}-\frac{\eta_1\eta_2}{\eta_1^2+\eta_2^2}\right)\partial_{\eta_2\eta_2\eta_2}\hat{\varphi}(\eta)\right]
\,d\eta_1 d\eta_2\\
\triangleq &\; -m_2x^{-2}+(f_{3,1}(x)+f_{3,2}(x)+f_{3,3}(x))x^{-2}.
\end{split}
\label{eqn: final formula for f_3}
\end{equation}
Note that the term $-m_2 x^{-2}$ arises from the discontinuity of $\arctan \frac{\eta_2}{\eta_1}$ across $\{\eta_1 = 0\}$.
Since the integrands in $f_{3,i}$ are all absolutely integrable, by the Riemann-Lebesgue Lemma, 
\begin{equation}
f_{3,1}(x)+f_{3,2}(x)+f_{3,3}(x)\rightarrow 0,\quad \mbox{ as } x\rightarrow\pm \infty.
\end{equation}
This together with the smoothness of $f_3$ implies that there exists a universal $C>0$, such that
\begin{equation}
|f_3(x)|\leq \frac{C}{1+x^2}.
\label{eqn: basic decay estimate for f_3}
\end{equation}

Next we prove estimates for $f'_3$ and $f''_3$.
It suffices to consider $f'_{3,i}$ and $f''_{3,i}$.
We start with $f_{3,1}$ and derive as in \eqref{eqn: g_1 is of order 1/x}, \eqref{eqn: antiderivative wrt to y second round}, and \eqref{eqn: simplification of f_3},
\begin{equation}\label{eqn: f_31 simplification}
\begin{split}
f_{3,1}(x)
=&\;\frac{1}{2\pi}\int_{\mathbb{R}^2\backslash\{\eta_1 = 0\}}  e^{ix\eta_1} \left(\arctan \frac{\eta_2}{\eta_1}-\frac{\eta_1\eta_2}{\eta_1^2+\eta_2^2}\right)\partial_{\eta_1\eta_1\eta_1\eta_2}\hat{\varphi}(\eta)\,d\eta_1 d\eta_2\\
&\;+\frac{1}{2\pi}\int_{\mathbb{R}^2\backslash\{\eta_1 = 0\}}  e^{ix\eta_1} \left(\ln(\eta_1^2+\eta_2^2)+\frac{\eta_1^2}{\eta_1^2+\eta_2^2}\right)\partial_{\eta_1\eta_1\eta_2\eta_2}\hat{\varphi}(\eta)\,d\eta_1 d\eta_2.
\end{split}
\end{equation}
Notice that this is in a form similar to \eqref{eqn: simplification of f_3}.
We proceed as in \eqref{eqn: further simplification of f_3}, \eqref{eqn: decay of f_3 intermediate step}, and \eqref{eqn: final formula for f_3} to obtain that
\begin{equation}
\begin{split}
f_{3,1}(x)
= &\;\frac{1}{x^2}\partial_{\eta_1}^4\hat{\varphi}(0)-\frac{3}{ x^2}\partial_{\eta_1}^2\partial_{\eta_2}^2\hat{\varphi}(0)\\
&\;-\frac{1}{2\pi x^2}\int_{\mathbb{R}^2\backslash\{\eta_1 = 0\}}  e^{ix\eta_1} \frac{\partial}{\partial\eta_1}\left[\left(\arctan \frac{\eta_2}{\eta_1}-\frac{\eta_1\eta_2}{\eta_1^2+\eta_2^2}\right)\partial_{\eta_1}^4\partial_{\eta_2}\hat{\varphi}(\eta)\right]\,d\eta_1 d\eta_2\\
&\;-\frac{1}{\pi x^2}\int_{\mathbb{R}^2\backslash\{\eta_1 = 0\}}e^{ix\eta_1} \frac{\partial}{\partial\eta_1}\left[\left(\ln(\eta_1^2+\eta_2^2)+\frac{\eta_1^2}{\eta_1^2+\eta_2^2}\right)\partial_{\eta_1}^3\partial_{\eta_2}^2\hat{\varphi}(\eta)
\right]\,d\eta_1 d\eta_2\\
&\;+\frac{1}{2\pi x^2}\int_{\mathbb{R}^2\backslash\{\eta_1 = 0\}}e^{ix\eta_1} \frac{\partial}{\partial\eta_1}\left[\left(3\arctan \frac{\eta_2}{\eta_1}-\frac{\eta_1\eta_2}{\eta_1^2+\eta_2^2}\right)\partial_{\eta_1}^2\partial_{\eta_2}^3\hat{\varphi}(\eta)\right]
\,d\eta_1 d\eta_2.
\end{split}
\label{eqn: final formula for f_31}
\end{equation}
Hence, for $k = 1,2$,
\begin{equation}\label{eqn: f_31 derivative}
\begin{split}
&\;(x^2f_{3,1})^{(k)}(x)\\
= &\;-\frac{1}{2\pi}\int_{\mathbb{R}^2\backslash\{\eta_1 = 0\}}  e^{ix\eta_1} (i\eta_1)^k \frac{\partial}{\partial\eta_1}\left[\left(\arctan \frac{\eta_2}{\eta_1}-\frac{\eta_1\eta_2}{\eta_1^2+\eta_2^2}\right)\partial_{\eta_1}^4\partial_{\eta_2}\hat{\varphi}(\eta)\right]\,d\eta_1 d\eta_2\\
&\;-\frac{1}{\pi}\int_{\mathbb{R}^2\backslash\{\eta_1 = 0\}}e^{ix\eta_1} (i\eta_1)^k\frac{\partial}{\partial\eta_1}\left[\left(\ln(\eta_1^2+\eta_2^2)+\frac{\eta_1^2}{\eta_1^2+\eta_2^2}\right)\partial_{\eta_1}^3\partial_{\eta_2}^2\hat{\varphi}(\eta)
\right]\,d\eta_1 d\eta_2\\
&\;+\frac{1}{2\pi}\int_{\mathbb{R}^2\backslash\{\eta_1 = 0\}}e^{ix\eta_1} (i\eta_1)^k \frac{\partial}{\partial\eta_1}\left[\left(3\arctan \frac{\eta_2}{\eta_1}-\frac{\eta_1\eta_2}{\eta_1^2+\eta_2^2}\right)\partial_{\eta_1}^2\partial_{\eta_2}^3\hat{\varphi}(\eta)\right]
\,d\eta_1 d\eta_2.
\end{split}
\end{equation}
We apply the Riemann-Lebesgue Lemma as before to claim that $(x^2f_{3,1})^{(k)}$ stays bounded as $x\rightarrow \pm\infty$ for $k = 0,1,2$.
Similarly, we can rewrite $f_{3,2}$ and $f_{3,3}$ as
\begin{equation}
\begin{split}
f_{3,2}(x)
= &\;-\frac{12}{ x^2}\partial_{\eta_1}^2\partial_{\eta_2}^2\hat{\varphi}(0)\\
&\;-\frac{1}{\pi x^2}\int_{\mathbb{R}^2\backslash\{\eta_1 = 0\}}e^{ix\eta_1} \frac{\partial}{\partial\eta_1}\left[\left(\ln(\eta_1^2+\eta_2^2)+\frac{\eta_1^2}{\eta_1^2+\eta_2^2}\right)\partial_{\eta_1}^3\partial_{\eta_2}^2\hat{\varphi}(\eta)\right]\,d\eta_1 d\eta_2\\
&\;+\frac{2}{\pi x^2}\int_{\mathbb{R}^2\backslash\{\eta_1 = 0\}}e^{ix\eta_1} \frac{\partial}{\partial\eta_1}\left[\left(3\arctan \frac{\eta_2}{\eta_1}-\frac{\eta_1\eta_2}{\eta_1^2+\eta_2^2}\right)\partial_{\eta_1}^2\partial_{\eta_2}^3\hat{\varphi}(\eta)
\right]\,d\eta_1 d\eta_2\\
&\;+\frac{1}{\pi x^2}\int_{\mathbb{R}^2\backslash\{\eta_1 = 0\}}e^{ix\eta_1} \frac{\partial}{\partial\eta_1}\left[\left(2\ln(\eta_1^2+\eta_2^2)+\frac{\eta_1^2}{\eta_1^2+\eta_2^2}\right)\partial_{\eta_1}\partial_{\eta_2}^4\hat{\varphi}(\eta)\right]
\,d\eta_1 d\eta_2,
\end{split}
\label{eqn: final formula for f_32}
\end{equation}
and
\begin{equation}
\begin{split}
f_{3,3}(x)
= &\;-\frac{3}{ x^2}\partial_{\eta_1}^2\partial_{\eta_2}^2\hat{\varphi}(0)+\frac{5}{ x^2}\partial_{\eta_2}^4\hat{\varphi}(0)\\
&\;+\frac{1}{2\pi x^2}\int_{\mathbb{R}^2\backslash\{\eta_1 = 0\}}e^{ix\eta_1} \frac{\partial}{\partial\eta_1}\left[\left(3\arctan \frac{\eta_2}{\eta_1}-\frac{\eta_1\eta_2}{\eta_1^2+\eta_2^2}\right)\partial_{\eta_1}^2\partial_{\eta_2}^3\hat{\varphi}(\eta)
\right]\,d\eta_1 d\eta_2\\
&\;+\frac{1}{\pi x^2}\int_{\mathbb{R}^2\backslash\{\eta_1 = 0\}}e^{ix\eta_1} \frac{\partial}{\partial\eta_1}\left[\left(2\ln(\eta_1^2+\eta_2^2)+\frac{\eta_1^2}{\eta_1^2+\eta_2^2}\right)\partial_{\eta_1}\partial_{\eta_2}^4\hat{\varphi}(\eta)\right]
\,d\eta_1 d\eta_2\\
&\;-\frac{1}{2\pi x^2}\int_{\mathbb{R}^2\backslash\{\eta_1 = 0\}}e^{ix\eta_1} \frac{\partial}{\partial\eta_1}\left[\left(5\arctan \frac{\eta_2}{\eta_1}-\frac{\eta_1\eta_2}{\eta_1^2+\eta_2^2}\right)\partial_{\eta_2}^5\hat{\varphi}(\eta)
\right]\,d\eta_1 d\eta_2.
\end{split}
\label{eqn: final formula for f_33}
\end{equation}
They give similar estimates as those for $f_{3,1}$ as $x\rightarrow\pm\infty$.
Combining them with \eqref{eqn: final formula for f_3} and by the smoothness of $f_3$ on $\mathbb{R}$, we prove that for $k = 0,1,2$,
\begin{equation*}
|f_3^{(k)}(x)| \leq \frac{C}{1+|x|^{k+2}}.
\end{equation*}
If in addition, $m_2=0$, $f_3$ enjoys the following improved estimate for $k = 0,1,2$,
\begin{equation*}
|f_3^{(k)}(x)| \leq \frac{C}{1+x^4}.
\end{equation*}

We analyze $f_2(x)$ using the same approach.
Noticing that
\begin{equation}
\frac{\eta_1(\eta_1^2-\eta_2^2)}{(\eta_1^2+\eta_2^2)^2} = \frac{\partial}{\partial \eta_2}\left(\frac{\eta_1\eta_2}{\eta_1^2+\eta_2^2}\right),
\end{equation}
we calculate in a similar manner,
\begin{equation}
\begin{split}
f_2(x)=&\;\frac{ix}{\pi}\int_{\mathbb{R}^2\backslash\{\eta_1 =0\}} e^{ix\eta_1} \frac{\eta_1\eta_2}{\eta_1^2+\eta_2^2}\partial_{\eta_2}\hat{\varphi}(\eta)\,d\eta\\
=&\;-\frac{1}{\pi}\int_{\mathbb{R}^2\backslash\{\eta_1 =0\}} e^{ix\eta_1} \frac{\partial}{\partial\eta_1}\left[ \frac{\eta_1\eta_2}{\eta_1^2+\eta_2^2}\partial_{\eta_2}\hat{\varphi}(\eta)\right]\,d\eta_1d\eta_2\\
=&\;-\frac{1}{\pi}\int_{\mathbb{R}^2\backslash\{\eta_1 =0\}} e^{ix\eta_1} \left(\frac{\eta_1\eta_2}{\eta_1^2+\eta_2^2}\partial_{\eta_1\eta_2}\hat{\varphi}(\eta)+\frac{-\eta_1^2\eta_2+\eta_2^3}{(\eta_1^2+\eta_2^2)^2}\partial_{\eta_2}\hat{\varphi}(\eta)\right)\,d\eta_1d\eta_2.
\end{split}
\label{eqn: f_2 first integration by parts}
\end{equation}
Since
\begin{equation}
\frac{-\eta_1^2\eta_2+\eta_2^3}{(\eta_1^2+\eta_2^2)^2} = \frac{\partial}{\partial\eta_2}\left(\frac{1}{2}\ln(\eta_1^2+\eta_2^2)+\frac{\eta_1^2}{\eta_1^2+\eta_2^2}\right),
\end{equation}
by integration by parts,
\begin{equation}
\begin{split}
&\;f_2(x)\\
=&\;\frac{i}{\pi x}\int_{\mathbb{R}^2\backslash\{\eta_1 =0\}} e^{ix\eta_1} \frac{\partial}{\partial \eta_1}\left[-\frac{\eta_1\eta_2}{\eta_1^2+\eta_2^2}\partial_{\eta_1\eta_2}\hat{\varphi}(\eta)+\left(\frac{1}{2}\ln(\eta_1^2+\eta_2^2)+\frac{\eta_1^2}{\eta_1^2+\eta_2^2}\right)\partial_{\eta_2\eta_2}\hat{\varphi}(\eta)\right]\,d\eta_1d\eta_2\\
=&\;-\frac{i}{\pi x}\int_{\mathbb{R}^2\backslash\{\eta_1 =0\}} e^{ix\eta_1}\cdot \frac{\eta_1\eta_2}{\eta_1^2+\eta_2^2}\partial_{\eta_1\eta_1\eta_2}\hat{\varphi}(\eta)\,d\eta_1d\eta_2\\
&\;+\frac{2i}{\pi x}\int_{\mathbb{R}^2\backslash\{\eta_1 =0\}} e^{ix\eta_1} \left(\frac{1}{2}\ln(\eta_1^2+\eta_2^2)+\frac{\eta_1^2}{\eta_1^2+\eta_2^2}\right)\partial_{\eta_1\eta_2\eta_2}\hat{\varphi}(\eta)\,d\eta_1 d\eta_2\\
&\;+\frac{i}{\pi x}\int_{\mathbb{R}^2\backslash\{\eta_1 =0\}} e^{ix\eta_1}\cdot \frac{\eta_1^3+3\eta_1\eta_2^2}{(\eta_1^2+\eta_2^2)^2}\cdot\partial_{\eta_2\eta_2}\hat{\varphi}(\eta)\,d\eta_1 d\eta_2.
\end{split}
\end{equation}
Since
\begin{equation}
\frac{\eta_1^3+3\eta_1\eta_2^2}{(\eta_1^2+\eta_2^2)^2} = \frac{\partial}{\partial \eta_2}\left(2\arctan\frac{\eta_2}{\eta_1}-\frac{\eta_1\eta_2}{\eta_1^2+\eta_2^2}\right),
\end{equation}
we proceed as in \eqref{eqn: decay of f_3 intermediate step} and \eqref{eqn: final formula for f_3} to obtain that
\begin{equation}
\begin{split}
f_2(x)=&\; 2m_2 x^{-2}+\frac{1}{\pi x^2}\int_{\mathbb{R}^2\backslash\{\eta_1 =0\}}
e^{ix\eta_1}\frac{\partial}{\partial\eta_1}\left[\frac{\eta_1\eta_2}{\eta_1^2+\eta_2^2} \cdot \partial_{\eta_1\eta_1\eta_2}\hat{\varphi}(\eta)\right]\,d\eta_1d\eta_2\\
&\;-\frac{2}{\pi x^2}\int_{\mathbb{R}^2\backslash\{\eta_1 =0\}} e^{ix\eta_1} \frac{\partial}{\partial\eta_1}\left[\left(\frac{1}{2}\ln(\eta_1^2+\eta_2^2)+\frac{\eta_1^2}{\eta_1^2+\eta_2^2}\right)\partial_{\eta_1\eta_2\eta_2}\hat{\varphi}(\eta)\right]\,d\eta_1 d\eta_2\\
&\;+\frac{1}{\pi x^2}\int_{\mathbb{R}^2\backslash\{\eta_1 =0\}} e^{ix\eta_1} \frac{\partial}{\partial\eta_1}\left[\left(2\arctan\frac{\eta_2}{\eta_1}-\frac{\eta_1\eta_2}{\eta_1^2+\eta_2^2}\right)\partial_{\eta_2\eta_2\eta_2}\hat{\varphi}(\eta)\right]\,d\eta_1 d\eta_2.
\end{split}
\label{eqn: final formula for f_2}
\end{equation}
Again by the Riemann-Lebesgue Lemma, the last three terms are $o(x^{-2})$ as $x\rightarrow \pm \infty$.
Combining this with the smoothness of $f_2$, we conclude that there exists a universal $C>0$, such that
\begin{equation}
|f_2(x)|\leq \frac{C}{1+x^2}.
\end{equation}
Bounds for the derivatives of $f_2$ and the improved estimates when $m_2 = 0$ can be justified in the same way as that for $f_3$.
We omit the details.
\end{proof}

\begin{proof}[Proof of Lemma \ref{lemma: integrals of f_2 and f_3}]
We start from $f_2$.
By \eqref{eqn: f_2 first integration by parts},
\begin{equation}
\begin{split}
\int_{\mathbb{R}}f_2(x)\,dx =\hat{f}_2(0)
=&\;\lim_{\eta_1\rightarrow 0}-2 \int_{\mathbb{R}}\frac{\eta_1\eta_2}{\eta_1^2+\eta_2^2}\partial_{\eta_1\eta_2}\hat{\varphi}(\eta)\,d\eta_2\\
&\;+\lim_{\eta_1\rightarrow 0}-2 \int_{|\eta_2|>\delta'}\left(\frac{-\eta_2}{\eta_1^2+\eta_2^2}+\frac{2\eta_2^3}{(\eta_1^2+\eta_2^2)^2}\right)\partial_{\eta_2}\hat{\varphi}(\eta)\,d\eta_2\\
&\;+\lim_{\eta_1\rightarrow 0}-2 \int_{|\eta_2|\leq \delta'}\left(\frac{-\eta_2}{\eta_1^2+\eta_2^2}+\frac{2\eta_2^3}{(\eta_1^2+\eta_2^2)^2}\right)(\partial_{\eta_2}\hat{\varphi}(\eta)-\partial_{\eta_2}\hat{\varphi}(\eta_1,0))\,d\eta_2\\
\triangleq &\;I_1+I_2+I_3,
\end{split}
\label{eqn: splitting the integral of f_2}
\end{equation}
where $\delta'>0$ is arbitrary.
Here in the last term, we used the fact that the first factor in the integrand is odd in $\eta_2$.
By the dominated convergence theorem,
\begin{equation}
I_1+I_2 = -2 \int_{|\eta_2|>\delta'}\frac{1}{\eta_2}\cdot\partial_{\eta_2}\hat{\varphi}(0,\eta_2)\,d\eta_2.
\end{equation}
For $I_3$,
\begin{equation}
|I_3|
\leq C\|\partial_{\eta_2\eta_2}\hat{\varphi}\|_{L^\infty(\mathbb{R}^2)} \int_{|\eta_2|\leq \delta'} \frac{|\eta_2|}{|\eta_2|}\,d\eta_2 \leq  C\delta',
\end{equation}
where $C>0$ is a universal constant only depending on $\varphi$.
Since $\delta'>0$ is arbitrary, 
\begin{equation}
\int_{\mathbb{R}}f_2(x)\,dx =-2 \cdot \mathrm{p.v.}\int_{\mathbb{R}}\frac{1}{\eta_2}\partial_{\eta_2}\hat{\varphi}(0,\eta_2)\,d\eta_2.
\end{equation}
In a similar fashion, thanks to \eqref{eqn: g_1 is of order 1/x}, one can justify that
\begin{equation}
\begin{split}
\int_{\mathbb{R}}f_3(x)\,dx=  &\;\lim_{\eta_1\rightarrow 0}\int_{\mathbb{R}}- \left(\arctan \frac{\eta_2}{\eta_1}-\frac{\eta_1\eta_2}{\eta_1^2+\eta_2^2}\right)\partial_{\eta_1\eta_2}\hat{\varphi}(\eta)+2 \frac{\eta_2^3}{(\eta_1^2+\eta_2^2)^2}\cdot\partial_{\eta_2}\hat{\varphi}(\eta)\,d\eta_2\\
=&\;2 \cdot \mathrm{p.v.}\int_{\mathbb{R}}\frac{1}{\eta_2}\partial_{\eta_2}\hat{\varphi}(0,\eta_2)\,d\eta_2.
\end{split}
\end{equation}

We then calculate that
\begin{equation}
\begin{split}
\mathrm{p.v.}\int_{\mathbb{R}}\frac{1}{\eta_2}\partial_{\eta_2}\hat{\varphi}(0,\eta_2)\,d\eta_2
=&\; \lim_{\delta'\rightarrow 0}\int_{|\eta_2|\geq \delta'}d\eta_2\,\frac{1}{\eta_2}\cdot \frac{\partial}{\partial \eta_2}\int_{\mathbb{R}^2}e^{-ix_2\eta_2}\varphi(x_1,x_2)\,dx_1dx_2\\
=&\; \lim_{\delta'\rightarrow 0}\int_{\mathbb{R}^2}\left[\int_{|\eta_2|\geq \delta'}\frac{1}{\eta_2} e^{-ix_2\eta_2}\,d\eta_2\right](-ix_2)\varphi(x_1,x_2)\,dx_1dx_2.
\end{split}
\label{eqn: Hilbert transform of hat phi in y_2}
\end{equation}
It is known that
\begin{equation}
\int_{|\eta_2|\geq \delta'}\frac{1}{\eta_2} e^{-ix_2\eta_2}\,d\eta_2 = \int_{|\eta_2|\geq \delta'}\frac{-i\sin(x_2\eta_2)}{\eta_2}\,d\eta_2 = -i\cdot\mathrm{sgn}(x_2)\int_{|\eta_2|\geq |x_2|\delta'}\frac{\sin(\eta_2)}{\eta_2}\,d\eta_2
\end{equation}
is uniformly bounded in $x_2$ and $\delta'$, which converges to $-i\pi\cdot\mathrm{sgn}(x_2)$ as $\delta' \rightarrow 0^+$.
Applying the dominated convergence theorem to \eqref{eqn: Hilbert transform of hat phi in y_2}, we obtain that
\begin{equation}
\begin{split}
\mathrm{p.v.}\int_{\mathbb{R}}\frac{1}{\eta_2}\partial_{\eta_2}\hat{\varphi}(0,\eta_2)\,d\eta_2 =&\; \int_{\mathbb{R}^2}(-i\pi\cdot \mathrm{sgn}(x_2))(-ix_2)\varphi(x_1,x_2)\,dx_1dx_2\\
=&\; -\pi\int_{\mathbb{R}^2}|x_2|\varphi(x_1,x_2)\,dx_1dx_2= -2 m_1.
\end{split}
\end{equation}
Indeed, by the radial symmetry of $\varphi$,
\begin{equation}
\int_{\mathbb{R}^2}|x_2|\varphi(x_1,x_2)\,dx_1dx_2 = \int_0^\infty \int_0^{2\pi}r^2|\sin\theta|\cdot \varphi(r)\,d\theta dr = 4\int_0^\infty r^2\varphi(r)\,dr = \frac{2m_1}{\pi}.
\label{eqn: a symmetrization trick}
\end{equation}
This completes the proof.
\end{proof}

\section{A Priori Estimates Involving $\mathcal{L}=-\frac{1}{4}(-\Delta)^{1/2}$}\label{section: a priori estimate for L}

\begin{lemma}\label{lemma: a priori estimate of nonlocal eqn}
Let $\mathcal{P}_N$ be defined by \eqref{eqn: Fourier transform on T} and \eqref{eqn: projection operator} with $N>1$.
Assume $Z_0\in H^{l+1}(\mathbb{T})$ with $l\geq 0$, and $f\in L^\infty_T H^l(\mathbb{T})$ for $T>0$. 
The model equation
\begin{equation}
\partial_t Z(s,t) = \mathcal{L}Z(s,t) +\mathcal{P}_N f(s,t),\quad Z(s,0) = Z_0(s),\quad s\in \mathbb{T},\; t\geq 0
\label{eqn: model nonlocal equation}
\end{equation}
has a unique solution $Z\in C_{[0,T]} H^{l+1}(\mathbb{T})$, satisfying that 
\begin{align}
\|Z\|_{C_{[0,T]} L^2(\mathbb{T})}\leq &\;\|Z_0\|_{L^2(\mathbb{T})}+\|\mathcal{P}_Nf\|_{L^1_TL^2(\mathbb{T})}.\label{eqn: estimate of L inf L^2 norm of solution of model equation}\\
\|Z\|_{C_{[0,T]} \dot{H}^{l+1}(\mathbb{T})}\leq &\;\|Z_0\|_{\dot{H}^{l+1}(\mathbb{T})}+C(\ln N)^{1/2} \|\mathcal{P}_Nf\|_{L_T^\infty\dot{H}^l(\mathbb{T})}.\label{eqn: estimate of L inf H l+1 semi norm of solution of model equation}
\end{align}
With abuse of notations, $\dot{H}^0(\mathbb{T})$ is understood as $L^2(\mathbb{T})$.

More generally, for arbitrary $1 \leq n_1< n_2\leq N $, define $\mathcal{P}_{n_1,n_2} = \mathcal{P}_{n_2} - \mathcal{P}_{n_1}$.
Then
\begin{equation}
\|\mathcal{P}_{n_1,n_2}Z\|_{C_{[0,T]} \dot{H}^{l+1}(\mathbb{T})}\leq \|\mathcal{P}_{n_1,n_2} Z_0\|_{\dot{H}^{l+1}(\mathbb{T})}+C\left(1+\ln \frac{n_2}{n_1}\right)^{1/2} \|\mathcal{P}_{n_1,n_2}f\|_{L_T^\infty\dot{H}^l(\mathbb{T})}.
\label{eqn: estimate of L inf H l+1 semi norm of solution of model equation in a part of spectrum}
\end{equation}
\begin{proof}
Since $\{e^{t\mathcal{L}}\}_{t\geq 0}$ is a strongly continuous contraction semigroup on $H^{l+1}(\mathbb{T})$ and $L^2(\mathbb{T})$, while $Z_0$ and $\mathcal{P}_N f(\cdot ,t)\in H^{l+1}(\mathbb{T})$, we obtain
\begin{equation}\label{eqn: semi group solution of the model equation}
Z(s,t) = e^{t\mathcal{L}}Z_0(s)+\int_0^t e^{(t- t' )\mathcal{L}}\mathcal{P}_N f(s, t' )\,d t' ,
\end{equation}
as a solution in $C_{[0,T]} H^{l+1}(\mathbb{T})$.
This also implies \eqref{eqn: estimate of L inf L^2 norm of solution of model equation}.
The uniqueness in $C_{[0,T]} H^{l+1}(\mathbb{T})$ follows from that in $L^2_T H^{l+1}(\mathbb{T})$; the latter can be proved by a classic energy estimate.

Now it suffices to show \eqref{eqn: estimate of L inf H l+1 semi norm of solution of model equation in a part of spectrum}, since \eqref{eqn: estimate of L inf H l+1 semi norm of solution of model equation} follows from \eqref{eqn: estimate of L inf L^2 norm of solution of model equation} and \eqref{eqn: estimate of L inf H l+1 semi norm of solution of model equation in a part of spectrum}. 
Applying $\mathcal{P}_{n_1,n_2}$ to \eqref{eqn: semi group solution of the model equation}, we first find that
\begin{equation}\label{eqn: H l+1 semi norm of the homogeneous part of the semigroup solution}
\|e^{t\mathcal{L}}\mathcal{P}_{n_1,n_2} Z_0\|_{\dot{H}^{l+1}(\mathbb{T})}  \leq e^{-n_1 t/4}\|\mathcal{P}_{n_1,n_2}Z_0\|_{\dot{H}^{l+1}(\mathbb{T})},\quad \forall\,t\geq 0.
\end{equation}
Then consider the integral in \eqref{eqn: semi group solution of the model equation}.
By Parseval's identity,
\begin{equation}
\left\|\int_0^t e^{(t- t' )\mathcal{L}}\mathcal{P}_{n_1,n_2} f(s, t' )\,d t' \right\|_{\dot{H}^{l+1}(\mathbb{T})}^2  = C\sum_{|k|\in (n_1,n_2]}|k|^{2(l+1)}\left|\int_0^t e^{-\frac{1}{4}|k|(t- t' )}\hat{f}_k( t' )\,d t' \right|^2.
\label{eqn: H^l semi norm of the integral term}
\end{equation}
By Cauchy-Schwarz inequality,
\begin{equation}
\begin{split}
&\;\sum_{|k|\in (n_1,n_2]}|k|^{2(l+1)}\left|\int_0^t e^{-\frac{1}{4}|k|(t- t' )}\hat{f}_k( t' )\,d t' \right|^2\\
\leq &\;\sum_{|k|\in (n_1,n_2]}|k|^{2(l+1)}\int_0^t e^{-\frac{1}{4}|k|(t- t' )}\,d t'  \int_0^t e^{-\frac{1}{4}|k|(t- t' )}|\hat{f}_k( t' )|^2\,d t' \\
\leq &\;\sum_{|k|\in (n_1,n_2]}\int_0^t |k|e^{-\frac{1}{4}|k|(t- t' )}\cdot |k|^{2l}|\hat{f}_k( t' )|^2\,d t' .
\end{split}
\end{equation}

Suppose $t\geq n_1^{-1}$. Then
\begin{equation}
\begin{split}
&\;\sum_{|k|\in (n_1,n_2]}\int_0^t |k|e^{-\frac{1}{4}|k|(t- t' )}\cdot |k|^{2l}|\hat{f}_k( t' )|^2\,d t' \\
=&\;\sum_{|k|\in (n_1,n_2]}\int_0^{t-n_1^{-1}}+\int_{t-n_1^{-1}}^t |k|e^{-\frac{1}{4}|k|(t- t' )}\cdot |k|^{2l}|\hat{f}_k( t' )|^2\,d t' .
\end{split}
\end{equation}
For $ t' \leq t-n_1^{-1}$,
\begin{equation}
|k|e^{-\frac{1}{4}|k|(t- t' )}\leq Cn_1e^{-\frac{1}{4}n_1(t- t' )},\quad \forall\,|k|\geq n_1.
\end{equation}
For $ t' \in [t-n_1^{-1},t]$,
\begin{equation}
|k|e^{-\frac{1}{4}|k|(t- t' )}\leq \min\left\{\frac{C}{t- t' }, n_2\right\},\quad \forall\,|k|\leq n_2.
\end{equation}
Hence, by Parseval's identity,
\begin{equation}
\begin{split}
&\;\sum_{|k|\in (n_1,n_2]}\int_0^t |k|e^{-\frac{1}{4}|k|(t- t' )}\cdot |k|^{2l}|\hat{f}_k( t' )|^2\,d t' \\
\leq &\;C\int_0^{t-n_1^{-1}} n_1e^{-\frac{1}{4}n_1(t- t' )} \sum_{|k|\in (n_1,n_2]}|k|^{2l}|\hat{f}_k( t' )|^2\,d t' \\
&\;\quad+\int_{t-n_1^{-1}}^t \min\left\{\frac{C}{t- t' }, n_2\right\}\sum_{|k|\in (n_1,n_2]} |k|^{2l}|\hat{f}_k( t' )|^2\,d t' \\
= &\;C\left(C\int_0^{t-n_1^{-1}} n_1e^{-\frac{1}{4}n_1(t- t' )} \,d t'  +\int_{t-n_1^{-1}}^t \min\left\{\frac{C}{t- t' }, n_2\right\} \,d t' \right)\|\mathcal{P}_{n_1,n_2}f\|^2_{L_T^\infty\dot{H}^l(\mathbb{T})}\\
\leq &\;C\left(1+\ln \frac{n_2}{n_1}\right) \|\mathcal{P}_{n_1,n_2}f\|^2_{L_T^\infty\dot{H}^l(\mathbb{T})}.
\end{split}
\label{eqn: log N bound for the integral term in the model problem}
\end{equation}
The case of $t\leq n_1^{-1}$ can be justified in the same way.

Combining \eqref{eqn: semi group solution of the model equation}, \eqref{eqn: H l+1 semi norm of the homogeneous part of the semigroup solution}, \eqref{eqn: H^l semi norm of the integral term}, and \eqref{eqn: log N bound for the integral term in the model problem}, we prove \eqref{eqn: estimate of L inf H l+1 semi norm of solution of model equation in a part of spectrum}.
\end{proof}
\end{lemma}

\begin{lemma}\label{lemma: a priori estimate of nonlocal eqn high freq}
Let $\mathcal{Q}_N = Id - \mathcal{P}_N$ with $N>1$.
Assume $Z_0\in H^{l+\gamma}(\mathbb{T})$ with $l\geq 0$ and $\gamma\in [0,1)$, and $f\in L^\infty_T H^l(\mathbb{T})$ for $T>0$. 
The model equation
\begin{equation}
\partial_t Z(s,t) = \mathcal{L}Z(s,t) +\mathcal{Q}_N f(s,t),\quad Z(s,0) = \mathcal{Q}_N Z_0(s),\quad s\in \mathbb{T},\; t\geq 0
\label{eqn: model nonlocal equation high freq}
\end{equation}
has a unique solution $Z\in C_{[0,T]} H^{l+\gamma}(\mathbb{T})$, satisfying that $Z = \mathcal{Q}_N Z$, and for all $t\in[0,T]$, 
\begin{equation}
\|Z(t)\|_{\dot{H}^{l+\gamma}(\mathbb{T})}\leq e^{-tN/4}\|\mathcal{Q}_NZ_0\|_{\dot{H}^{l+\gamma}(\mathbb{T})}+CN^{\gamma-1}\|\mathcal{Q}_N f\|_{L_T^\infty \dot{H}^l}.
\label{eqn: estimate of L inf H l+1 semi norm of solution of model equation high freq}
\end{equation}
With abuse of notations, $\dot{H}^0(\mathbb{T})$ is understood as $L^2(\mathbb{T})$.

\begin{proof}
Once again,
\begin{equation}\label{eqn: semi group solution of the model equation high freq}
Z(s,t) = e^{t\mathcal{L}}\mathcal{Q}_N Z_0(s)+\int_0^t e^{(t- t' )\mathcal{L}}\mathcal{Q}_N f(s, t' )\,d t' ,
\end{equation}
gives a unique solution in $C_{[0,T]} H^{l+\gamma}(\mathbb{T})$.
Obviously, $Z = \mathcal{Q}_NZ$.

It is known that
\begin{equation}
\|e^{t\mathcal{L}}\mathcal{Q}_N Z_0\|_{\dot{H}^{l+\gamma}(\mathbb{T})}\leq e^{-tN/4}\|\mathcal{Q}_N Z_0\|_{\dot{H}^{l+\gamma}(\mathbb{T})}.
\end{equation}
For the second term in \eqref{eqn: semi group solution of the model equation high freq}, by Parseval's identity,
\begin{equation}
\begin{split}
&\;\left\|\int_0^t e^{(t- t')\mathcal{L}}\mathcal{Q}_{N} f(s, t' )\,d t' \right\|_{\dot{H}^{l+\gamma}(\mathbb{T})}^2\\
=&\; C\sum_{|k|>N}|k|^{2(l+\gamma)}\left|\int_0^t e^{-\frac{1}{4}|k|(t- t' )}\hat{f}_k( t' )\,d t' \right|^2\\
\leq &\;C\sum_{|k|>N}|k|^{2(l+\gamma)}\int_0^t e^{-\frac{1}{4}|k|(t- t' )}\,d t'  \int_0^t e^{-\frac{1}{4}|k|(t- t' )}|\hat{f}_k( t' )|^2\,d t' \\
\leq &\;C\sum_{|k|>N}\int_0^t |k|^{2\gamma-1}e^{-\frac{1}{4}|k|(t- t' )}\cdot |k|^{2l}|\hat{f}_k( t' )|^2\,d t'\\
\leq &\;C\int_0^t |t-t'|^{-(2\gamma-1)}e^{-\frac{1}{8}|N|(t- t' )}\sum_{|k|>N}|k|^{2l}|\hat{f}_k( t' )|^2\,d t'\\
\leq &\;CN^{2\gamma-2}\|\mathcal{Q}_N f\|_{L_T^\infty \dot{H}^l}^2.
\end{split}
\end{equation}
This completes the proof.
\end{proof}
\end{lemma}

\section{A Priori Estimates involving $g_Y$}
\label{section: a priori estimate for g_Y}
This section aims at proving estimates concerning $g_Y$ in Section \ref{section: proof of convergence and error estimates}, which improves the results in \cite{lin2017solvability}.
We first recall some previous results.

Let $Y\in H^2(\mathbb{T})$.
Denote $\tau = s'-s\in[-\pi,\pi)$.
For $s'\not = s$, with abuse of notations, define
\begin{equation}
L(s,s') = \frac{Y(s')-Y(s)}{\tau},\quad M(s,s') = \frac{Y'(s')-Y'(s)}{\tau},\quad N(s,s') = \frac{L(s,s')-Y'(s)}{\tau}.
\label{eqn: definition of L M N}
\end{equation}
and
\begin{equation}
L(s,s) = Y'(s),\quad M(s,s) = Y''(s),\quad N(s,s) =\frac{1}{2}Y''(s).
\label{eqn: definition of L M N at s}
\end{equation}
Then $L$, $M$, and $N$ enjoy the follows estimates.
\begin{lemma}[\cite{lin2017solvability}, Lemma 3.1]
\label{lemma: estimates for L M N}
With the notations above,
\begin{enumerate}
\item
For $\forall\, 1< p\leq q \leq \infty$ and any interval $I\subset\mathbb{T}$ satisfying $0\in I$
\begin{align}
\|L(s,s')\|_{L^q_{s}(\mathbb{T})L^p_{s'}(s+I)} \leq &\;C|I|^{1/q}\|Y'\|_{L^p(\mathbb{T})},\label{eqn: double Lp estimate for L}\\
\|M(s,s')\|_{L^q_{s}(\mathbb{T})L^p_{s'}(s+I)} \leq &\;C|I|^{1/q}\|Y''\|_{L^p(\mathbb{T})},\label{eqn: double Lp estimate for M}\\
\|N(s,s')\|_{L^q_{s}(\mathbb{T})L^p_{s'}(s+I)} \leq &\;C|I|^{1/q}\|Y''\|_{L^p(\mathbb{T})},\label{eqn: double Lp estimate for N}
\end{align}
where the constants $C>0$ only depend on $p$ and $q$.
Here
\begin{equation*}
\|f(s,s')\|_{L^q_{s}(\mathbb{T})L^p_{s'}(s+I)} \triangleq \left\|\|f(s,s')\|_{L^p_{s'}(s+I)}\right\|_{L^q_{s}(\mathbb{T})}.
\end{equation*}
\item 
For $\forall\, s,s'\in\mathbb{T}$,
\begin{equation}
|L(s,s')|\leq 2\mathcal{M} Y'(s),\quad |M(s,s')|\leq 2\mathcal{M} Y''(s),\quad |N(s,s')|\leq 2\mathcal{M} Y''(s).\label{eqn: bound for L M N by maximal function}
\end{equation}
In particular,
\begin{equation}
|L(s,s')|\leq C\|Y'\|_{L^\infty(\mathbb{T})}.
\end{equation}
\end{enumerate}
\end{lemma}

\begin{lemma}[\cite{lin2017solvability}, Remark 2.1, Lemma 3.5, Lemma 3.6, and Lemma C.1]
\label{lemma: formula for g_Y'}
We have
\begin{equation}
g_{Y}(s) = \frac{1}{4\pi}\mathrm{p.v.}\int_\mathbb{T} \left(-\frac{|Y'(s')|^2}{|L|^2}+\frac{2(L\cdot Y'(s'))^2}{|L|^4}-1\right)\frac{L}{\tau} + \left(\frac{1}{\tau}-\frac{\tau}{4\sin^2 (\frac{\tau}{2})}\right)L\,ds',
\label{eqn: formula for g_Y}
\end{equation}
and
\begin{equation}
g'_{Y}(s) = \int_\mathbb{T} \Gamma_{1}(s,s')\,ds',
\label{eqn: introduce the notation Gamma_1}
\end{equation}
where for $s\not = s'$,
\begin{equation}
\begin{split}
&\;4\pi\Gamma_{1}(s,s')\\
= &\;\frac{(Y'(s)-L)\cdot N}{|L|^2}M - \frac{2(N\cdot L)(Y'(s)\cdot L)}{|L|^4}M - \left(\frac{\tau^2 - 4\sin^2(\frac{\tau}{2})}{4\tau\sin^2(\frac{\tau}{2})}\right)M\\
&\;+\frac{(M-2N)\cdot M}{|L|^2}Y'(s)+\frac{2(N\cdot L)( L\cdot M)}{|L|^4}Y'(s)\\
&\; +\frac{2 (L\cdot M) (L\cdot (M-N)) (L\cdot Y'(s))}{|L|^6}L+\frac{2 ((N-M)\cdot M)(L\cdot Y'(s))}{|L|^4}L\\
&\;+\frac{2 (L\cdot M) (L\cdot Y'(s'))}{|L|^4} N+\frac{2 (N\cdot M) (L\cdot Y'(s'))}{|L|^4}L\\
&\;+\frac{2 (L\cdot M) (N\cdot Y'(s'))}{|L|^4}L-\frac{6 (L\cdot M) (L\cdot Y'(s')) (L\cdot N)}{|L|^6}L.
\end{split}
\label{eqn: simplified Gamma order 1}
\end{equation}
Moreover,
\begin{equation}
g''_{Y}(s) = \int_\mathbb{T} \partial_s\Gamma_{1}(s,s')\,ds'.
\label{eqn: formula for g_X''}
\end{equation}
\end{lemma}

We first prove Lemma \ref{lemma: improved H2 estimate for g_X}.

\begin{proof}[Proof of Lemma \ref{lemma: improved H2 estimate for g_X}]

With $\tau = s'-s$, we calculate
\begin{align}
\partial_s L(s,s') = &\;N(s,s'),\\
\partial_s M(s,s') = &\;\frac{Y'(s')-Y'(s)-\tau Y''(s)}{\tau^2}=\frac{M(s,s')- Y''(s)}{\tau},\label{eqn: s derivative of M}\\
\partial_s N(s,s') = &\;\frac{Y(s')-Y(s)-\tau Y'(s)-\frac{\tau^2}{2}Y''(s)}{\tau^3}.
\end{align}
We claim that, by taking $s$-derivative in \eqref{eqn: simplified Gamma order 1}, 
\begin{equation}
\begin{split}
&\;|\partial_s\Gamma_1(s,s')|\\
\leq &\; C\lambda^{-2}\|Y'\|_{L^\infty(\mathbb{T})}(|\partial_s M||M|+|\partial_s M||N|+|\partial_s N||M|)\\
&\;+C\lambda^{-3}\|Y'\|_{L^\infty(\mathbb{T})}|M||N|(|M|+|N|)\\
&\;+C\lambda^{-2}|Y''(s)||M|(|M|+|N|)+C|M|+C|\tau||\partial_s M|.
\end{split}
\label{eqn: derivative of gamma 1}
\end{equation}
In fact, on the right hand side, the first term bounds all the terms in $\partial_s\Gamma_{1}(s,s')$ whenever the $s$-derivative falls on $M$ or $N$ in  \eqref{eqn: simplified Gamma order 1};
the second term comes from the terms when the derivative falls on $L$, including those in the denominators;
the third term shows up because the derivative may hit $Y'(s)$; 
and the last two terms come from the $s$-derivative of the third term in \eqref{eqn: simplified Gamma order 1}.

By Lemma \ref{lemma: formula for g_Y'}, it suffices to bound
\begin{equation}
\left\|\int_{\mathbb{T}}|\partial_s \Gamma_1(s,s+\tau)|\,d\tau\right\|_{L^2_s(\mathbb{T})}.
\end{equation}
It is not difficult to show that by Lemma \ref{lemma: estimates for L M N} and \eqref{eqn: s derivative of M},
\begin{equation}
\begin{split}
&\;\left\|\int_{\mathbb{T}}\lambda^{-3}\|Y'\|_{L^\infty(\mathbb{T})}|M||N|(|M|+|N|)\right.\\
&\;\quad +\lambda^{-2}|Y''(s)||M|(|M|+|N|)+|M|+|\tau||\partial_s M|\,d\tau \bigg\|_{L^2_s(\mathbb{T})}\\
\leq &\; C\lambda^{-3}\|Y'\|_{L^\infty(\mathbb{T})}\|M\|_{L^\infty_s L^2_{s'}(\mathbb{T})}\|N\|_{L^\infty_s L^2_{s'}(\mathbb{T})}\|\mathcal{M}Y''\|_{L^2(\mathbb{T})}\\
&\; +C\lambda^{-2}\|Y''\|_{L^2(\mathbb{T})}\|M\|_{L^\infty_s L^2_{s'}(\mathbb{T})}(\|M\|_{L^\infty_s L^2_{s'}(\mathbb{T})}+\|N\|_{L^\infty_s L^2_{s'}(\mathbb{T})})\\
&\; +C\|Y''\|_{L^2(\mathbb{T})}\\
\leq &\; C\lambda^{-3}\|Y''\|_{L^2(\mathbb{T})}^4.
\end{split}
\label{eqn: easy terms in H2 estimate of g_Y}
\end{equation}

We still need to show
\begin{equation}
\left\|\int_{\mathbb{T}}|\partial_s M||M|+|\partial_s M||N|+|\partial_s N||M|\,d\tau \right\|_{L^2_s(\mathbb{T})}\leq C\|Y\|_{\dot{H}^{9/4}(\mathbb{T})}^2.
\label{eqn: most singular term in H2 estimate of g_Y}
\end{equation}
Notice that
\begin{align}
N(s,s') = &\;\frac{1}{\tau^2}\int_0^\tau Y'(s+\eta)-Y'(s)\,d\eta,\label{eqn: representation of N}\\
\partial_s M(s,s') = &\;\frac{1}{\tau^2}\int_0^\tau Y''(s+\eta)-Y''(s)\,d\eta,\label{eqn: representation of M derivative}\\
\partial_s N(s,s') = &\;\frac{1}{\tau^3}\int_0^\tau\int_0^\eta  Y''(s+\zeta)-Y''(s)\,d\zeta d\eta.\label{eqn: representation of N derivative}
\end{align}
For all $p\in[1,\infty]$,
\begin{equation}
\begin{split}
\|N(s,s+\tau)\|_{L^p_s(\mathbb{T})} \leq &\;\frac{1}{\tau^2}\left|\int_0^\tau \|Y'(s+\eta)-Y'(s)\|_{L^p_s(\mathbb{T})}\right|\,d\eta\\
\leq &\;\frac{C}{|\tau|}\sup_{|\eta|\leq |\tau|}\|Y'(s+\eta)-Y'(s)\|_{L^p_s(\mathbb{T})},
\end{split}
\label{eqn: L^p_s estimate for N}
\end{equation}
and similarly,
\begin{align}
\|\partial_s M(s,s+\tau)\|_{L^p_s(\mathbb{T})}\leq  &\;\frac{C}{|\tau|}\sup_{|\eta|\leq |\tau|}\|Y''(s+\eta)-Y''(s)\|_{L^p_s(\mathbb{T})},\\
\|\partial_s N(s,s+\tau)\|_{L^p_s(\mathbb{T})}\leq  &\;\frac{C}{|\tau|}\sup_{|\eta|\leq |\tau|}\|Y''(s+\eta)-Y''(s)\|_{L^p_s(\mathbb{T})}.\label{eqn: L^p_s estimate for N s derivative}
\end{align}
Hence,
\begin{equation}
\begin{split}
&\;\left\|\int_{\mathbb{T}}|\partial_s M||M|+|\partial_s M||N|+|\partial_s N||M|\,d\tau \right\|_{L^2_s(\mathbb{T})}\\
\leq &\;\int_{\mathbb{T}}(\|\partial_s M\|_{L^3_s(\mathbb{T})}+\|\partial_s N\|_{L^3_s(\mathbb{T})})(\|M\|_{L^6_s(\mathbb{T})}+\|N\|_{L^6_s(\mathbb{T})})\,d\tau\\
\leq &\;C\int_{\mathbb{T}}\frac{1}{|\tau|^{7/12}}\sup_{|\eta|\leq |\tau|}\|Y''(s+\eta)-Y''(s)\|_{L^3_s(\mathbb{T})}\\
&\;\qquad\cdot \frac{1}{|\tau|^{17/12}}\sup_{|\eta|\leq |\tau|}\|Y'(s+\eta)-Y'(s)\|_{L^6_s(\mathbb{T})}\,d\tau\\
\leq &\;C\left(\int_{\mathbb{T}}\frac{1}{|\tau|^{7/6}}\sup_{|\eta|\leq |\tau|}\|Y''(s+\eta)-Y''(s)\|_{L^3_s(\mathbb{T})}^2\,d\tau\right)^{1/2}\\
&\;\quad\cdot \left(\int_{\mathbb{T}}\frac{1}{|\tau|^{17/6}}\sup_{|\eta|\leq |\tau|}\|Y'(s+\eta)-Y'(s)\|_{L^6_s(\mathbb{T})}^2\,d\tau\right)^{1/2}\\
\leq &\;C\|Y''\|_{\dot{B}^{1/12}_{3,2}(\mathbb{T})}\|Y'\|_{\dot{B}^{11/12}_{6,2}(\mathbb{T})}\\
\leq &\;C\|Y\|_{\dot{B}^{2+1/4}_{2,2}(\mathbb{T})}^2.
\end{split}
\label{eqn: bounding the most singular term in g_X H^2 estimate}
\end{equation}
In the last two inequalities, we used equivalent norms of Besov spaces and  embedding theorems between them \cite[\S\,2.5.12 and \S\,2.7.1]{triebel2010theory}.
This proves \eqref{eqn: most singular term in H2 estimate of g_Y}.

Combining \eqref{eqn: derivative of gamma 1}, \eqref{eqn: easy terms in H2 estimate of g_Y}, and \eqref{eqn: most singular term in H2 estimate of g_Y}, we complete the proof of Lemma \ref{lemma: improved H2 estimate for g_X}.
\end{proof}

Next, we prove Lemma \ref{lemma: improved L2 estimate for g_X1-g_X2} and Lemma \ref{lemma: improved H1 estimate for g_X1-g_X2}.

\begin{proof}[Proof of Lemma \ref{lemma: improved L2 estimate for g_X1-g_X2}]

Let $L_i$, $M_i$, and $N_i$ be defined as in \eqref{eqn: definition of L M N} and \eqref{eqn: definition of L M N at s} with $Y$ replaced by $Y_i$ $(i=1,2)$.
Denote $\delta Y = Y_1-Y_2$, and let $\delta L$, $\delta M$ and $\delta N$ be defined in a similar manner.
Define $V = \|Y_1'\|_{L^\infty(\mathbb{T})}+ \|Y_2'\|_{L^\infty(\mathbb{T})}$.
%
By Lemma \ref{lemma: formula for g_Y'}, 
\begin{equation}
\begin{split}
&\;4\pi(g_{Y_1}-g_{Y_2})\\
= &\;\mathrm{p.v.}\int_\mathbb{T} \left(-\frac{\delta Y'(s')\cdot (Y_1'(s')+Y_2'(s'))}{|L_1|^2}+\frac{|Y_2'(s')|^2(L_1+L_2)\cdot \delta L}{|L_1|^2|L_2|^2}\right)\frac{L_1}{\tau}\,ds'\\
&\;-\mathrm{p.v.}\int_\mathbb{T} \frac{2(L_1\cdot Y'_1(s'))^2(L_1+L_2)\cdot \delta L}{|L_1|^4|L_2|^2}\cdot \frac{L_1}{\tau}\,ds'\\
&\;+\mathrm{p.v.}\int_\mathbb{T} \frac{2(L_1\cdot Y'_1(s')+L_2\cdot Y'_2(s'))\cdot (\delta L\cdot Y_1'(s')+L_2 \cdot \delta Y'(s'))}{|L_1|^2|L_2|^2}\cdot\frac{L_1}{\tau}\,ds'\\
&\;-\mathrm{p.v.}\int_\mathbb{T} \frac{2(L_2\cdot Y'_2(s'))^2(L_1+L_2)\cdot \delta L}{|L_1|^2|L_2|^4}\cdot\frac{L_1}{\tau}\,ds'\\
&\;+\mathrm{p.v.}\int_\mathbb{T} \left(-\frac{|Y'_2(s')|^2}{|L_2|^2}+\frac{2(L_2\cdot Y'_2(s'))^2}{|L_2|^4}-1\right)\frac{\delta L}{\tau} + \left(\frac{1}{\tau}-\frac{\tau}{4\sin^2 (\frac{\tau}{2})}\right)\delta L\,ds'.
\end{split}
\end{equation}
We shall do integration by parts to remove the derivative from the $\delta Y'(s')$ terms in the integrand.
Notice that
\begin{equation}
\delta Y'(s') = \partial_{s'}(\delta Y(s')-\delta Y(s)) = \partial_{s'}(\tau \delta L).
\end{equation}
Thanks to the regularity of $Y_i$, it is not difficult to justify the integration by parts in spite of the singularity in the integrand.
Indeed, we obtain
\begin{equation}
\begin{split}
&\;4\pi(g_{Y_1}-g_{Y_2})\\
= &\;\mathrm{p.v.}\int_\mathbb{T} \frac{\delta L}{\tau}\cdot\left(-\frac{(Y_1'(s')+Y_2'(s'))\otimes L_1}{|L_1|^2}+\frac{|Y_2'(s')|^2(L_1+L_2)\otimes L_1}{|L_1|^2|L_2|^2}\right)\,ds'\\
&\;+\mathrm{p.v.}\int_\mathbb{T} \delta L \cdot \partial_{s'}\left(\frac{Y_1'(s')+Y_2'(s')}{|L_1|^2} \otimes L_1\right)\,ds'\\
&\;-\mathrm{p.v.}\int_\mathbb{T} \frac{\delta L}{\tau}\cdot\frac{2(L_1\cdot Y'_1(s'))^2(L_1+L_2)\otimes L_1}{|L_1|^4|L_2|^2}\,ds'\\
&\;+\mathrm{p.v.}\int_\mathbb{T} \frac{\delta L}{\tau}\cdot\frac{2(L_1\cdot Y'_1(s')+L_2\cdot Y'_2(s'))(Y_1'(s')+L_2)\otimes L_1}{|L_1|^2|L_2|^2}\,ds'\\
&\;-\mathrm{p.v.}\int_\mathbb{T} \delta L \cdot \partial_{s'}\left(\frac{2(L_1\cdot Y'_1(s')+L_2\cdot Y'_2(s'))\cdot L_2 \otimes L_1}{|L_1|^2|L_2|^2}\right)\,ds'\\
&\;-\mathrm{p.v.}\int_\mathbb{T} \frac{\delta L}{\tau}\cdot\frac{2(L_2\cdot Y'_2(s'))^2(L_1+L_2)\otimes L_1}{|L_1|^2|L_2|^4}\,ds'\\
&\;+\mathrm{p.v.}\int_\mathbb{T} \frac{\delta L}{\tau}\cdot\left(-\frac{|Y'_2(s')|^2}{|L_2|^2}+\frac{2(L_2\cdot Y'_2(s'))^2}{|L_2|^4}-1\right)\,ds'\\
&\;+\mathrm{p.v.}\int_\mathbb{T} \delta L \left(\frac{1}{\tau}-\frac{\tau}{4\sin^2 (\frac{\tau}{2})}\right)\,ds'\\
\triangleq&\;\sum_{i = 1}^8 J_i. 
\end{split}
\label{eqn: rewrite g_Y1-g_Y2}
\end{equation}

Take $\beta'>\beta$ such that it also  satisfies \eqref{eqn: admissible range of beta}.
Since $Y_i\in H^{2+\theta}(\mathbb{T})$, $|Y_i'(s')-L_i|\leq C|\tau|^{\frac{1}{2}+\beta'} \|Y_i\|_{\dot{H}^{2+\theta}}$.
In $J_1$, we have
\begin{equation*}
\begin{split}
&\;\left|-\frac{(Y_1'(s')+Y_2'(s'))\otimes L_1}{|L_1|^2}+\frac{|Y_2'(s')|^2(L_1+L_2)\otimes L_1}{|L_1|^2|L_2|^2}\right|\\
= &\;\left|\frac{(Y_2'(s')+L_2)\cdot (Y_2'(s')-L_2)(L_1+L_2)\otimes L_1}{|L_1|^2|L_2|^2}+\frac{(L_1+L_2-Y_1'(s')-Y_2'(s'))\otimes L_1}{|L_1|^2}\right|\\
\leq &\;C\lambda^{-2}V|\tau|^{\frac{1}{2}+\beta'} (\|Y_1\|_{\dot{H}^{2+\theta}}+\|Y_2\|_{\dot{H}^{2+\theta}}).
\end{split}
\end{equation*}
Hence, by Cauchy-Schwarz inequality and $\beta'>\beta$,
\begin{equation}
\begin{split}
|J_1|\leq &\; C\lambda^{-2}V(\|Y_1\|_{\dot{H}^{2+\theta}}+\|Y_2\|_{\dot{H}^{2+\theta}}) \int_{\mathbb{T}} \frac{|\delta Y(s+\tau)-\delta Y(s)|}{|\tau|^{\frac{3}{2}-\beta'}}\,d\tau\\
\leq &\; C\lambda^{-2}V(\|Y_1\|_{\dot{H}^{2+\theta}}+\|Y_2\|_{\dot{H}^{2+\theta}})
\left(\int_{\mathbb{T}} \frac{|\delta Y(s+\tau)-\delta Y(s)|^2}{|\tau|^{2-2\beta}}\,d\tau\right)^{1/2}.
\end{split}
\label{eqn: bound for J_1}
\end{equation}
Similarly,
\begin{equation}
|J_3+J_4+J_6|+|J_7|\leq C\lambda^{-2}V(\|Y_1\|_{\dot{H}^{2+\theta}}+\|Y_2\|_{\dot{H}^{2+\theta}})
\left(\int_{\mathbb{T}} \frac{|\delta Y(s+\tau)-\delta Y(s)|^2}{|\tau|^{2-2\beta}}\,d\tau\right)^{1/2}.
\label{eqn: bound for J_3467}
\end{equation}
Since $\partial_{s'}L_i = M_i-N_i$, by Lemma \ref{lemma: estimates for L M N} and Sobolev embedding,
\begin{equation}
\begin{split}
&\;|J_2|+|J_5|\\
\leq &\;C\lambda^{-2}V\int_{\mathbb{T}} |\delta L|\cdot (|Y_1''(s')|+|Y_2''(s')|+|M_1|+|M_2|+|N_1|+|N_2|)\,d\tau\\
\leq &\;C\lambda^{-2}V\left(\int_{\mathbb{T}} \frac{|\delta Y(s+\tau)-\delta Y(s)|^2}{|\tau|^{2-2\beta}}\,d\tau\right)^{1/2}
\left\||\tau|^{-\beta}\right\|_{L^{1/\beta'}}\\
&\;\qquad \cdot
\||Y_1''(s')|+|Y_2''(s')|+|M_1|+|M_2|+|N_1|+|N_2|\|_{L^{1/(\frac{1}{2}-\beta')}_{s'}}\\
\leq &\;C\lambda^{-2}V(\|Y_1\|_{\dot{H}^{2+\theta}}+\|Y_2\|_{\dot{H}^{2+\theta}})
\left(\int_{\mathbb{T}} \frac{|\delta Y(s+\tau)-\delta Y(s)|^2}{|\tau|^{2-2\beta}}\,d\tau\right)^{1/2}.
\end{split}
\label{eqn: bound for J_25}
\end{equation}
Finally,
\begin{equation}
|J_8|\leq C\int_{\mathbb{T}} |\delta Y(s+\tau)-\delta Y(s)|\,d\tau.
\label{eqn: bound for J_8}
\end{equation}
Combining \eqref{eqn: rewrite g_Y1-g_Y2}-
\eqref{eqn: bound for J_8} and taking $L^2$-norm in $s$, we prove the desired estimate \eqref{eqn: improved L2 estimate for g_X1-g_X2}.

\end{proof}

\begin{proof}[Proof of Lemma \ref{lemma: improved H1 estimate for g_X1-g_X2}]
To prove \eqref{eqn: improved H1 estimate for g_X1-g_X2}, let $\Gamma_{1,i}$ $(i =1,2)$ be defined by \eqref{eqn: introduce the notation Gamma_1} and \eqref{eqn: simplified Gamma order 1} with $Y$ replaced by $Y_i$.
By Lemma \ref{lemma: formula for g_Y'}, it suffices to bound
\begin{equation}
\left\|\int_{\mathbb{T}} \Gamma_{1,1} (s,s')-\Gamma_{1,2} (s,s')\,ds'\right\|_{L_s^2(\mathbb{T})}.
\label{eqn: goal in bounding H1 norm of g-g}
\end{equation}
For conciseness, we only show how to bound the part of the difference arising from the last term of \eqref{eqn: simplified Gamma order 1}.
We write
\begin{equation}
\begin{split}
&\;\frac{ (L_1\cdot M_1) (L_1\cdot Y_1'(s')) (L_1\cdot N_1)}{|L_1|^6}L_1-\frac{ (L_2\cdot M_2) (L_2\cdot Y_2'(s')) (L_2\cdot N_2)}{|L_2|^6}L_2\\
=&\;\frac{ (L_1\cdot \delta M) (L_1\cdot Y_1'(s')) (L_1\cdot N_1)}{|L_1|^6}L_1+\frac{ (L_1\cdot M_2) (L_1\cdot Y_1'(s')) (L_1\cdot \delta N)}{|L_1|^6}L_1\\
&\;+\frac{ (\delta L\cdot M_2) (L_1\cdot Y_1'(s')) (L_1\cdot N_2)}{|L_1|^6}L_1\\
&\;+\frac{ (L_2\cdot M_2) (L_1\cdot Y_1'(s')) (L_1\cdot N_2)}{|L_1|^4}L_1\cdot \frac{|L_2|^2-|L_1|^2}{|L_1|^2|L_2|^2}\\
&\;+\frac{ (L_2\cdot M_2) (\delta L\cdot Y_1'(s')) (L_1\cdot N_2)}{|L_1|^4|L_2|^2}L_1+\frac{ (L_2\cdot M_2) (L_2\cdot \delta Y'(s')) (L_1\cdot N_2)}{|L_1|^4|L_2|^2}L_1\\
&\;+\frac{ (L_2\cdot M_2) (L_2\cdot Y_2'(s')) (L_1\cdot N_2)}{|L_1|^2|L_2|^2}L_1\cdot \frac{|L_2|^2-|L_1|^2}{|L_1|^2|L_2|^2}\\
&\;+\frac{ (L_2\cdot M_2) (L_2\cdot Y_2'(s')) (\delta L\cdot N_2)}{|L_1|^2|L_2|^4}L_1\\
&\;+\frac{ (L_2\cdot M_2) (L_2\cdot Y_2'(s')) (L_2\cdot N_2)}{|L_2|^4}L_1\cdot \frac{|L_2|^2-|L_1|^2}{|L_1|^2|L_2|^2}\\
&\;+\frac{ (L_2\cdot M_2) (L_2\cdot Y_2'(s')) (L_2\cdot N_2)}{|L_2|^6}\delta L.
\end{split}
\label{eqn: split a term in the difference of H 1 estimate of g-g}
\end{equation}
If $\theta \in [\frac{1}{4},\frac{1}{2}]$, we bound it as follows.
\begin{equation}
\begin{split}
&\;\left|\frac{ (L_1\cdot M_1) (L_1\cdot Y_1'(s')) (L_1\cdot N_1)}{|L_1|^6}L_1-\frac{ (L_2\cdot M_2) (L_2\cdot Y_2'(s')) (L_2\cdot N_2)}{|L_2|^6}L_2\right|\\
\leq &\;C\lambda^{-3}[ V^2|\delta M||N_1|+V^2|M_2||\delta N|+V(|\delta L|+|\delta Y'(s')|)|M_2||N_2|].
\label{eqn: crude bounding one term in H1 estimate of g-g small theta case}
\end{split}
\end{equation}
Then we proceed as in \eqref{eqn: bounding the most singular term in g_X H^2 estimate}.
By Lemma \ref{lemma: estimates for L M N} and \eqref{eqn: L^p_s estimate for N}, for $\beta$ satisfying \eqref{eqn: admissible range of beta larger},
\begin{equation}
\begin{split}
&\;\left\|\int_{\mathbb{T}}\left|\frac{ (L_1\cdot M_1) (L_1\cdot Y_1'(s')) (L_1\cdot N_1)}{|L_1|^6}L_1-\frac{ (L_2\cdot M_2) (L_2\cdot Y_2'(s')) (L_2\cdot N_2)}{|L_2|^6}L_2\right|\,ds'\right\|_{L^2(\mathbb{T})}\\
\leq &\;C\lambda^{-3}V^2\left\|\int_{\mathbb{T}}\frac{|\delta Y'(s+\tau)-\delta Y'(s)|}{|\tau|}|N_1|+\frac{|Y'_2(s+\tau)- Y'_2(s)|}{|\tau|}||\delta N|\,d\tau\right\|_{L^2(\mathbb{T})}\\
&\;+C\lambda^{-3} V(\|\delta L\|_{L^\infty_sL^2_{s'}}+\|\delta Y'\|_{L^2})\|M_2\|_{L^2_sL^\infty_{s'}}\|N_2\|_{L^\infty_sL^2_{s'}}\\
\leq &\;C\lambda^{-3}V^2\left\|\frac{\|\delta Y'(s+\tau)-\delta Y'(s)\|_{L^2_s}}{|\tau|^{1-\beta}}\right\|_{L^2_{\tau}}\left\|\frac{\|N_1\|_{L^\infty_s} }{|\tau|^{\beta}}\right\|_{L^2_\tau}\\
&\;+C\lambda^{-3}V^2\left\|\frac{\| Y'_2(s+\tau)-Y'_2(s)\|_{L^\infty_s}}{|\tau|^{1+\beta}}\right\|_{L^2_{\tau}}\left\|\|\delta N\|_{L^2_s} |\tau|^{\beta}\right\|_{L^2_\tau}\\
&\;+C\lambda^{-3}V\|\delta Y\|_{\dot{H}^1(\mathbb{T})}\|\mathcal{M}Y_2''\|_{L^2}\|Y_2''\|_{L^2}\\
\leq &\;C\lambda^{-3}V^2\|\delta Y'\|_{\dot{B}_{2,2}^{\frac{1}{2}-\beta}}\left\|\frac{1}{|\tau|^{1+\beta}}\sup_{|\eta|\leq |\tau|}\|Y_1'(s+\eta)-Y_1'(s)\|_{L^\infty_s(\mathbb{T})}\right\|_{L^2_\tau}\\
&\;+C\lambda^{-3}V^2\|Y'_2\|_{\dot{B}_{\infty,2}^{\frac{1}{2}+\beta}}\left\|\frac{1}{|\tau|^{1-\beta}}\sup_{|\eta|\leq |\tau|}\|\delta Y'(s+\eta)-\delta Y'(s)\|_{L^2_s(\mathbb{T})}\right\|_{L^2_\tau}\\
&\;+C\lambda^{-3}\|Y_2''\|_{L^2}^3\|\delta Y\|_{\dot{H}^1}\\
\leq &\;C\lambda^{-3}V^2(\|\delta Y'\|_{\dot{H}^{\frac{1}{2}-\beta}}\|Y_1'\|_{B_{\infty,2}^{\frac{1}{2}+\beta}}+\|Y'_2\|_{\dot{B}_{\infty,2}^{\frac{1}{2}+\beta}}\|\delta Y'\|_{\dot{B}_{2,2}^{\frac{1}{2}-\beta}})+C\lambda^{-3}\|Y_2''\|_{L^2}^3\|\delta Y\|_{\dot{H}^1}\\
\leq &\;C\lambda^{-3}(\|Y_1''\|_{\dot{H}^\beta(\mathbb{T})}+\|Y_2''\|_{\dot{H}^\beta(\mathbb{T})})^3\|\delta Y\|_{\dot{H}^{\frac{3}{2}-\beta}(\mathbb{T})}. 
\end{split}
\label{eqn: bounding one term in H1 estimate of g-g small theta case}
\end{equation}
Here we used equivalent norms and embedding theorems of Besov spaces again \cite[\S\,2.5.12 and \S\,2.7.1]{triebel2010theory}.
The other terms in \eqref{eqn: goal in bounding H1 norm of g-g} can be handled in a similar manner.
This proves \eqref{eqn: improved H1 estimate for g_X1-g_X2} when $\theta \in [\frac{1}{4},\frac{1}{2}]$.

In the case of $\theta \in(\frac{1}{2},1)$, we need to take special care of the first two terms in \eqref{eqn: split a term in the difference of H 1 estimate of g-g}.
Denote them as
\begin{equation}
\begin{split}
&\;\frac{ (L_1\cdot \delta M) (L_1\cdot Y_1'(s')) (L_1\cdot N_1)}{|L_1|^6}L_1+\frac{ (L_1\cdot M_2) (L_1\cdot Y_1'(s')) (L_1\cdot \delta N)}{|L_1|^6}L_1\\
= &\;A_1(s,s')\delta M(s,s') + A_2(s,s')\delta N(s,s'),
\end{split}
\end{equation}
where
\begin{align}
A_1(s,s') = &\;\frac{ (L_1(s,s')\otimes L_1(s,s')) (L_1(s,s')\cdot Y_1'(s')) (L_1(s,s')\cdot N_1(s,s'))}{|L_1(s,s')|^6},\\
A_2(s,s') = &\;\frac{ (L_1(s,s')\otimes L_1(s,s')) (L_1(s,s')\cdot Y_1'(s')) (L_1(s,s')\cdot M_2(s,s'))}{|L_1(s,s')|^6}.
\end{align}
In particular, by definition,
\begin{align}
A_1(s,s) = &\;\frac{ (Y_1'(s)\otimes Y_1'(s)) (Y_1'(s)\cdot Y_1'(s)) (Y_1'(s)\cdot \frac{1}{2}Y''_1(s))}{|Y_1'(s)|^6},\\
A_2(s,s) = &\;\frac{ (Y_1'(s)\otimes Y_1'(s)) (Y_1'(s)\cdot Y_1'(s)) (Y_1'(s)\cdot Y''_2(s))}{|Y_1'(s)|^6}.
\end{align}
It is then not difficult to derive that
\begin{equation}
\begin{split}
&\;|A_1(s,s')-A_1(s,s)|\\
\leq  &\;C\lambda^{-2}|Y_1'(s')-Y'_1(s)||N_1|+C\lambda^{-2}|L_1-Y'_1(s)||N_1|+C\lambda^{-1}\left|N_1-\frac{1}{2}Y_1''(s)\right|\\
\leq  &\;C|\tau|(\lambda^{-2}|M_1||N_1|+\lambda^{-2}|N_1|^2+\lambda^{-1}|\partial_s N_1|).
\end{split}
\end{equation}
By \eqref{eqn: representation of N}, \eqref{eqn: representation of N derivative}, and the fact that $H^{2+\theta}(\mathbb{T})\hookrightarrow C^{2,\theta-\frac{1}{2}}(\mathbb{T})$,
\begin{equation}
|A_1(s,s')-A_1(s,s)|\leq C|\tau|(\lambda^{-2}\|Y_1''\|_{\dot{H}^{\theta}}^2+\lambda^{-1}|\tau|^{\theta -\frac{3}{2}}\|Y_1''\|_{\dot{H}^{\theta}})
\leq C|\tau|^{\theta -\frac{1}{2}}\lambda^{-2}\|Y_1''\|_{\dot{H}^{\theta}}^2.
\end{equation}
Similarly,
\begin{equation}
|A_2(s,s')-A_2(s,s)|\leq C|\tau|^{\theta -\frac{1}{2}}\lambda^{-2}\|Y_1''\|_{\dot{H}^{\theta}}\|Y_2''\|_{\dot{H}^{\theta}}.
\end{equation}
Hence,
\begin{equation}
\begin{split}
&\;\left\|\int_{\mathbb{T}}A_1(s,s')\delta M(s,s') + A_2(s,s')\delta N(s,s')\,ds'\right\|_{L_s^2}\\
\leq &\; \left\|\int_{\mathbb{T}}(A_1(s,s')-A_1(s,s))\cdot \tau^{-1}(\delta Y'(s')-\delta Y'(s))\,ds'\right\|_{L_s^2}\\
&\;+\left\|\int_{\mathbb{T}}(A_2(s,s')-A_2(s,s))\cdot \tau^{-1}(\delta L(s,s')-\delta Y'(s))\,ds'\right\|_{L_s^2}\\
&\;+\|A_1(s,s)\|_{L^\infty_s}\left\|\mathrm{p.v.}\int_{\mathbb{T}}\frac{\delta Y'(s+\tau)}{\tau}\,d\tau\right\|_{L_s^2}\\
&\;+\|A_2(s,s)\|_{L^\infty_s}\left\|\mathrm{p.v.}\int_{\mathbb{T}}\frac{\delta L(s,s+\tau)}{\tau}\,d\tau\right\|_{L_s^2}\\
\leq &\;C\lambda^{-2}\|Y_1''\|_{\dot{H}^{\theta}}^2\int_{\mathbb{T}} |\tau|^{\theta-\frac{3}{2}}(\|\delta Y'(s+\tau)\|_{L^2_s}+\|\delta Y'\|_{L^2})\,d\tau\\
&\;+C\lambda^{-2}\|Y_1''\|_{\dot{H}^{\theta}}\|Y_2''\|_{\dot{H}^{\theta}}\int_{\mathbb{T}}|\tau|^{\theta-\frac{3}{2}}(\|\delta L(s,s+\tau)\|_{L^2_s}+\|\delta Y'\|_{L^2})\,d\tau\\
&\;+C\lambda^{-1}\|Y_1''\|_{\dot{H}^\theta}\\
&\;\qquad \cdot \left\|\mathrm{p.v.}\int_{\mathbb{T}}\frac{\delta Y'(s+\tau)}{2\tan\frac{\tau}{2}}+\delta Y'(s+\tau)\left(\frac{1}{\tau}-\frac{1}{2\tan\frac{\tau}{2}}\right)\,d\tau\right\|_{L_s^2}\\
&\;+C\lambda^{-1}\|Y_1''\|_{\dot{H}^\theta}\\
&\;\qquad \cdot \left\|\mathrm{p.v.}\int_{\mathbb{T}}\frac{\delta Y(s+\tau)-\delta Y(s)}{4\sin^2\frac{\tau}{2}}+\tau \delta L\cdot \left(\frac{1}{\tau^2}-\frac{1}{4\sin^2\frac{\tau}{2}}\right)\,d\tau\right\|_{L_s^2}\\
\leq &\;C\lambda^{-2}(\|Y_1''\|_{\dot{H}^{\theta}}+\|Y_2''\|_{\dot{H}^{\theta}})^2\|\delta Y'\|_{L^2}\\
&\;+C\lambda^{-1}\|Y_1''\|_{\dot{H}^\theta} \left\||\mathcal{H}\delta Y'|+\int_{\mathbb{T}}|\delta Y'(s+\tau)|\,d\tau\right\|_{L_s^2}\\
&\;+C\lambda^{-1}\|Y_1''\|_{\dot{H}^\theta} \left\||\mathcal{H}\delta Y'|+\int_{\mathbb{T}}|\delta L(s,s+\tau)|\,d\tau\right\|_{L_s^2}\\
\leq &\;C\lambda^{-2}(\|Y_1''\|_{\dot{H}^{\theta}}+\|Y_2''\|_{\dot{H}^{\theta}})^2\|\delta Y'\|_{L^2}.
\end{split}
\end{equation}
This bounds the $L^2(\mathbb{T})$-norm of the first two terms in \eqref{eqn: split a term in the difference of H 1 estimate of g-g}.
For the other terms, we argue as in \eqref{eqn: crude bounding one term in H1 estimate of g-g small theta case} and \eqref{eqn: bounding one term in H1 estimate of g-g small theta case} to bound them by $C\lambda^{-3}(\|Y_1''\|_{\dot{H}^{\theta}}+\|Y_2''\|_{\dot{H}^{\theta}})^3\|\delta Y'\|_{L^2}$.
The other terms in \eqref{eqn: goal in bounding H1 norm of g-g} can be handled in a similar manner.
This completes the proof of \eqref{eqn: improved H1 estimate for g_X1-g_X2} when $\theta \in (\frac{1}{2},1)$.
\end{proof}

\begin{remark}
A weaker estimate than \eqref{eqn: improved H1 estimate for g_X1-g_X2} can be proved more easily, which reads
\begin{equation}
\|g_{Y_1}-g_{Y_2}\|_{\dot{H}^1(\mathbb{T})}\leq
C\lambda^{-3}(\|Y_1\|_{\dot{H}^{2+\theta}}+\|Y_2\|_{\dot{H}^{2+\theta}})^3\|\delta Y\|_{\dot{H}^{3/2}(\mathbb{T})}.
\end{equation}
In fact, this is sufficient for proving Theorem \ref{thm: error estimates of the regularized problem} and Theorem \ref{thm: convergence and error estimates of the low-frequency regularized IB method}.
However, we pursue a more refined estimate in order to obtain better error estimates in Theorem \ref{thm: convergence and error estimates of the low-frequency regularized IB method}.
\end{remark}

\section*{Acknowledgement}
The author would like to thank Prof.\;Fang-Hua Lin, Prof.\;Charles Peskin, Dr.\;Yuanxun Bao, Dr.\;Xiaochuan Tian, Prof.\;Yuning Liu, and Prof.\;Inwon Kim for many inspiring discussions.
A major part of this work was done when the author was a Ph.D.\;student at Courant Institute at New York University, where he was partially supported by the Graduate School of Arts and Science under Dean's Dissertation Fellowship and by National Science Foundation under Award DMS-1501000.


\begin{thebibliography}{10}

\bibitem{peskin2002immersed}
Charles~S Peskin.
\newblock The immersed boundary method.
\newblock {\em Acta numerica}, 11:479--517, 2002.

\bibitem{lin2017solvability}
Fang-Hua Lin and Jiajun Tong.
\newblock Solvability of the {Stokes} immersed boundary problem in two
  dimensions.
\newblock {\em Communications on Pure and Applied Mathematics}, 72(1):159--226,
  2019.

\bibitem{peskin1972flowPhD}
Charles~S Peskin.
\newblock {\em Flow patterns around heart valves: a digital computer method for
  solving the equations of motion}.
\newblock PhD thesis, Sue Golding Graduate Division of Medical Sciences, Albert
  Einstein College of Medicine, Yeshiva University, 1972.

\bibitem{peskin1972flow}
Charles~S Peskin.
\newblock Flow patterns around heart valves: a numerical method.
\newblock {\em Journal of Computational Physics}, 10(2):252--271, 1972.

\bibitem{dillon1995microscale}
Robert Dillon, Lisa Fauci, and Donald Gaver~III.
\newblock A microscale model of bacterial swimming, chemotaxis and substrate
  transport.
\newblock {\em Journal of theoretical biology}, 177(4):325--340, 1995.

\bibitem{bottino1998computational}
Dean~C Bottino and Lisa~J Fauci.
\newblock A computational model of ameboid deformation and locomotion.
\newblock {\em European Biophysics Journal}, 27(5):532--539, 1998.

\bibitem{mcqueen2000three}
David~M McQueen and Charles~S Peskin.
\newblock A three-dimensional computer model of the human heart for studying
  cardiac fluid dynamics.
\newblock {\em ACM SIGGRAPH Computer Graphics}, 34(1):56--60, 2000.

\bibitem{lim2004simulations}
Sookkyung Lim and Charles~S Peskin.
\newblock Simulations of the whirling instability by the immersed boundary
  method.
\newblock {\em SIAM Journal on Scientific Computing}, 25(6):2066--2083, 2004.

\bibitem{zhu2002simulation}
Luoding Zhu and Charles~S Peskin.
\newblock Simulation of a flapping flexible filament in a flowing soap film by
  the immersed boundary method.
\newblock {\em Journal of Computational Physics}, 179(2):452--468, 2002.

\bibitem{miller2005computational}
Laura~A Miller and Charles~S Peskin.
\newblock A computational fluid dynamics of `clap and fling' in the smallest
  insects.
\newblock {\em Journal of Experimental Biology}, 208(2):195--212, 2005.

\bibitem{tong2018thesis}
Jiajun Tong.
\newblock {\em On the Stokes Immersed Boundary Problem in Two Dimensions}.
\newblock PhD thesis, New York University, 2018.

\bibitem{pozrikidis1992boundary}
Constantine Pozrikidis.
\newblock {\em Boundary integral and singularity methods for linearized viscous
  flow}.
\newblock Cambridge University Press, 1992.

\bibitem{mori2017well}
Yoichiro Mori, Analise Rodenberg, and Dan Spirn.
\newblock Well-posedness and global behavior of the {Peskin} problem of an
  immersed elastic filament in {Stokes} flow.
\newblock {\em arXiv preprint arXiv:1704.08392}, 2017.

\bibitem{rodenberg20182d}
Analise Rodenberg.
\newblock {\em 2D Peskin Problems of an Immersed Elastic Filament in Stokes
  Flow}.
\newblock PhD thesis, University of Minnesota, 2018.

\bibitem{mittal2005immersed}
Rajat Mittal and Gianluca Iaccarino.
\newblock Immersed boundary methods.
\newblock {\em Annu. Rev. Fluid Mech.}, 37:239--261, 2005.

\bibitem{leonard1980vortex}
Anthony Leonard.
\newblock Vortex methods for flow simulation.
\newblock {\em Journal of Computational Physics}, 37(3):289--335, 1980.

\bibitem{cortez2001method}
Ricardo Cortez.
\newblock The method of regularized {Stokeslets}.
\newblock {\em SIAM Journal on Scientific Computing}, 23(4):1204--1225, 2001.

\bibitem{cortez2005method}
Ricardo Cortez, Lisa Fauci, and Alexei Medovikov.
\newblock The method of regularized {Stokeslets} in three dimensions: analysis,
  validation, and application to helical swimming.
\newblock {\em Physics of Fluids}, 17(3):031504, 2005.

\bibitem{beyer1992analysis}
Richard~P Beyer and Randall~J LeVeque.
\newblock Analysis of a one-dimensional model for the immersed boundary method.
\newblock {\em SIAM Journal on Numerical Analysis}, 29(2):332--364, 1992.

\bibitem{griffith2005order}
Boyce~E Griffith and Charles~S Peskin.
\newblock On the order of accuracy of the immersed boundary method: Higher
  order convergence rates for sufficiently smooth problems.
\newblock {\em Journal of Computational Physics}, 208(1):75--105, 2005.

\bibitem{griffith2007adaptive}
Boyce~E Griffith, Richard~D Hornung, David~M McQueen, and Charles~S Peskin.
\newblock An adaptive, formally second order accurate version of the immersed
  boundary method.
\newblock {\em Journal of computational physics}, 223(1):10--49, 2007.

\bibitem{stein2016immersed}
David~B Stein, Robert~D Guy, and Becca Thomases.
\newblock Immersed boundary smooth extension: a high-order method for solving
  pde on arbitrary smooth domains using fourier spectral methods.
\newblock {\em Journal of Computational Physics}, 304:252--274, 2016.

\bibitem{stein2017immersed}
David~B Stein, Robert~D Guy, and Becca Thomases.
\newblock Immersed boundary smooth extension {(IBSE)}: A high-order method for
  solving incompressible flows in arbitrary smooth domains.
\newblock {\em Journal of Computational Physics}, 335:155--178, 2017.

\bibitem{liu2012properties}
Yang Liu and Yoichiro Mori.
\newblock Properties of discrete delta functions and local convergence of the
  immersed boundary method.
\newblock {\em SIAM Journal on Numerical Analysis}, 50(6):2986--3015, 2012.

\bibitem{liu2014p}
Yang Liu and Yoichiro Mori.
\newblock ${L}^p$ convergence of the immersed boundary method for stationary
  {Stokes} problems.
\newblock {\em SIAM Journal on Numerical Analysis}, 52(1):496--514, 2014.

\bibitem{mori2008convergence}
Yoichiro Mori.
\newblock Convergence proof of the velocity field for a {Stokes} flow immersed
  boundary method.
\newblock {\em Communications on Pure and Applied Mathematics},
  61(9):1213--1263, 2008.

\bibitem{tornberg2004numerical}
Anna-Karin Tornberg and Bj{\"o}rn Engquist.
\newblock Numerical approximations of singular source terms in differential
  equations.
\newblock {\em Journal of Computational Physics}, 200(2):462--488, 2004.

\bibitem{majda2002vorticity}
Andrew~J Majda and Andrea~L Bertozzi.
\newblock {\em Vorticity and incompressible flow}, volume~27.
\newblock Cambridge University Press, 2002.

\bibitem{warner1972kinetic}
Harold~R Warner~Jr.
\newblock Kinetic theory and rheology of dilute suspensions of finitely
  extendible dumbbells.
\newblock {\em Industrial \& Engineering Chemistry Fundamentals},
  11(3):379--387, 1972.

\bibitem{cortez2004parametric}
Ricardo Cortez, Charles~S Peskin, John~M Stockie, and Douglas Varela.
\newblock Parametric resonance in immersed elastic boundaries.
\newblock {\em SIAM Journal on Applied Mathematics}, 65(2):494--520, 2004.

\bibitem{solonnikov1986solvability}
VA~Solonnikov.
\newblock Solvability of the problem of evolution of an isolated volume of
  viscous, incompressible capillary fluid.
\newblock {\em Journal of Soviet Mathematics}, 32(2):223--228, 1986.

\bibitem{tanaka1993global}
Naoto Tanaka.
\newblock Global existence of two phase nonhomogeneous viscous incompbessible
  fluid flow.
\newblock {\em Communications in partial differential equations},
  18(1-2):41--81, 1993.

\bibitem{solonnikov2014theory}
V~Solonnikov.
\newblock ${L}_p$-theory of the problem of motion of two incompressible
  capillary fluids in a container.
\newblock {\em Journal of Mathematical Sciences}, 198(6), 2014.

\bibitem{temam1984navier}
Roger Temam.
\newblock {\em {Navier-Stokes} equations}, volume~2.
\newblock North-Holland Amsterdam, 1984.

\bibitem{peskin1993improved}
Charles~S Peskin and Beth~Feller Printz.
\newblock Improved volume conservation in the computation of flows with
  immersed elastic boundaries.
\newblock {\em Journal of computational physics}, 105(1):33--46, 1993.

\bibitem{stockie1997analysis}
John~Michael Stockie.
\newblock {\em Analysis and computation of immersed boundaries, with
  application to pulp fibres}.
\newblock PhD thesis, University of British Columbia, 1997.

\bibitem{brady1988dynamic}
John~F Brady, Ronald~J Phillips, Julia~C Lester, and Georges Bossis.
\newblock Dynamic simulation of hydrodynamically interacting suspensions.
\newblock {\em Journal of Fluid Mechanics}, 195:257--280, 1988.

\bibitem{roma1999adaptive}
Alexandre~M Roma, Charles~S Peskin, and Marsha~J Berger.
\newblock An adaptive version of the immersed boundary method.
\newblock {\em Journal of Computational Physics}, 153(2):509--534, 1999.

\bibitem{bringley2008analysis}
Thomas~T Bringley.
\newblock {\em Analysis of the immersed boundary method for Stokes flow}.
\newblock PhD thesis, New York University, 2008.

\bibitem{yang2009smoothing}
Xiaolei Yang, Xing Zhang, Zhilin Li, and Guo-Wei He.
\newblock A smoothing technique for discrete delta functions with application
  to immersed boundary method in moving boundary simulations.
\newblock {\em Journal of Computational Physics}, 228(20):7821--7836, 2009.

\bibitem{bao2016gaussian}
Yuanxun Bao, Jason Kaye, and Charles~S Peskin.
\newblock A {Gaussian-like} immersed-boundary kernel with three continuous
  derivatives and improved translational invariance.
\newblock {\em Journal of Computational Physics}, 316:139--144, 2016.

\bibitem{stotsky2018posteriori}
Jay~A Stotsky and David~M Bortz.
\newblock A posteriori error analysis of fluid-stucture interactions: Time
  dependent error.
\newblock {\em arXiv preprint arXiv:1807.03279}, 2018.

\bibitem{triebel2010theory}
H.~Triebel.
\newblock {\em Theory of Function Spaces}.
\newblock Modern Birkh{\"a}user Classics. Springer Basel, 2010.

\end{thebibliography}

\bigskip
\noindent Jiajun Tong\\
Department of Mathematics, UCLA\\
Box 951555\\
Los Angeles, CA 90095-1555\\
USA\\
E-mail: jiajun@math.ucla.edu

\end{document}